\documentclass[11pt,reqno,twoside]{amsart}

\usepackage[utf8]{inputenc}   
\usepackage[T1]{fontenc}

\usepackage[asymmetric,top=2.5cm,bottom=2.8cm,left=2.38cm,right=2.38cm]{geometry}
\usepackage{amsmath,amssymb,amsfonts,amsthm}
\usepackage{mathrsfs}   
\usepackage{bm}         
\newtheorem{thm}{Theorem}[section]

\newtheorem{lem}{Lemma}[section]

\newtheorem{assumption}{Assumption}[section] 
\theoremstyle{definition}

\theoremstyle{remark}
\newtheorem{rem}{Remark}[section]
\numberwithin{equation}{section}

\usepackage{graphicx}
\usepackage{caption}
\usepackage{subcaption} 
\usepackage{float}      

\usepackage{tikz}
\usetikzlibrary{calc,arrows.meta}
\usetikzlibrary{positioning}

\usepackage{booktabs}
\usepackage{array}
\usepackage{paralist}
\usepackage{verbatim}
\usepackage{tabularx}
\usepackage{mathrsfs} 
\usepackage{extarrows} 
\usepackage{xcolor}
\usepackage{bm}        
\usepackage{cite}      

\usepackage{hyperref}
\hypersetup{
  colorlinks=true,
  linkcolor={blue!80!black},
  urlcolor={blue!80!black},
  citecolor={blue!80!black}
}

\newcommand{\rmd}{\mathrm{d}}

\newcommand{\bu}{\mathbf{u}}
\newcommand{\bv}{\mathbf{v}}

\newcommand{\bff}{\mathbf{f}}
\newcommand{\bh}{\mathbf{h}}
\newcommand{\bg}{\mathbf{g}}
\newcommand{\bq}{\mathbf{q}}

\newcommand{\bS}{\mathbf{S}}
\newcommand{\bQ}{\mathbf{Q}}

\newcommand{\bx}{\mathbf{x}}
\newcommand{\by}{\mathbf{y}}
\newcommand{\bz}{\mathbf{z}}

\newcommand{\bF}{\mathbf{F}}
\newcommand{\bphi}{\bm{\phi}}
\newcommand{\dist}{\mathrm{dist}}      
\newcommand{\diam}{\mathrm{diam}} 

\newcommand{\calK}{\mathcal{K}}
\newcommand{\calP}{\mathcal{P}}
\newcommand{\calN}{\mathcal{N}}

\newcommand{\calW}{\mathcal{W}}
\newcommand{\bbH}{\mathbb{H}}
\newcommand{\bbL}{\mathbb{L}}
\newcommand{\bbW}{\mathbb{W}}
\newcommand\rmi{{\mathrm{i}}}

\allowdisplaybreaks[4]

\title[ Elastic Cald\'eron Problem via  Resonant Hard Inclusions ]{ 
Elastic Calderón Problem via Resonant Hard Inclusions: Linearisation of the N-D Map and Density Reconstruction}

\author{Huaian Diao}
 \address{School of Mathematics, Jilin University and Key Laboratory of Symbolic Computation and Knowledge Engineering of Ministry of Education, Changchun, Jilin, China.}
 \email{diao@jlu.edu.cn, hadiao@gmail.com}

\author{Mourad Sini}
 \address{Radon Institute, Austrian Acedmy of Sciences, Linz, Austria.}
 \email{mourad.sini@oeaw.ac.at}

 \author{Ruixiang Tang}
 \address{School of Mathematics, Jilin University, Changchun 130012, China.}
 \email{tangrx97@gmail.com; tangrx23@mails.jlu.edu.cn}

\date{} 

\begin{document}
\maketitle

\begin{abstract}

We study an elastic version of the Calderón problem: determine the internal mass density $\rho(\mathbf{x})$ from the Neumann-to-Dirichlet (N-D) map associated with the isotropic Lam\'e system
\[
\mathcal{L}_{\lambda,\mu} \mathbf{u} + \omega^2 \rho(\mathbf{x}) \mathbf{u} = \mathbf{0}
\]
in a bounded elastic body $\Omega \subset \mathbb{R}^3$. To the best of our knowledge, this work provides the first constructive strategy, based on embedding resonant hard inclusions, for the Calderón-type inverse problem in the isotropic Lam\'e system to reconstruct the density $\rho$. The key to our strategy is to induce a uniform negative shift in the effective density (i.e., a negative effective density) by embedding a subwavelength periodic array of resonant high-density inclusions.

We insert a periodic cluster of high-density inclusions of size $a$ and density $\rho_1 \simeq a^{-2}$ into $\Omega$, away from $\partial \Omega$. For excitation frequencies $\omega$ tuned to a suitable eigenvalue of the elastic Newton operator (i.e., Kelvin operator) associated with a single inclusion, we show that the N-D map $\Lambda_D$ of the composite medium converges, as $a \to 0$ and the number $M$ of inclusions tends to infinity, to an effective map $\Lambda_{\mathcal{P}}$ corresponding to an elastic medium with a uniform negative density shift $-\mathcal{P}^2$. We prove an operator norm estimate
\[
\|\Lambda_D - \Lambda_{\mathcal{P}}\| \leq C a^{\alpha} \mathcal{P}^6,
\]
with $\alpha > 0$ depending on the geometric scaling. We then derive a first-order linearization formula for $\Lambda_{\mathcal{P}}$ around this negative background, expressed in terms of $\rho$ and the Newton volume potential for the shifted Lam\'e operator. By testing this linearized relation with suitable complex geometric optics solutions for the Lam\'e system, we obtain a reconstruction formula for the Fourier transform of $\rho$, and hence a global density recovery scheme.

The method proposed in this paper demonstrates how metamaterial-inspired effective media can be exploited as an analytic tool for inverse coefficient problems in linear elasticity, enabling a tractable linearization around a negative background and an explicit global reconstruction procedure. This provides a novel strategy and paradigm for using nanoscale metamaterials to solve inverse problems.

	\medskip

		
 
\end{abstract}


\section{Introduction and Statement of the Main Results}

\subsection{Introduction}


The Calder\'on problem asks whether one can determine the electrical conductivity $\sigma(\bx)$ inside a bounded domain $\Omega\subset\mathbb{R}^3$ from boundary measurements of voltages and currents. More precisely, given the conductivity equation for the electric potential $u$ with prescribed boundary voltage $u|_{\partial\Omega}=f$,
\begin{align}
\nabla\cdot(\sigma\nabla u)=0\quad\text{in }\Omega,
\end{align}
one considers the Dirichlet-to-Neumann (D--N) map, which maps boundary voltages to the induced boundary currents, and asks whether this map uniquely determines $\sigma$ (in the isotropic scalar setting). This question, posed by Calder\'on in his seminal 1980 lecture notes \cite{C80}, is directly motivated by electrical impedance tomography (EIT), whose goal is to reconstruct the interior conductivity of a body from boundary voltage-to-current measurements.

In 1984, Kohn and Vogelius~\cite{KV84} showed that the D--N map determines the boundary values of the conductivity, and indeed its full boundary jet under appropriate smoothness assumptions. For $\sigma\in C^{2}$, global uniqueness in dimensions $n\ge 3$ was established by Sylvester and Uhlmann via the construction of complex geometrical optics (CGO) solutions in \cite{SU87}. Subsequent contributions progressively weakened the regularity assumptions on $\sigma$: first to $C^{\frac{3}{2}+\varepsilon}$, $\varepsilon>0$, in \cite{B96}, then to $\sigma \in \bbW^{\frac{3}{2},\infty}(\Omega)$ in \cite{P03} and $\sigma\in \mathbb{W}^{\frac{3}{2},p}(\Omega)$ with $p>2n$ (in particular $p>6$ when $n=3$) in \cite{B03}. Haberman and Tataru~\cite{HT13} proved uniqueness for Lipschitz conductivities under a suitable closeness assumption (and for arbitrary $C^{1}$ conductivities), and Caro and Rogers~\cite{C16} later established global uniqueness for Lipschitz conductivities, i.e.\ $\sigma\in \bbW^{1,\infty}(\Omega)$. Further low-regularity uniqueness results include conductivities in $\bbW^{1,n}(\Omega)$ (in particular $\bbW^{1,3}$ in three dimensions) in \cite{H15}. In two dimensions, Nachman developed a reconstruction scheme for $C^2$ conductivities \cite{N88,N96}, and later Astala and P\"aiv\"arinta solved the problem for essentially bounded conductivities in the plane using quasiconformal mappings and nonlinear Fourier analysis \cite{AP06}. For further results on uniqueness and stability in the Calder\'on problem, see, e.g., \cite{A88,A90,B97,S08} and the references therein. More recent developments include reconstruction at low regularity, notably the work \cite{CGR25}.

An important direction in this area is to extend inverse boundary value methods to systems of equations modeling more complex physical phenomena. The problem of recovering the constitutive parameters---typically the Lam\'e moduli $(\lambda,\mu)$, and in the time-harmonic setting also the mass density $\rho$---from a D--N or N--D map associated with the time-harmonic elasticity system is commonly referred to as the elastic inverse boundary value problem. Early progress includes work on the linearized inverse boundary value problem in isotropic elasticity, where explicit inversion formulas for perturbations of the Lam\'e parameters were derived in \cite{I90}. Local uniqueness in determining the Lam\'e parameters from boundary measurements of the D--N map for the two-dimensional elasticity system was established by Nakamura and Uhlmann~\cite{NU93}. In the context of isotropic linear elasticity, global uniqueness in recovering the Lam\'e parameters of an inhomogeneous medium from the associated D--N map was established by Nakamura and Uhlmann~\cite{NU94,NU03}. For elastic media, Nakamura and Uhlmann~\cite{NU95} proved that the full Taylor series at the boundary of the elasticity tensor can be determined from the D--N map. Uniqueness, stability, and boundary-determination results for the Lam\'e moduli and their derivatives, including results based on localized and corrupted boundary measurements, were obtained in \cite{ER02,BCY08,LN17}. Recently, inverse boundary value problems for elasticity and related partial differential equation models have been extensively studied in various geometric, fractional, and viscoelastic settings; see, for instance, \cite{YLL21,FLL23,UZ21,UZ24,TL23,Li22} and the references therein. These works provide uniqueness, stability, and/or reconstruction results for material parameters or nonlinear terms from boundary measurements or D--N data.

\medskip

Alongside these developments, there has been growing interest in
\emph{imaging with contrast agents}, where one deliberately modifies the medium
by injecting small-scale inclusions to enhance measurements. On the engineering
side, contrast agents ranging from gas microbubbles to a variety of nanoparticle
systems have been developed to improve contrast in ultrasound, photoacoustic,
magnetic resonance, and X-ray imaging \cite{BCU01,QFF09,VSJ20,SL23,AMA22,HC20}.
On the mathematical side, this has motivated a programme in which resonant
subwavelength inclusions (bubbles, droplets, plasmonic nanoparticles, highly
dense inclusions, etc.) are used as \emph{active probes} for inverse problems;
see, for example, \cite{ABG15,DGS-21,GS-23,CGS24,CGS25,IP_Mourad4,Soumen-M,GS25} and the
references therein.

In this work, we exploit the interacting property induced by negative effective density when inclusions resonate. This represents the first approach to invert the density $\rho$ of the elastic medium by injecting hard inclusions and using the N--D map. Compared to acoustic waves, elastic wave systems are vector-valued, involving coupling between compressional waves and shear waves, which makes their analysis more complex and challenging. Meanwhile, the well-posedness of the N--D map usually requires that the frequency is chosen away from the eigen-frequencies or resonant frequencies of the system. Here, we avoid the above-mentioned problems by exploiting the resonant frequencies of the injected hard inclusions, thereby creating an interacting effect similar to that produced by materials with effective negative density. Additionally, by combining this with a tractable linearisation around a negative background, we can effectively overcome the above difficulties. To this end, we assume that elastic boundary measurements can be performed both before and after injecting a large number of small high-density inclusions into a prescribed subregion. The inclusions are arranged periodically, densely populating this subregion, as detailed in Assumption~\ref{assumption1}, and their density contrast is scaled as $\rho_1 \sim a^{-2}$, where $a$ denotes their characteristic size. For frequencies $\omega$ chosen near a resonance determined by an eigenvalue of the elastic Newtonian (Kelvin) operator associated with a single inclusion, this array exhibits a collective resonance. Our analysis shows that, in a suitable multiple-interacting regime, the inclusions behave as an \emph{effective elastic medium} with negative effective density $-\mathcal{P}^2$, where $\mathcal{P}^2$ is defined in \eqref{eq:P2 def} and enters the effective density in \eqref{eq:system qg}. In particular, the injected inclusions act as an elastic contrast-agent ``cloud'' that realizes a tunable negative background for the density.

\medskip

 We now summarize the main contributions of this paper and introduce a novel process of injecting a cluster of hard inclusions to invert the density $\rho$.

\begin{itemize}
\item[(i)] \textbf{Resonant homogenization for the elastic N--D map.}
To the best of our knowledge, this is the first work to pursue the elastic inverse boundary value problem via a controlled negative effective density induced by a periodic array of small high-density inclusions. We prove that, as $a\to0$ and the number of inclusions $M\sim a^{h-1}$ with $h\in(0,1)$ tends to infinity under the geometric assumptions of Assumption~\ref{assumption1}, the N--D operator $\Lambda_D$ of the inclusion-perturbed medium converges to the N--D operator $\Lambda_\calP$ of a homogenized elastic medium with effective density $\rho(x)-\calP^2$:
\[
\|\Lambda_D-\Lambda_{\calP}\|_{\mathcal{L}\!\big(\bbH^{-1/2}(\partial\Omega)^3,\,\bbH^{1/2}(\partial\Omega)^3\big)}
\;\lesssim\; a^\alpha \calP^6
\]
for some $\alpha>0$ depending on $h$; see Theorem~\ref{thm:N--D} and
\eqref{eq:LambdaD-LambdaP}. This step relies on a detailed spectral and mapping
analysis of the elastic Newtonian operator associated with the Kelvin
fundamental solution.

\item[(ii)] \textbf{Linearisation around a negative elastic background.}
The second step is the analysis of the inverse boundary value problem for the
effective medium with density $\rho(x)-\calP^2$. In Theorem~\ref{thm:linearisation}
(see \eqref{eq:Lambda P f-Qf in thm}), we derive a first-order linearisation
formula for the difference $\Lambda_D-\Lambda_\calP$ between the N--D maps of
the composite and effective media. For a suitable family of boundary tractions,
we construct interior fields solving the unperturbed Lam\'e system with constant
density $-\calP^2$ and show that their interaction with the density perturbation
$\rho(x)$ on top of this background can be recovered from boundary measurements
up to an error term of order $O(\calP^{-4})$. This reduces the original nonlinear
problem of determining $\rho(\bx)$ from the elastic N--D map to a \emph{linear}
inverse problem for the density on top of a negative background.

\item[(iii)] \textbf{Density reconstruction by elastic CGO solutions.}
Finally, in Theorem~\ref{thm:CGO} we combine the linearized relation from (ii)
with standard CGO solutions for the Lam\'e system to derive an explicit
reconstruction formula for the Fourier transform of $\rho(\bx)$; see, in
particular, \eqref{eq: rho x in thm}.
\end{itemize}


\medskip


In what follows, we provide an overview of the main methodological approaches developed in this paper. For the first part (i), we consider the elastic system \eqref{eq:system} before injecting a cluster of nano-hard inclusions and the elastic system \eqref{eq:system vg} after the injection. We then present the difference between the N--D operators in these two cases, as shown in \eqref{eq:Lambda_D f g}. We present the equivalent elastic system after injecting this cluster of nano-hard inclusions, as shown in \eqref{eq:system qg}, where $\mathcal{P}$ is defined in \eqref{eq:P2 def}. From this, we derive the difference between the N--D operator of the equivalent elastic system and that of the system before injecting the inclusions, as shown in \eqref{eq:LambdaD-LambdaP}. Employing detailed and intricate analyses, such as the Lippmann--Schwinger formula, Taylor expansion, and singularity analysis of the fundamental solution, we prove the validity of \eqref{eq:Lambda D-Lambda P in thm} throughout Section \ref{sec:Theorem 1}, which constitutes the main part of this article. Owing to the precision, technicality, and complexity of the proof process, we provide a flowchart of the theorem's proof in Figure \ref{fig:J_derivation_flowchart}.

For the second part (ii), we prove Theorem \ref{thm:linearisation} in Section \ref{sec:Theorem 2}, establishing that \eqref{eq:Lambda P f-Qf in thm} holds. This constitutes the linearisation of the N--D operator after injecting a cluster of inclusions. Specifically, in \eqref{eq:Lambda P f-Qf in thm}, within a certain error range, we can approximately obtain $\mathcal{W}^{\mathbf{Q}^\bff}$, which satisfies \eqref{eq:WQf def}. For the third part (iii), we employ standard CGO solutions, as shown in \eqref{eq:Qf def in Section3} and \eqref{eq:Qg CGO def}, in Section \ref{sec:Theorem 3 CGO}; this allows us to explicitly calculate the density $\rho(\mathbf{x})$ of the elastic medium. From the perspective of imaging with contrast agents, our results demonstrate that a cloud of resonant high-density inclusions can be designed such that, at a suitable excitation frequency, the resulting effective medium exhibits a uniform negative density shift. In this regime, the boundary response of the elastic body encodes the unknown density in a nearly linear fashion that is amenable to CGO-based reconstruction.

The remainder of the paper is organized as follows. In Subsection \ref{subsec:1.2}, we present the mathematical setting and the main results. We state the corresponding theorems for parts (i), (ii), and (iii) mentioned above in Subsections \ref{subsec:1.2.1}, \ref{subsec:1.2.2}, and \ref{subsec:1.2.3}, respectively. In Section \ref{sec:Theorem 2}, we derive and rigorously justify the linearisation procedure underlying Theorem \ref{thm:linearisation}, whereas Section \ref{sec:Theorem 3 CGO} is devoted to the reconstruction of $\rho(\mathbf{x})$ from the resulting linearized N--D map stated in Theorem \ref{thm:CGO}. Section \ref{sec:Theorem 1} is devoted to establishing the validity of the effective N--D map described in Theorem \ref{thm:N--D}; this constitutes the main part of the paper and contains a detailed and technical analysis. Finally, in Section \ref{sec:appendix}, we provide detailed proofs of the lemmas and estimates used in the previous sections.

\bigskip

\subsection{Mathematical setting and the main results}\label{subsec:1.2}

 Let $\Omega\subset\mathbb{R}^3$ be a bounded elastic body with density $\rho(\mathbf{x})$, embedded in an unbounded background $\mathbb{R}^3\setminus\overline{\Omega}$ of density $\rho_0$. We write $(\lambda,\mu)$ for the Lam\'e parameters (interior/exterior, as appropriate). In the time-harmonic regime with angular frequency $\omega>0$, the displacement $\mathbf{u}^{\mathbf{f}}$ produced by the traction $\mathbf{f}$ is governed by
\begin{align}\label{eq:system}
\begin{cases}
\mathcal{L}_{\lambda,\mu}\,\mathbf{u}^{\mathbf{f}}(\mathbf{x})+\omega^2\rho(\mathbf{x})\,\mathbf{u}^{\mathbf{f}}(\mathbf{x})=\mathbf{0}, & \text{in }\Omega,\\[2mm]
\partial_{\boldsymbol{\nu}}\mathbf{u}^{\mathbf{f}} = \mathbf{f}, & \text{on }\partial\Omega,
\end{cases}
\end{align}
where the Lam\'e operator is defined by
\begin{align}\label{eq:mathcal L def}
    \mathcal{L}_{\lambda,\mu}\mathbf{u}
    := \mu\,\Delta\mathbf{u}+(\lambda+\mu)\,\nabla(\nabla\!\cdot\!\mathbf{u})
    = \nabla\cdot \sigma(\mathbf{u}),
\end{align}
and the stress tensor is given by
\begin{align}\notag
    \sigma(\mathbf{u}) = \lambda (\nabla\cdot \mathbf{u})\, \mathcal{I} + 2\mu\, \nabla^{s}\mathbf{u}, 
    \qquad 
    \nabla^{s}\mathbf{u}:=\tfrac12\left(\nabla\mathbf{u}+(\nabla\mathbf{u})^{\top}\right).
\end{align}

\noindent
The traction (co-normal) operator associated with $(\lambda,\mu)$ is
\begin{align}\label{eq:partial nu def}
    \partial_{\boldsymbol{\nu}}\mathbf{u} 
    = \lambda\,(\nabla\!\cdot\!\mathbf{u})\,\boldsymbol{\nu} 
      + 2\mu\,(\nabla^{s}\mathbf{u})\,\boldsymbol{\nu},
\end{align}
with $\boldsymbol{\nu}$ the outward unit normal on $\partial\Omega$ and $\nabla\mathbf{u}=(\partial_j u_i)_{i,j=1}^{3}$. The N--D map 
\[
\Lambda_e:\mathbb{H}^{-1/2}(\partial\Omega)\to\mathbb{H}^{1/2}(\partial\Omega)
\]
sends a traction $\mathbf{f}$ to the boundary trace of the corresponding displacement,
\begin{align}\label{eq:Lambda_e def}
    \Lambda_e(\mathbf{f}):=\mathbf{u}^{\mathbf{f}}\big|_{\partial\Omega}.
\end{align}

Let $\Gamma^\omega(\mathbf{x},\mathbf{y})$ denote the Kupradze fundamental matrix of the Navier system at frequency $\omega$ in the homogeneous background of density $\rho_0$; it solves
\begin{align}
\mathcal{L}_{\lambda,\mu}\Gamma^\omega(\mathbf{x},\mathbf{y})
+\omega^2\rho_0\,\Gamma^\omega(\mathbf{x},\mathbf{y})
= -\delta(\mathbf{x}-\mathbf{y})\,\mathcal{I} \quad \text{in }\mathbb{R}^3.
\end{align}
Let the compressional and shear wave speeds and wavenumbers be
\begin{align}
    c_p:=\sqrt{\frac{\lambda+2\mu}{\rho_0}},\quad 
c_s:=\sqrt{\frac{\mu}{\rho_0}},\qquad
k_p:=\frac{\omega}{c_p},\quad
k_s:=\frac{\omega}{c_s}.
\end{align}
With $\phi_k(\mathbf{x},\mathbf{y}):=\dfrac{e^{\rmi k|\mathbf{x}-\mathbf{y}|}}{4\pi|\mathbf{x}-\mathbf{y}|}$ the Helmholtz fundamental solution \cite{C19}, one has \cite{ABG15} 
\begin{align}\notag
\Gamma^\omega(\mathbf{x},\mathbf{y})
= \frac{1}{\mu}\,\phi_{k_s}(\mathbf{x},\mathbf{y})\,\mathcal{I}
+ \frac{1}{\mu\,k_s^2}\,\nabla_{\mathbf{x}}\nabla_{\mathbf{x}}^{\top}\left[\phi_{k_s}(\mathbf{x},\mathbf{y})-\phi_{k_p}(\mathbf{x},\mathbf{y})\right], 
\end{align}
and, for $k,l\in\{1,2,3\}$,
\begin{align}\notag
\Gamma^\omega_{kl}(\mathbf{x},\mathbf{y})
= \frac{1}{4\pi\mu}\frac{e^{\rmi k_s|\mathbf{x}-\mathbf{y}|}}{|\mathbf{x}-\mathbf{y}|}\,\delta_{kl}
+ \frac{1}{4\pi\omega^2\rho_0}\,
\partial_k\partial_l\frac{e^{\rmi k_s|\mathbf{x}-\mathbf{y}|}-e^{\rmi k_p|\mathbf{x}-\mathbf{y}|}}{|\mathbf{x}-\mathbf{y}|},  
\end{align}
which admits the series form (see \cite{ABG15})
\begin{align}
\Gamma_{kl}^\omega(\mathbf{x},\mathbf{y})
&= \frac{1}{4\pi\rho_0}\sum_{n=0}^{\infty}\frac{\rmi^n}{(n+2)\,n!}
\!\left(\frac{n+1}{c_s^{n+2}}+\frac{1}{c_p^{n+2}}\right)\!
\omega^n\,\delta_{kl}\,|\mathbf{x}-\mathbf{y}|^{\,n-1} \notag\\
&\quad - \frac{1}{4\pi\rho_0}\sum_{n=0}^{\infty}\frac{\rmi^n(n-1)}{(n+2)\,n!}
\!\left(\frac{1}{c_s^{n+2}}-\frac{1}{c_p^{n+2}}\right)\!
\omega^n\,|\mathbf{x}-\mathbf{y}|^{\,n-3}\,(\mathbf{x}-\mathbf{y})_k(\mathbf{x}-\mathbf{y})_l.  
\end{align}
At $\omega=0$, the Kelvin matrix $\Gamma^0$ is symmetric and its entries simplify to
\begin{align}\label{eq:Gamma0 def}
\Gamma_{kl}^0(\mathbf{x},\mathbf{y})
&= \dfrac{\gamma_1}{4\pi}\,\dfrac{\delta_{kl}}{|\mathbf{x}-\mathbf{y}|}
+ \dfrac{\gamma_2}{4\pi}\,\dfrac{(\mathbf{x}-\mathbf{y})_k(\mathbf{x}-\mathbf{y})_l}{|\mathbf{x}-\mathbf{y}|^{3}},\\[1mm]
\gamma_1&=\tfrac12\!\left(\tfrac{1}{\mu}+\tfrac{1}{2\mu+\lambda}\right),\qquad
\gamma_2=\tfrac12\!\left(\tfrac{1}{\mu}-\tfrac{1}{2\mu+\lambda}\right).\notag
\end{align}
For $|\mathbf{x}|\to\infty$, with $\hat{\mathbf{x}}:=\mathbf{x}/|\mathbf{x}|$, the far-field expansion reads
\begin{align}\notag
\Gamma^\omega(\mathbf{x},\mathbf{y})
= \frac{1}{4\pi(\lambda+2\mu)}\,\hat{\mathbf{x}}\hat{\mathbf{x}}^{\top}
\frac{e^{\rmi k_p|\mathbf{x}|}}{|\mathbf{x}|}\,e^{- \rmi k_p\hat{\mathbf{x}}\cdot\mathbf{y}}
+ \frac{1}{4\pi\mu}\,(\mathcal{I}-\hat{\mathbf{x}}\hat{\mathbf{x}}^{\top})
\frac{e^{\rmi k_s|\mathbf{x}|}}{|\mathbf{x}|}\,e^{-\rmi k_s\hat{\mathbf{x}}\cdot\mathbf{y}}
+ \mathcal{O}(|\mathbf{x}|^{-2}).  
\end{align}
Accordingly, the far-field matrix is decomposed as $\Gamma^\infty=\Gamma_p^\infty+\Gamma_s^\infty$, where
\begin{align}\notag
\Gamma_p^\infty(\hat{\mathbf{x}},\mathbf{y})
:= \frac{1}{4\pi(\lambda+2\mu)}\,\hat{\mathbf{x}}\hat{\mathbf{x}}^{\top}\,e^{-\rmi k_p\hat{\mathbf{x}}\cdot\mathbf{y}},
\qquad
\Gamma_s^\infty(\hat{\mathbf{x}},\mathbf{y})
:= \frac{1}{4\pi\mu}\,(\mathcal{I}-\hat{\mathbf{x}}\hat{\mathbf{x}}^{\top})\,e^{-\rmi k_s\hat{\mathbf{x}}\cdot\mathbf{y}}.  
\end{align}

We next consider high-density hard inclusions used as contrast agents. Each inclusion is of the form
\[
D_j = \mathbf{z}_j + a\,B,\qquad j=1,\ldots,M,
\]
where $B\subset\mathbb{R}^3$ is a smooth reference domain containing the origin with unit-size diameter, $\mathbf{z}_j$ is the center, and $a>0$ is the maximal radius; in particular $a\ll 1$. The inclusion density scales as
\begin{equation}\label{eq:rho1 def}
\rho_1=\tilde{\rho}_1\, a^{-2},
\end{equation}
with $\tilde{\rho}_1$ independent of $a$, and we write
\[
D:=\bigcup_{j=1}^{M}D_j.
\]
%
\noindent We recall that the static elastic fundamental solution $\Gamma^0(\mathbf{x},\mathbf{y})$ in $\mathbb{R}^3$ satisfies
\begin{align}\label{eq:Gamma0 satisfy}
\mathcal{L}_{\lambda,\mu}\,\Gamma^0(\mathbf{x},\mathbf{y}) = -\,\delta_{\mathbf{y}}(\mathbf{x})\,\mathcal{I}.
\end{align}

\noindent For $D_j$, we introduce the Newtonian volume potential
\begin{align}\label{eq:Newtonian def}
N_{D_j}[\mathbf{f}](\mathbf{x})
:= \int_{D_j} \Gamma^0(\mathbf{x},\mathbf{y})\,\mathbf{f}(\mathbf{y})\,\rmd \mathbf{y},
\qquad \mathbf{x}\in D_j .
\end{align}

\noindent For later use, we also introduce the boundary layer potentials associated with $\Gamma^0$. For a surface density $\varphi$ on $\partial\Omega$, the single-layer potential is
\begin{align}\label{eq:SL def}
SL_{\partial\Omega}[\varphi](\mathbf{x})
:= \int_{\partial\Omega} \Gamma^0(\mathbf{x},\mathbf{y})\,\varphi(\mathbf{y})\,\rmd S(\mathbf{y}),
\qquad \mathbf{x}\in \mathbb{R}^3\setminus\partial\Omega .
\end{align}
For a surface density $\psi$ on $\partial\Omega$, the double-layer potential is
\begin{align}\label{eq:DL def}
DL_{\partial\Omega}[\psi](\mathbf{x})
:= \int_{\partial\Omega} \partial_{\nu_{\mathbf{y}}}\Gamma^0(\mathbf{x},\mathbf{y})\,\psi(\mathbf{y})\,\rmd S(\mathbf{y}),
\qquad \mathbf{x}\in \mathbb{R}^3\setminus\partial\Omega .
\end{align}

The operator $N_{D_j}$ defined in \eqref{eq:Newtonian def} is self-adjoint and compact  \cite{CGS24,CGS25}, and hence it admits a strictly positive discrete spectrum 
$\{\lambda_{n}^{D_j}\}_{n\in\mathbb{N}}$. Under the dilation $D_j = \mathbf{z}_j + a B$, the eigenvalues satisfy the scaling law
\begin{align}\label{eq:lambda Dj=a2 lambda B}
    \lambda_{n}^{D_j}=a^{2}\lambda_{n}^{B},
\end{align}
where $\{\lambda_{n}^{B}\}_{n\in\mathbb{N}}$ are the eigenvalues of the corresponding operator
$N_{B}$ on the reference set $B$. 
Fix any $n_{0}\in\mathbb{N}$ and $j\in\{1,\dots,M\}$ and focus on the eigenvalue $\lambda_{n_{0}}^{D_j}$. 
In our model, we choose the driving angular frequency $\omega$ in the form
\begin{align}\label{eq:the angular frequency}
\frac{\omega^{2}}{\omega_{0}^{2}} \;=\; 1- c_{n_{0}}\,a^{h}.
\end{align}
Here $h>0$ and $c_{n_{0}}\in\mathbb{R}$ is a negative constant independent of $a$. 
The reference frequency $\omega_{0}$ is defined by
\begin{align}
\omega_{0}\;:=\;\left(\frac{\rho_1}{\lambda_{n_{0}}^{D_j}}\right)^{1/2}
\;=\;\left(\frac{\tilde{\rho}_1}{\lambda_{n_{0}}^{B}}\right)^{1/2},
\end{align}
where the second equality follows from the eigenvalue scaling above (cf. \eqref{eq:rho1 def} and \eqref{eq:lambda Dj=a2 lambda B}).

We are interested in the asymptotic regime $a\ll 1$, where the number of inclusions satisfies
\begin{equation}\label{eq:M def}
M \sim a^{\,h-1},\qquad 0<h<1,
\end{equation}
and the minimal distance between distinct hard inclusions satisfies
\begin{equation}\label{eq:dmin}
d := \min_{1\le i\neq j\le M}\,|\mathbf{z}_i-\mathbf{z}_j|
\sim a^{\frac{1-h}{3}}.
\end{equation}
These relations are consistent with the volumetric scaling law $M\sim d^{-3}$.

\vspace{0.2cm}

 Throughout the paper, we shall work under the following geometric assumption on the distribution of the hard inclusions inside $\Omega$.

\begin{assumption}\label{assumption1}
We examine a periodic configuration of hard inclusions within a bounded domain $\Omega$ of unit volume. Specifically, $\Omega$ encompasses a collection of subdomains $D_j$, for $j = 1, \ldots, M$. The domain $\Omega$ is divided as follows:
\begin{equation}\label{Decoupage-Omega-new}
\Omega = \Omega_{\mathrm{cube}} \cup \Omega_{r}, \qquad 
\Omega_{\mathrm{cube}} = \bigcup_{j=1}^{M} \Omega_j, \qquad 
\Omega_{r} = \bigcup_{j=1}^{\aleph} \Omega_j^{\star},
\end{equation}
where the integer-valued functions $M=M(a)$ and $\aleph=\aleph(a)$ depend on a small parameter $a>0$. Each cubic subdomain $\Omega_j$ is fully contained within the interior of $\Omega$, ensuring that $\Omega_{\mathrm{cube}} \subsetneq \Omega$ and none of the $\Omega_j$ touch the boundary $\partial \Omega$.

Within each cubic subdomain $\Omega_j$, there exists precisely one hard inclusion $D_j$, centered at $\mathbf{z}_j \in D_j \subset \Omega_j$, and the volume of each $\Omega_j$ is given by
\[
\lvert \Omega_j \rvert = a^{1-h}, \qquad j = 1,\ldots,M, \qquad 0 < h < 1,
\]
while the complementary subdomains $\Omega_j^{\star}$ are devoid of hard inclusions. A diagrammatic representation of this setup is shown in Figure~\ref{Fig1}.  
By utilizing a reference cell $\Omega_0$, all subdomains $\Omega_j$ are obtained through translations of $\Omega_0$. Thus, the periodic hard-inclusion arrangement can be described as
\begin{equation}\notag
D = \Omega_{\mathrm{cube}} \cap d\left( \mathbb{Z}^3 + \left( z + \frac{a}{d} B \right) \right),
\end{equation}
where $d$ is the minimal distance between hard inclusions as defined in \eqref{eq:dmin}, $z$ is a point within a unit cell, and $B \subset \mathbb{R}^3$ is a Lipschitz domain with characteristic diameter $\diam(B) \sim 1$. Additionally, we assume that the region $\Omega_{\mathrm{cube}}$ is kept at a distance from the outer boundary $\partial \Omega$ on the order of
\begin{equation}
\dist\!\left( \partial \Omega_{\mathrm{cube}}, \, \partial \Omega \right) \sim \kappa(a) \sim a^{\frac{1-h}{3}}, 
\qquad \text{with } a \ll 1,\; 0 < h < 1.
\end{equation}
This separation ensures that the collection of hard inclusions $D$ does not contact $\partial \Omega$, i.e.,
\begin{equation}\label{eq:dist D partial Omega}
\dist\!\left( D, \, \partial \Omega \right) \sim \kappa(a) \sim a^{\frac{1-h}{3}}, 
\qquad \text{for } a \ll 1,\; 0 < h < 1.
\end{equation}
Thus, the hard inclusions are arranged periodically within $\Omega$ while maintaining a finite distance from the boundary, which guarantees that no inclusion is positioned adjacent to $\partial \Omega$.
\end{assumption}

 \begin{figure}[H]
\begin{center}
\begin{tikzpicture}[scale=0.75]
  \draw[black, ultra thick] (0,0) ellipse (7 and 3.5);

  \begin{scope}
    \clip (0,0) ellipse (7 and 3.5);
    \foreach \x in {-7,-6.5,...,7} {
      \draw[gray, thick] (\x,-3.5) -- (\x,3.5);
    }
    \foreach \y in {-3.5,-3,...,3.5} {
      \draw[gray, thick] (-7,\y) -- (7,\y);
    }
  \end{scope}

  \foreach \x in {-7,-6.5,...,6.5} {
    \foreach \y in {-3.5,-3,...,3} {
      \pgfmathparse{
        ( (\x/7)^2 + (\y/3.5)^2 <= 1 )
        &&
        ( ((\x+0.5)/7)^2 + (\y/3.5)^2 <= 1 )
        &&
        ( (\x/7)^2 + ((\y+0.5)/3.5)^2 <= 1 )
        &&
        ( ((\x+0.5)/7)^2 + ((\y+0.5)/3.5)^2 <= 1 )
      }
      \ifnum\pgfmathresult>0
        \fill[blue] (\x+0.25,\y+0.25) circle (0.12);
      \fi
    }
  }

  \draw[red, dashed, very thick, ->] (-7,-2.75) -- (-3.75,-2.75);
  \node[black, anchor=east] at (-7.1,-2.75) {$\boldsymbol{\Omega_{n}^{\star}}$};

  \draw[red, ultra thick] (6,0) -- (6.5,0);
  \draw[red, ultra thick] (6,-0.5) -- (6.5,-0.5);
  \draw[red, ultra thick] (6,0) -- (6,-0.5);
  \draw[red, ultra thick] (6.5,0) -- (6.5,-0.5);

  \draw[red, very thick, ->] (6.8,-0.25) -- (6.5,-0.25);
  \draw[red, dashed, very thick, ->] (8.5,-0.25) -- (6.5,-0.25);
  \node[black, anchor=west] at (8.55,-0.25) {$\boldsymbol{\Omega_{j}}$};

  \draw[red, dashed, very thick, ->] (6.25,-2.25) -- (6.25,-0.25);
  \node[black, anchor=north] at (6.25,-2.35) {$\boldsymbol{D_{j}}$};

  \coordinate (P) at (3.25,2.75);
  \coordinate (Q) at (3.3385261155,3.0762901105);
  \coordinate (R) at (3.2814214000,2.8658131928);

  \draw[red, ultra thick] (R) -- (Q);

  \coordinate (M) at ($(R)!0.5!(Q)$);
  \draw[red, dashed, very thick, ->] (M) -- ++(2,0);
  \node[black, anchor=west] at ($(M)+(2,0)$) {$\boldsymbol{\kappa(a)}$};

  \node at (0,4.25) {$\boldsymbol{\Omega}$};
\end{tikzpicture}
\end{center}
\caption{An illustration of how the hard inclusion are distributed in $\Omega$.}
\label{Fig1}
\end{figure}
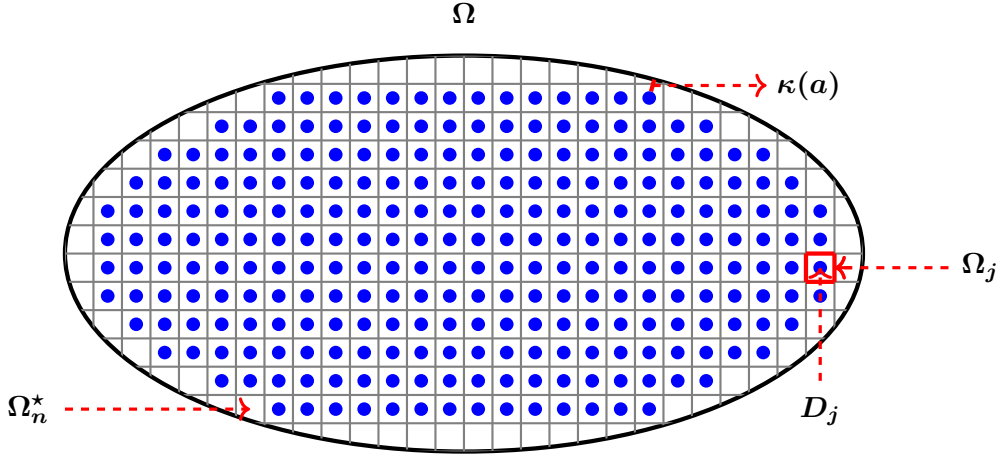

\noindent With the notation established above, we formulate the perturbed boundary value problem as
\begin{align}\label{eq:system vg}
  \begin{cases}
    \big( \mathcal{L}_{\lambda,\mu} + \omega^{2} \rho(\mathbf{x})(1 - \chi_{D}) + \omega^{2} \rho_{1} \chi_{D} \big)\, \mathbf{v}^{\mathbf{g}} = \mathbf{0} & \text{in } \Omega, \\[1mm]
    \partial_{\boldsymbol{\nu}} \mathbf{v}^{\mathbf{g}} = \mathbf{g} & \text{on } \partial\Omega .
  \end{cases}
\end{align}

Let $\mathbf{v}^{\mathbf{g}}(\cdot)$ denote the solution of \eqref{eq:system vg}. By multiplying \eqref{eq:system vg} by $\mathbf{u}^{\mathbf{f}}(\cdot)$, the solution of \eqref{eq:system}, and integrating over $\Omega$, we obtain
\begin{align}\label{eq:Lambda_e f g}
\langle \Lambda_{e}(\mathbf{f}) ; \mathbf{g} \rangle_{\mathbb{H}^{\frac{1}{2}}(\partial \Omega) \times \mathbb{H}^{-\frac{1}{2}}(\partial \Omega)} 
&= \lambda \langle \nabla \cdot \mathbf{u}^{\mathbf{f}}, \nabla \cdot \mathbf{v}^{\mathbf{g}} \rangle_{\mathbb{L}^{2}(\Omega)}  
+ \mu \int_{\Omega} \big( \nabla \mathbf{u}^{\mathbf{f}} : \nabla \mathbf{v}^{\mathbf{g}} + \nabla \mathbf{u}^{\mathbf{f}} : (\nabla \mathbf{v}^{\mathbf{g}})^{\top} \big)\, \rmd \mathbf{x} \notag\\
&\quad - \omega^{2} \langle \mathbf{u}^{\mathbf{f}}, \rho(\mathbf{x}) \mathbf{v}^{\mathbf{g}} \rangle_{\mathbb{L}^{2}(\Omega)} 
      - \omega^{2} \langle \mathbf{u}^{\mathbf{f}}, (\rho_{1} - \rho(\mathbf{x})) \mathbf{v}^{\mathbf{g}} \rangle_{\mathbb{L}^{2}(D)} .
\end{align}

\noindent Here $\Lambda_{e}(\cdot)$ is the N--D operator introduced in \eqref{eq:Lambda_e def}, viewed as a map from $\mathbb{H}^{-\frac{1}{2}}\!\left( \partial \Omega \right)$ to $\mathbb{H}^{\frac{1}{2}}\!\left( \partial \Omega \right)$ and defined via
\begin{align}
\langle \Lambda_{e}(\mathbf{f}) ; \mathbf{g} \rangle_{\mathbb{H}^{\frac{1}{2}}(\partial \Omega) \times \mathbb{H}^{-\frac{1}{2}}(\partial \Omega)} 
:= \int_{\partial \Omega} \mathbf{u}^{\mathbf{f}}(\mathbf{x}) \cdot \mathbf{g}(\mathbf{x}) \, \rmd \sigma(\mathbf{x}) .
\end{align}

For brevity, we henceforth employ integral notation. We denote by $\Lambda_{D}(\cdot)$ the  N--D map corresponding to the background medium with a cluster of hard inclusions, that is, the map associated with \eqref{eq:system vg}. Multiplying \eqref{eq:system} by $\mathbf{v}^{\mathbf{g}}(\cdot)$, integrating over $\Omega$, and using the self-adjointness of $\Lambda_{D}$ together with \eqref{eq:Lambda_e f g}, we obtain
\begin{equation}\label{eq:Lambda_D f g}
\langle \Lambda_{D}(\mathbf{f}) ; \mathbf{g} \rangle_{\mathbb{H}^{\frac{1}{2}}(\partial \Omega) \times \mathbb{H}^{-\frac{1}{2}}(\partial \Omega)} 
- \langle \Lambda_{e}(\mathbf{f}) ; \mathbf{g} \rangle_{\mathbb{H}^{\frac{1}{2}}(\partial \Omega) \times \mathbb{H}^{-\frac{1}{2}}(\partial \Omega)}   
= \omega^{2} \left\langle \big( \rho_{1} - \rho(\mathbf{x}) \big) \mathbf{v}^{\mathbf{g}} ; \mathbf{u}^{\mathbf{f}} \right\rangle_{\mathbb{L}^{2}(D)} .
\end{equation}

In an analogous fashion, define $\mathbf{q}^{\mathbf{g}}(\cdot)$ to be the solution of
\begin{equation}\label{eq:system qg}
\left\{
\begin{aligned}
(\mathcal{L}_{\lambda,\mu} + \omega^{2} \rho(\mathbf{x}) - \mathcal{P}^{2})\, \mathbf{q}^{\mathbf{g}} &= \mathbf{0}, && \text{in } \Omega,\\
\partial_{\boldsymbol{\nu}} \mathbf{q}^{\mathbf{g}} &= \mathbf{g}, && \text{on } \partial\Omega .
\end{aligned}
\right.
\end{equation}

\noindent Here
\begin{equation}\label{eq:P2 def}
\mathcal{P}^{2} := -\frac{\left(\langle \mathcal{I} ; \tilde{\mathbf{e}}_{n_{0}} \rangle_{\mathbb{L}^{2}(B)}\right)^{2}}
              {\lambda^{B}_{n_{0}}\,c_{n_{0}}}  ,
\end{equation}
and we write $\Lambda_{\mathcal{P}}(\cdot)$ for the N--D map corresponding to the effective (equivalent) background. Then
\begin{equation}\label{eq:LambdaP-Lambdae}
\langle \Lambda_{\mathcal{P}}(\mathbf{f}) ; \mathbf{g} \rangle_{\mathbb{H}^{\frac{1}{2}}(\partial \Omega) \times \mathbb{H}^{-\frac{1}{2}}(\partial \Omega)} 
- \langle \Lambda_{e}(\mathbf{f}) ; \mathbf{g} \rangle_{\mathbb{H}^{\frac{1}{2}}(\partial \Omega) \times \mathbb{H}^{-\frac{1}{2}}(\partial \Omega)} 
= - \mathcal{P}^{2} \langle \mathbf{q}^{\mathbf{g}} ; \mathbf{u}^{\mathbf{f}} \rangle_{\mathbb{L}^{2}(\Omega)} .
\end{equation}

Combining \eqref{eq:Lambda_D f g} with \eqref{eq:LambdaP-Lambdae} gives
\begin{equation}\label{eq:LambdaD-LambdaP}
\langle ( \Lambda_{D} - \Lambda_{\mathcal{P}} )(\mathbf{f}) ; \mathbf{g} \rangle_{\mathbb{H}^{\frac{1}{2}}(\partial \Omega) \times \mathbb{H}^{-\frac{1}{2}}(\partial \Omega)}  
= \omega^{2} \left\langle \big( \rho_{1} - \rho(\mathbf{x}) \big) \mathbf{v}^{\mathbf{g}} ; \mathbf{u}^{\mathbf{f}} \right\rangle_{\mathbb{L}^{2}(D)}  
+ \mathcal{P}^{2} \langle \mathbf{q}^{\mathbf{g}} ; \mathbf{u}^{\mathbf{f}} \rangle_{\mathbb{L}^{2}(\Omega)} .
\end{equation}

In the sequel, we establish that, when $M$ is taken sufficiently large or the size parameter $a$ is made sufficiently small, a medium modified by a cluster of $M$ hard inclusions becomes indistinguishable from the effective background; equivalently, the N--D operator $\Lambda_{D}(\cdot)$ converges to $\Lambda_{\mathcal{P}}(\cdot)$.

\subsubsection{Theorem \ref{thm:N--D}: linking the original N--D map $\Lambda_{D}$ with the effective N--D map $\Lambda_{\mathcal{P}}$}\label{subsec:1.2.1}

\begin{thm}\label{thm:N--D}
Assume that the domain $\Omega$ is $C^{2}$-smooth. Let the density $\rho(\bx)$ satisfy
$\rho(\cdot)\in \mathbb{W}^{1,\infty}(\Omega)$, and suppose the operating frequency $\omega$
meets the condition \eqref{eq:the angular frequency}. Choose a parameter $h$ with $\frac{1}{3}<h<1$, and assume that the hard inclusions $D_m$ are arranged as in Assumption \ref{assumption1}. Then the following holds:
\begin{align}
   \left\langle \Lambda_{D}(\mathbf{f}) ; \mathbf{g} \right\rangle_{\mathbb{H}^{1/2}(\partial\Omega)\times \mathbb{H}^{-1/2}(\partial\Omega)}
&\xrightarrow{a\to 0}  
\left\langle \Lambda_{\mathcal{P}}(\mathbf{f}) ; \mathbf{g} \right\rangle_{\mathbb{H}^{1/2}(\partial\Omega)\times \mathbb{H}^{-1/2}(\partial\Omega)} 
\end{align}
uniformly for $(\mathbf{f},\mathbf{g})\in \mathbb{H}^{-1/2}(\partial\Omega)\times \mathbb{H}^{-1/2}(\partial\Omega)$. Moreover, we have the quantitative estimate
\begin{align}\label{eq:Lambda D-Lambda P in thm}
    \left\|\Lambda_{D}-\Lambda_{\mathcal{P}}\right\|_{\mathcal{L}\!\left(\mathbb{H}^{-1/2}(\partial\Omega), \mathbb{H}^{1/2}(\partial\Omega)\right)}
&\lesssim a^{\frac{(1-h)\,(9-5\epsilon)}{18\,(3-\epsilon)}}\,\mathcal{P}^{6},
\qquad a\ll 1.
\end{align}
Here $\epsilon>0$ denotes an arbitrarily small positive constant.
\end{thm}

Regarding the parameter $\epsilon$ appearing in \eqref{eq:Lambda D-Lambda P in thm} and the parameter $h$ appearing in \eqref{eq:the angular frequency} and \eqref{eq:M def}, we now provide some additional comments. We also comment on the conclusion of Theorem~\ref{thm:N--D}.

\begin{rem}
Four remarks are in order.
    \begin{enumerate}
        \item The parameter $\epsilon$ appearing in \eqref{eq:Lambda D-Lambda P in thm} stems from the regularity properties of the fundamental solution $\Gamma_0(\cdot, \cdot)$, specifically its integrability in the space $\mathbb{L}^{3 - \epsilon}(\Omega)$ as defined in \eqref{eq:Gamma0 def}.

        \item The parameter $h$, found in \eqref{eq:the angular frequency} and \eqref{eq:M def}, characterizes the density of the hard inclusion distribution within $\Omega$. A value of $h$ approaching $0$ corresponds to a dense regime, whereas $h$ approaching $1$ indicates a dilute (sparse) distribution.

        \item Under the hypotheses of Theorem \ref{thm:N--D}, the N--D map $\Lambda_\mathcal{P}(\cdot)$ can be computed effectively.

        \item To obtain \eqref{eq:Err-final} from \eqref{eq:thm h from equation}, we require the inequality $\frac{7h-1}{6}>\frac{1-h}{3}$, which leads to the condition $h>\frac{1}{3}$ stated in the theorem.

        \item Given the scaling law $M \sim a^{h-1}$ and assuming $\epsilon$ is sufficiently small, the estimate \eqref{eq:Lambda D-Lambda P in thm} can be reformulated as
        \begin{align}
            \|\Lambda_D - \Lambda_\mathcal{P}\|_{\mathcal{L}(\mathbb{H}^{-1/2}(\partial\Omega), \mathbb{H}^{1/2}(\partial\Omega))} \lesssim M^{\frac{5\epsilon - 9}{18(3 - \epsilon)}} \mathcal{P}^6, \quad \text{for } M \gg 1.
        \end{align}
        Consequently, it is possible to tune $M$ (or equivalently, the radius $a$) such that
        \begin{align}
            M^{\frac{5\epsilon - 9}{18(3 - \epsilon)}} \mathcal{P}^6 \ll 1.
        \end{align}
    \end{enumerate}
\end{rem}

The proof of Theorem \ref{thm:N--D} relies on the point-interaction method, classically referred to as the Foldy–Lax approximation. Our strategy begins by approximating the left-hand side of \eqref{eq:Lambda D-Lambda P in thm} using a linear combination of components from a vector that satisfies a specific algebraic system. This system encapsulates the multiple interacting effects arising from the injected hard inclusions; notably, the associated interaction  matrix consists entirely of positive coefficients, a property ensured by the sign convention chosen for $c_{n_0}$ in \eqref{eq:the angular frequency}. To establish the invertibility of this algebraic system uniformly with respect to the large parameter $M$, we first prove the invertibility of the corresponding limiting continuous integral equation. Subsequently, we demonstrate that the algebraic system functions as a discrete approximation of this continuous integral equation.

\subsubsection{Theorem \ref{thm:linearisation}: characterization of the linearized effective N--D map $\Lambda_{\mathcal{P}}(\cdot)$}\label{subsec:1.2.2}

\begin{thm}\label{thm:linearisation}
Let $\bff\in \bbH^{-1/2}(\partial\Omega)$ and let $\bQ^{\bff}$ be the solution to
\begin{align}\label{eq:Qf def}
\begin{cases}
(\mathcal{L}_{\lambda,\mu} - \mathcal{P}^{2})\, \bQ^{\bff} = \mathbf{0} & \text{in } \Omega,\\ 
\partial_{\nu} \bQ^{\bff} = \bff & \text{on } \partial\Omega.
\end{cases}
\end{align}
Denote by $\gamma:\bbH^{s}(\Omega)\to \bbH^{s-1/2}(\partial\Omega)$, $s\ge \tfrac12$, the trace operator. Then the N--D map of the effective background, written as $\Lambda_{\mathcal{P}}$, satisfies the first–order linearisation
\begin{align}\label{eq:Lambda P f-Qf in thm}
\Lambda_{\calP}(\bff) - \gamma \left(\bQ^{\bff}\right) 
= \omega^{2} \gamma\big( \calW^{\bQ^{\bff}} \big)
  + \mathcal{O}\!\left( \|\bff\|_{\bbH^{-1/2}(\partial\Omega)} \,\calP^{-4} \right),
\end{align}
where $\calW^{\bQ^{\bff}}=\mathcal{N}^{\calP}(\rho(\bx)\bQ^{\bff})$ is uniquely determined by
\begin{align}\label{eq:WQf def}
\begin{cases}
(\mathcal{L}_{\lambda,\mu} - \calP^{2}) \calW^{\bQ^{\bff}} = - \rho(\bx) \bQ^{\bff} & \text{in } \Omega,\\[2mm]
\partial_{\nu} \calW^{\bQ^{\bff}} = \mathbf{0} & \text{on } \partial\Omega.
\end{cases}
\end{align}
\end{thm}

{

Therefore, knowledge of $\Lambda_{\calP}(\bff)$ determines $\bQ^{\bff}$ and hence allows one to construct $\calW^{\bQ^{\bff}}$ for every $\bff\in\bbH^{-1/2}(\partial\Omega)$. Moreover, the remainder term defines a bounded operator from $\bbH^{-1/2}(\partial\Omega)$ to $\bbH^{1/2}(\partial\Omega)$ with norm $O(\calP^{-4})$, i.e.,
\[
\|R(\bff)\|_{\bbH^{1/2}(\partial\Omega)} \le C\,\|\bff\|_{\bbH^{-1/2}(\partial\Omega)}\,\calP^{-4}.
\]
The proof is based on a Lippmann--Schwinger representation, which decomposes the field into the solution of the homogeneous problem and a perturbative correction. The perturbative smallness is quantified via spectral and scaling estimates for the Newtonian volume potential associated with the homogeneous operator, combined with Calder\'on--Zygmund-type bounds.
}

\subsubsection{Reconstruction of $\rho(\bx)$ from the linearisation of $\Lambda_\calP (\cdot)$: Theorem \ref{thm:CGO}}\label{subsec:1.2.3}

The following result provides an explicit reconstruction procedure for the density $\rho(\bx)$ from the linearized operator $\Lambda_{\mathcal{P}}(\cdot)$.

\begin{thm}\label{thm:CGO}
Let $\boldsymbol{\xi}\in\mathbb{Z}^3$ be any nonzero vector and choose an orthonormal basis $\{\mathbf{e}_{1},\mathbf{e}_{2},\mathbf{e}_{3}\}$ of $\mathbb{R}^{3}$ such that
\[
\mathbf{e}_{1} := \frac{\boldsymbol{\xi}}{|\boldsymbol{\xi}|}.
\]
With respect to this basis, define
\begin{align}\label{eq:zeta1 def in thm}
\boldsymbol{\zeta}_{1}
= -\frac{|\boldsymbol{\xi}|}{2} \mathbf{e}_{1}
  + \rmi \sqrt{ t^2 - k_s^2 + \frac{|\boldsymbol{\xi}|^{2}}{4}} \,\mathbf{e}_{2}
  + t \mathbf{e}_3,
\qquad
\boldsymbol{\eta}_1 = \mathbf{e}_1 + \frac{|\boldsymbol{\xi}|}{2 t} \mathbf{e}_2,
\end{align}
We now define $\bQ^\bff(\cdot)$ by
\begin{align}\label{eq:Qf def in th3}
\bQ^\bff(\bx)
:= e^{\rmi \boldsymbol{\zeta}_1 \cdot \bx} \big( \boldsymbol{\eta}_1 + \bF_1(\bx) \big),
\quad \bx \in \mathbb{R}^3,
\end{align}
where $\mathbf{F}_1(\bx)$ is the solution of
\begin{align}
& (\lambda+\mu)\Big[
     \nabla(\nabla\cdot\bF_1)
     + \rmi \boldsymbol{\zeta}_1 (\nabla\cdot\bF_1)
     + \rmi \nabla( \boldsymbol{\zeta}_1\cdot\bF_1)
     - \boldsymbol{\zeta}_1( \boldsymbol{\zeta}_1\cdot\bF_1)
   \Big] \notag\\
&\quad + \mu\Big(
     \Delta\bF_1
     + 2\rmi\,(\boldsymbol{\zeta}_1\cdot\nabla)\bF_1
     - k_s^2 \bF_1\Big) - \mathcal{P}^2\bF_1
  = \big(\mu k_s^2 + \mathcal{P}^2\big) \boldsymbol{\eta}_1,
\qquad \bx\in\mathbb{R}^3.
\end{align}
In an analogous way, introduce $\bQ^\bg(\cdot)$ via
\begin{align}\label{eq:Qg def in th3}
\bQ^\bg(\bx)
:= e^{\rmi \boldsymbol{\zeta}_2 \cdot \bx} \big(\boldsymbol{\eta}_2 + \bF_2(\bx)\big),
\quad \bx \in \mathbb{R}^3,
\end{align}
where $\boldsymbol{\zeta}_2$ and $\boldsymbol{\eta}_2$ are defined by
\begin{align}\label{eq:zeta2 def}
\boldsymbol{\zeta}_{2}
= -\frac{|\boldsymbol{\xi}|}{2} \mathbf{e}_{1}
  - \rmi \sqrt{ t^2 - k_s^2 + \frac{|\boldsymbol{\xi}|^{2}}{4}} \,\mathbf{e}_{2}
  - t \mathbf{e}_3,
\qquad
\boldsymbol{\eta}_2 = \mathbf{e}_1 - \frac{|\boldsymbol{\xi}|}{2 t} \mathbf{e}_2.
\end{align}
Here $\bF_2(\cdot)$ solves
\begin{align}\label{eq:F2 def1}
& (\lambda+\mu)\Big[
     \nabla(\nabla\cdot\bF_2)
     + \rmi \boldsymbol{\zeta}_2(\nabla\cdot\bF_2)
     + \rmi \nabla( \boldsymbol{\zeta}_2\cdot\bF_2)
     - \boldsymbol{\zeta}_2( \boldsymbol{\zeta}_2\cdot\bF_2)
   \Big] \notag\\
&\quad + \mu\Big(
     \Delta\bF_2
     + 2\rmi\,(\boldsymbol{\zeta}_2 \cdot\nabla)\bF_2
     - k_s^2 \bF_2 \Big) - \mathcal{P}^2\bF_2
  = \big(\mu k_s^2 + \mathcal{P}^2\big) \boldsymbol{\eta}_2,
\qquad \bx\in\mathbb{R}^3.
\end{align}
If we choose
\begin{align}
t = \sqrt{\mathcal{P}^{4+2\iota}+k_s^2}, \qquad \iota > 0,
\end{align}
in the definitions \eqref{eq:zeta1 def in thm} and \eqref{eq:zeta2 def} of $\boldsymbol{\zeta}_1$ and $\boldsymbol{\zeta}_2$, respectively, then the density $\rho(\bx)$ admits the approximate reconstruction
\begin{align}\label{eq: rho x in thm}
\rho(\bx)
= (2\pi)^{-3} \sum_{\boldsymbol{\xi}\in\mathbb{Z}^3}
\big\langle \mathcal{W}^{\bQ^\bff}_{\boldsymbol{\xi}}, \bg \big\rangle_{\bbH^{1/2}(\partial\Omega) \times \bbH^{-1/2}(\partial\Omega)}\,
e^{\rmi \boldsymbol{\xi} \cdot \bx}  
+ \mathcal{O}(\mathcal{P}^{-\iota}),
\end{align}
where the convergence holds in the $\bbL^2(\Omega)$ sense as $\mathcal{P}\to\infty$.
\end{thm}

The conclusion of Theorem \ref{thm:CGO} provides a reconstruction procedure for $\rho(\cdot)$. Here, $\mathcal{W}^{\bQ^\bff}_{\boldsymbol{\xi}}$ can be obtained from \eqref{eq:Lambda P f-Qf in thm}. In this remark, we present an alternative choice of CGO solutions in order to reconstruct the Fourier transform of $\rho(\bx)$ and thereby invert the density $\rho(\bx)$. The core idea is to test the equation satisfied by $\mathcal{W}^{\bQ^\bff}$ against a suitably chosen solution of the adjoint Lam\'e-type operator. 

{
\begin{rem}
In this paper we work with CGO solutions of the form \eqref{eq:Qf def in th3} and
\eqref{eq:Qg def in th3}. Their construction follows standard arguments; see,
for instance, \cite[Theorem~2.3]{BFP18} or our earlier work \cite[Lemma~2.1]{DGT25}.
Accordingly, we omit the details of the CGO construction. Moreover, in the proof of the theorem in Section~\ref{sec:Theorem 3 CGO}, the
estimates for the terms involving $\bF_1(\bx)$, such as \eqref{eq:F1 estimation}
and \eqref{eq:nabla F1 estimation}, can be obtained by the same method as in
\cite[Theorem~2.3]{BFP18}, and we do not repeat the argument here.
\end{rem}
}

\begin{rem} \label{rem:CGO_inversion}
Let $\bv$ be any solution to
\[
(\mathcal{L}_{\lambda,\mu}- \mathcal{P}^{2})\bv = \mathbf{0} \quad \text{in } \Omega.
\]
Applying Green's formula (integration by parts) to the pair $(\mathcal{W}^{\bQ^{\bff}},\bv)$ over $\Omega$ yields
\begin{align}\label{eq:rho Qf vx}
\int_{\Omega} \rho(\bx)\, \bQ^{\bff}(\bx) \cdot \bv(\bx)\, \rmd \bx
= \int_{\partial \Omega} \mathcal{W}^{\bQ^{\bff}}(\bx) \, \partial_{\nu} \bv(\bx) \, \rmd \mathcal{S}(\bx).
\end{align}

\noindent
Fix any nonzero vector $\boldsymbol{\xi} \in \mathbb{R}^{3}$ and choose an orthonormal basis $\{\mathbf{e}_{1},\mathbf{e}_{2},\mathbf{e}_{3}\}$ of $\mathbb{R}^{3}$ such that
\[
\mathbf{e}_{1} := \frac{\boldsymbol{\xi}}{|\boldsymbol{\xi}|}.
\]
With respect to this basis, we define the complex vectors
\begin{subequations}
\begin{align}
\boldsymbol{\zeta}_1 &=- \frac{|\boldsymbol{\xi}|}{2} \mathbf{e}_1 + \rmi \mathbf{e}_2 \sqrt{\frac{\mathcal{P}^2}{\mu} + \frac{|\boldsymbol{\xi}|^2}{4}} ,\qquad
\boldsymbol{\eta}_1 = \rmi \sqrt{ 1+\frac{4\mathcal{P}^2}{\mu |\boldsymbol{\xi}|^2} }\, \mathbf{e}_1 + \mathbf{e}_2 , \\
\boldsymbol{\zeta}_2 &=- \frac{|\boldsymbol{\xi}|}{2} \mathbf{e}_1 - \rmi \mathbf{e}_2 \sqrt{\frac{\mathcal{P}^2}{\mu} + \frac{|\boldsymbol{\xi}|^2}{4}} ,\qquad
\boldsymbol{\eta}_2 = \rmi \sqrt{ 1+\frac{4\mathcal{P}^2}{\mu |\boldsymbol{\xi}|^2} }\, \mathbf{e}_1 - \mathbf{e}_2. \label{eq:zeta2 eta2 def}
\end{align}
\end{subequations}
We explicitly define the test functions by
\begin{align}
\bQ^{\bff}(\bx) := \boldsymbol{\eta}_1 \, e^{ \rmi \boldsymbol{\zeta}_1 \cdot \bx}, 
\qquad
\bv(\bx) := \boldsymbol{\eta}_2 \, e^{\rmi \boldsymbol{\zeta}_2 \cdot \bx}.
\end{align}
A direct computation shows that
\begin{align}
\boldsymbol{\zeta}_1 \cdot \boldsymbol{\zeta}_1 = \boldsymbol{\zeta}_2 \cdot \boldsymbol{\zeta}_2 = -\frac{\mathcal{P}^2}{\mu}, \quad  
\boldsymbol{\zeta}_1 \cdot \boldsymbol{\eta}_1 = \boldsymbol{\zeta}_2 \cdot \boldsymbol{\eta}_2 = 0,
\end{align}
which implies that
\[
(\mathcal{L}_{\lambda,\mu}- \mathcal{P}^{2}) \bQ^\bff = \mathbf{0} ,
\qquad
(\mathcal{L}_{\lambda,\mu}- \mathcal{P}^{2})\bv = \mathbf{0} 
\quad \text{in } \mathbb{R}^{3}.
\]
Furthermore,
\begin{align}
\boldsymbol{\zeta}_{1} + \boldsymbol{\zeta}_{2} = -|\boldsymbol{\xi}| \mathbf{e}_{1} = -\boldsymbol{\xi}.
\end{align}
With this specific choice of CGO solutions, the dot product in the volume integral \eqref{eq:rho Qf vx} simplifies to
\begin{align}
\bQ^{\bff}(\bx) \cdot \bv(\bx) 
= (\boldsymbol{\eta}_1 \cdot \boldsymbol{\eta}_2) \, e^{ \rmi(\boldsymbol{\zeta}_{1}+\boldsymbol{\zeta}_{2}) \cdot \bx} 
= \left(-2-\frac{4\mathcal{P}^2}{\mu|\boldsymbol{\xi}|^2}\right) e^{ -\rmi \boldsymbol{\xi} \cdot \bx}.
\end{align}
Substituting this identity into \eqref{eq:rho Qf vx}, we obtain
\begin{align}\label{eq:Fourier-rho-relation}
\left(-2-\frac{4\mathcal{P}^2}{\mu|\boldsymbol{\xi}|^2}\right) \int_{\Omega}\rho(\bx) e^{-\rmi \boldsymbol{\xi} \cdot \bx}\, \rmd \bx
= \int_{\partial \Omega} \mathcal{W}^{\bQ^{\bff}}(\bx) \, \partial_{\nu} \big(  \boldsymbol{\eta}_2 e^{\rmi \boldsymbol{\zeta}_2 \cdot \bx} \big)\, \rmd \mathrm{S}(\bx).
\end{align}

Combining the definitions of $\boldsymbol{\eta}_2$ and $\boldsymbol{\xi}$ in \eqref{eq:zeta2 eta2 def} with \eqref{eq:partial nu def}, we obtain
\begin{align}
    \partial_{\boldsymbol{\nu}}\big( \boldsymbol{\eta}_2 e^{\rmi \boldsymbol{\zeta}_2\cdot\bx}\big)
= \rmi\mu\, e^{\rmi\boldsymbol{\zeta}_2\cdot\bx}
\Big[ (\boldsymbol{\zeta}_2\!\cdot\!\boldsymbol{\nu})\,\boldsymbol{\eta}_2
      + (\boldsymbol{\eta}_2\!\cdot\!\boldsymbol{\nu})\,\boldsymbol{\zeta}_2\Big].
\end{align}
Substituting this expression into \eqref{eq:Fourier-rho-relation} yields
\begin{align}
\left(-2-\frac{4\mathcal{P}^2}{\mu|\boldsymbol{\xi}|^2}\right) \int_{\Omega}\rho(\bx) e^{-\rmi \boldsymbol{\xi} \cdot \bx}\, \rmd \bx
= \rmi \mu \int_{\partial \Omega} \mathcal{W}^{\bQ^{\bff}}(\bx)\,
e^{\rmi\boldsymbol{\zeta}_2\cdot\bx}
\Big[ (\boldsymbol{\zeta}_2\!\cdot\!\boldsymbol{\nu})\,\boldsymbol{\eta}_2
      + (\boldsymbol{\eta}_2\!\cdot\!\boldsymbol{\nu})\,\boldsymbol{\zeta}_2\Big] \rmd \mathrm{S}(\bx).
\end{align}
 Since the $-2-\frac{4\mathcal{P}^2}{\mu|\boldsymbol{\xi}|^2}$ is non-vanishing for $\boldsymbol{\xi} \neq \mathbf{0}$, the boundary data on the right-hand side determine the Fourier transform of $\rho(\cdot)$ for all $\boldsymbol{\xi} \in \mathbb{R}^3 \setminus \{\mathbf{0}\}$. By the inversion formula for the Fourier transform, we obtain
\begin{align}\label{eq:rho x in remark}
    \rho(\bx)
    = - \frac{\rmi \mu^2}{(2\pi)^3}
      \int_{\mathbb{R}^3}
      \frac{|\boldsymbol{\xi}|^2}{2\mu|\boldsymbol{\xi}|^2+4\mathcal{P}^2}
      \left(
      \int_{\partial\Omega} \mathcal{W}^{\bQ^{\bff}}(\by) \,
      e^{\rmi\boldsymbol{\zeta}_2\cdot\by}
      \Big[ (\boldsymbol{\zeta}_2\!\cdot\!\boldsymbol{\nu}(\by))\,\boldsymbol{\eta}_2
            + (\boldsymbol{\eta}_2\!\cdot\!\boldsymbol{\nu}(\by))\,\boldsymbol{\zeta}_2\Big]
      \rmd \mathrm{S}(\by)
      \right) e^{\rmi\boldsymbol{\xi}\cdot\bx}  \rmd \boldsymbol{\xi}.
\end{align}

\end{rem}

\section{Proof of Theorem \ref{thm:linearisation} }\label{sec:Theorem 2}

The purpose of this section is to prove Theorem \ref{thm:linearisation}, that is, to obtain a first-order linearisation of the N--D map $\Lambda_{\mathcal{P}}(\cdot)$ associated with the equivalent background medium. Let $\bq^{\bff}$ solve the Lippmann–Schwinger-type integral equation
\begin{align}
\bq^{\bff}(\bx)-\omega^{2}\,\mathcal{N}^{\calP}\!\big(\rho(\cdot)\,\bq^{\bff}\big)(\bx)=\bQ^{\bff}(\bx), \qquad \bx\in\Omega,
\end{align}
where the Newtonian potential $\mathcal{N}^{\calP}:\bbL^{2}(\Omega)\to\bbH^{2}(\Omega)$ is
\begin{align}\label{eq:Newtonian potential NP in section2}
\mathcal{N}^{\calP}(\bg)(\bx)=\int_{\Omega} \Gamma_{\calP}(\bx,\by)\,\bg(\by)\,\rmd \by,\qquad \bx\in\Omega,
\end{align}
and $\Gamma_{\calP}$ is determined by
\begin{align}\label{eq:Gamma P def}
\begin{cases}
(\mathcal{L}_{\lambda,\mu}- \calP^{2})\,\Gamma_{\calP}(\bx,\by)
= -\delta_{\by}(\bx)\,\mathcal{I} & \text{in }\Omega,\\[2mm]
\partial_{\nu_{\bx}} \Gamma_{\calP}(\bx,\by)=\mathbf{0} & \text{on }\partial\Omega.
\end{cases}
\end{align}

\noindent
{ By Lemma~\ref{lem:NP-spectral}, we conclude that $\mathcal{N}^{\calP}$ is a bounded, self-adjoint, positive, compact operator.} Iterating the integral equation yields, on the boundary $\partial\Omega$,
\begin{align}\label{eq:qf-Qf=}
\bq^{\bff}-\bQ^{\bff}
=\omega^{2}\,\gamma \calN^{\calP}\!\big( \rho(\cdot)\,\bQ^{\bff} \big)
+\sum_{j\ge 2}\big(\calK_{j}\otimes(\rho(\cdot))\big),
\end{align}
where, for instance,
\begin{align}
\calK_{2}\otimes(\rho(\cdot))
&=(\omega^{2})^{2}\,\gamma \calN^{\calP}\!\big(\rho(\cdot)\,\calN^{\calP}(\rho(\cdot)\bQ^{\bff})\big),\\
\calK_{3}\otimes(\rho(\cdot))
&=(\omega^{2})^{3}\,\gamma \calN^{\calP}\!\big(\rho(\cdot)\,\calN^{\calP}\big(\rho(\cdot)\,\calN^{\calP}(\rho(\cdot)\bQ^{\bff})\big)\big),
\end{align}
and, in general,
\begin{align}
\calK_{j}\otimes(\rho(\cdot))
&=(\omega^{2})^{j}\,\gamma
\int_{\Omega}\!\cdots\!\int_{\Omega}
\Gamma_{\calP}(\cdot,\by_{1}) \rho(\by_{1})\cdots 
\Gamma_{\calP}(\by_{j-1},\by_{j}) \rho(\by_{j})\,\bQ^{\bff}(\by_{j})\,\rmd \by_{1}\cdots \rmd \by_{j}.
\end{align}

Throughout, we denote by \(\calN^{\calP}\) the Newtonian operator introduced in \eqref{eq:Newtonian potential NP in section2}. We use \(\gamma\) for the trace operator \(\bbH^{s}(\Omega)\to \bbH^{s-1/2}(\partial\Omega)\) valid for all \(s\ge \tfrac{1}{2}\), where \(\Omega\) is a smooth domain. The lemma below will be employed to analyze the convergence of the preceding series in the \(\bbH^{1/2}(\partial\Omega)\) norm.

\begin{lem}\label{lem:NP norm lem} 
The operator $\calN^{\calP}$ satisfies
\begin{align}
\|\calN^{\calP}\|_{\mathcal{L}(\bbL^{2}(\Omega);\bbL^{2}(\Omega))}&=\mathcal{O}(\calP^{-2}),    \label{eq:NP norm}  \\
\|\gamma \calN^{\calP}\|_{\mathcal{L}(\bbL^{2}(\Omega);\bbH^{1/2}(\partial\Omega))}&=\mathcal{O}(\calP^{-1}). \label{eq:gamma NP norm}
\end{align}
\end{lem}

\begin{proof}
    See Subsection \ref{subsec:NP norm proof}.
\end{proof}

\noindent
Using these bounds one finds, for $j=2,3$,
\begin{align}
\|\calK_{2}\otimes(\rho(\cdot))\|_{\bbH^{1/2}(\partial\Omega)}
&\le (\omega^{2})^{2}\,\|\gamma \calN^{\calP}\|_{\mathcal{L}(\bbL^{2}(\Omega);\bbH^{1/2}(\partial\Omega))} \|\rho(\cdot)\|_{\bbL^{\infty}(\Omega)}^{2}\,\|\calN^{\calP}\|_{\mathcal{L}(\bbL^{2}(\Omega);\bbL^{2}(\Omega))} \|\bQ^{\bff}\|_{\bbL^{2}(\Omega)},  \notag \\
\|\calK_{3}\otimes(\rho(\cdot))\|_{\bbH^{1/2}(\partial\Omega)}
&\le (\omega^{2})^{3}\,\|\gamma \calN^{\calP}\|_{\mathcal{L}(\bbL^{2}(\Omega);\bbH^{1/2}(\partial\Omega))}\|\rho(\cdot)\|_{\bbL^{\infty}(\Omega)}^{3}\,\|\calN^{\calP}\|_{\mathcal{L}(\bbL^{2}(\Omega);\bbL^{2}(\Omega))}^{2} \|\bQ^{\bff}\|_{\bbL^{2}(\Omega)}. \notag
\end{align}
For any $j\ge 2$,
\begin{align}\notag
    \|\calK_{j}\otimes(\rho(\cdot))\|_{\bbH^{1/2}(\partial\Omega)}
\le \mathcal{C}\,\|\rho(\cdot)\|_{\bbL^{\infty}(\Omega)}
\big(\omega^{2}\|\calN^{\calP}\|_{\mathcal{L}(\bbL^{2}(\Omega);\bbL^{2}(\Omega))}\,\|\rho(\cdot)\|_{\bbL^{\infty}(\Omega)}\big)^{j-1},
\end{align}
where
\begin{align}\label{eq:calC def}
    \mathcal{C}:=\omega^{2}\,\|\bQ^{\bff}\|_{\bbL^{2}(\Omega)}\,\|\gamma \calN^{\calP}\|_{\mathcal{L}(\bbL^{2}(\Omega);\bbH^{1/2}(\partial\Omega))}.
\end{align}

\noindent
Therefore,
\begin{align}\label{eq:sum calKj}
    \Big\|\sum_{j\ge 2}\calK_{j}\otimes(\rho(\cdot))\Big\|_{\bbH^{1/2}(\partial\Omega)}
    &\le \sum_{j\ge 2}\mathcal{C}\,\|\rho(\cdot)\|_{\bbL^{\infty}(\Omega)}
    \big(\omega^{2}\|\calN^{\calP}\|_{\mathcal{L}(\bbL^{2}(\Omega);\bbL^{2}(\Omega))}\|\rho(\cdot)\|_{\bbL^{\infty}(\Omega)}\big)^{j-1} \notag\\
    &= \mathcal{C}\,\|\rho(\cdot)\|_{\bbL^{\infty}(\Omega)}^{2}\,\omega^{2}\,
       \|\calN^{\calP}\|_{\mathcal{L}(\bbL^{2}(\Omega);\bbL^{2}(\Omega))}
       \sum_{j\ge 0}\eta^{\,j},
\end{align}
where
\begin{align}\notag
    \eta:=\omega^{2}\|\rho(\cdot)\|_{\bbL^{\infty}(\Omega)}\|\calN^{\calP}\|_{\mathcal{L}(\bbL^{2}(\Omega);\bbL^{2}(\Omega))}.
\end{align}

\noindent
In particular, if
\begin{equation}\label{eq:eta-small}
    \eta<1, 
\end{equation}
then the geometric series on the right-hand side converges and we obtain
\begin{align}
    \Big\|\sum_{j\ge 2}\calK_{j}\otimes(\rho(\cdot))\Big\|_{\bbH^{1/2}(\partial\Omega)}
    \le \mathcal{C}\,\|\rho(\cdot)\|_{\bbL^{\infty}(\Omega)}^{2}\,\omega^{2}\,\|\calN^{\calP}\|_{\mathcal{L}(\bbL^{2}(\Omega);\bbL^{2}(\Omega))}
    \sum_{j\ge 0}\eta^{\,j}.
\end{align}
By \eqref{eq:NP norm} we have the operator bound, so for $\calP$ sufficiently large—and recalling that $\omega^{2}\|\rho(\cdot)\|_{\bbL^{\infty}(\Omega)}$ is uniformly bounded—the smallness condition \eqref{eq:eta-small} is indeed fulfilled. Moreover, from \eqref{eq:sum calKj} we infer
\begin{align}\label{eq:sum kj rho=O(Qf)}
    \Big\|&\sum_{j\ge 2}\calK_{j}\otimes(\rho(\cdot))\Big\|_{\bbH^{1/2}(\partial\Omega)}
    = \mathcal{O}\!\big(\mathcal{C}\|\calN^{\calP}\|_{\mathcal{L}(\bbL^{2}(\Omega);\bbL^{2}(\Omega))}\big)   \\
    &\stackrel{\eqref{eq:calC def}}{=} \mathcal{O}\!\Big(\|\calN^{\calP}\|_{\mathcal{L}(\bbL^{2}(\Omega);\bbL^{2}(\Omega))} \|\bQ^{\bff}\|_{\bbL^{2}(\Omega)} 
        \|\gamma\calN^{\calP}\|_{\mathcal{L}(\bbL^{2}(\Omega);\bbH^{1/2}(\partial\Omega))}\Big)  
      \stackrel{\text{Lemma \ref{lem:NP norm lem}}}{=}  
     \mathcal{O}\!\Big(\|\bQ^{\bff}\|_{\bbL^{2}(\Omega)} \calP^{-3}\Big). \notag
\end{align}
The next lemma will translate this bound into an estimate that depends explicitly on the datum $\bff(\cdot)$ and the parameter $\calP$.

\begin{lem}\label{lem:Qf L2 bound elastic} 
Let $\bQ^{\bff}$ be the solution to the Neumann problem \eqref{eq:Qf def}. Then
\begin{align}\label{eq:Qf L2 norm in lemma}
\|\bQ^{\bff}\|_{\bbL^{2}(\Omega)}=\mathcal{O}\!\left(\|\bff\|_{\bbH^{-1/2}(\partial\Omega)}\,\calP^{-1}\right).
\end{align}
\end{lem}

\begin{proof}
We recall that $\Gamma_{\calP}$ is the fundamental solution of \eqref{eq:Gamma P def}. We introduce the single-layer potential at the parameter $\calP$:
\begin{equation}\label{eq:SLP def elastic}
\mathbf{SL}^{\calP}(\bff)(\bx)
:= \int_{\partial\Omega} \Gamma_{\calP}(\bx,\by)\,\bff(\by)\, \mathrm{d}\sigma(\by),
\qquad \bx\in\Omega.
\end{equation}
By the very choice of $\Gamma_{\calP}$, the field $\mathbf{SL}^{\calP}(\bff)$ satisfies
\[
(\mathcal{L}_{\lambda,\mu}-\calP^{2})\,\mathbf{SL}^{\calP}(\bff)=\mathbf{0}\quad\text{in }\Omega.
\]
{ 
As shown in \cite[(1.81)--(1.83)]{ABG15} and Subsection~\ref{subsec:partial SL=f}, we have 
}
\begin{align}\label{eq:partial SL=f}
    \partial_{\nu}\mathbf{SL}^{\calP}(\bff)=\bff\quad\text{on }\partial\Omega.
\end{align}
Hence $\mathbf{SL}^{\calP}(\bff)$ solves \eqref{eq:Qf def}; by uniqueness we have the representation
\begin{equation}\label{eq:Qf as SLP}
\bQ^{\bff}(\bx)=\mathbf{SL}^{\calP}(\bff)(\bx), \qquad \bx\in\Omega.
\end{equation}

\noindent
Next, recall that the Newtonian operator $\calN^{\calP}$ was defined in \eqref{eq:Newton potential in Omega} by
\[
\calN^{\calP}(\bg)(\bx)=\int_{\Omega} \Gamma_{\calP}(\bx,\by)\,\bg(\by)\,\mathrm{d}\by,
\qquad \bg\in\bbL^{2}(\Omega), \quad  \bx\in\Omega.
\]
The symmetry property of the Green tensor,
\[
\Gamma_{\calP}(\bx,\by) = \Gamma_{\calP}(\by,\bx)^{\top}, \qquad \bx,\by\in\Omega,
\]
implies the usual duality identity: for every $\bphi\in\bbL^{2}(\Omega)$ and every boundary density $\bff\in\bbH^{-1/2}(\partial\Omega)$,
\begin{equation}\label{eq:duality identity}
\int_{\Omega} \bphi(\bx)\cdot \mathbf{SL}^{\calP}(\bff)(\bx)\,\mathrm{d}\bx
= \int_{\partial\Omega} \bff(\by)\cdot \gamma\calN^{\calP}(\bphi)(\by)\,\mathrm{d}\sigma(\by).
\end{equation}
We now apply \eqref{eq:duality identity} with $\bphi=\bQ^{\bff}$. Using \eqref{eq:Qf as SLP}, the left-hand side becomes
\[
\int_{\Omega} \bQ^{\bff}(\bx)\cdot \bQ^{\bff}(\bx)\,\mathrm{d}\bx
= \|\bQ^{\bff}\|_{\bbL^{2}(\Omega)}^{2}.
\]
Therefore,
\begin{equation}\label{eq:L2 identity Qf}
\|\bQ^{\bff}\|_{\bbL^{2}(\Omega)}^{2}
= \int_{\partial\Omega} \bff(\by)\cdot \gamma\calN^{\calP}(\bQ^{\bff})(\by)\,\mathrm{d}\sigma(\by)
= \big\langle \bff,\ \gamma\calN^{\calP}(\bQ^{\bff}) \big\rangle_{\bbH^{-1/2}(\partial\Omega)\times\bbH^{1/2}(\partial\Omega)}.
\end{equation}
By Cauchy–Schwarz in the boundary duality,
\[
\|\bQ^{\bff}\|_{\bbL^{2}(\Omega)}^{2}
\le \|\bff\|_{\bbH^{-1/2}(\partial\Omega)}\,
\|\gamma\calN^{\calP}(\bQ^{\bff})\|_{\bbH^{1/2}(\partial\Omega)}.
\]
Using the operator norm of the trace of the Newtonian potential, we obtain
\begin{align}\label{eq:Qf L2 norm}
    \|\bQ^{\bff}\|_{\bbL^{2}(\Omega)}^{2}
\le \|\bff\|_{\bbH^{-1/2}(\partial\Omega)}\,
\|\gamma\calN^{\calP}\|_{\mathcal{L}(\bbL^{2}(\Omega);\bbH^{1/2}(\partial\Omega))}\,
\|\bQ^{\bff}\|_{\bbL^{2}(\Omega)}.
\end{align}
Combining \eqref{eq:gamma NP norm} and \eqref{eq:Qf L2 norm}, we conclude that
\[
\|\bQ^{\bff}\|_{\bbL^{2}(\Omega)} = \mathcal{O}\!\left(\|\bff\|_{\bbH^{-1/2}(\partial\Omega)}\,\calP^{-1}\right),
\]
which completes the proof of Lemma \ref{lem:Qf L2 bound elastic}.
\end{proof}

By making use of \eqref{eq:Qf L2 norm in lemma}, the estimate \eqref{eq:sum kj rho=O(Qf)} takes the form
\begin{align}
\left\|\sum_{j\geq 2} \calK_j \otimes \rho(\cdot)\right\|_{\bbH^{1/2}(\partial\Omega)}
= \mathcal{O}\left(\|\bff\|_{\bbH^{-1/2}(\partial\Omega)}\,\calP^{-4} \right).
\end{align}
Thus, from \eqref{eq:qf-Qf=}, we obtain
\begin{align}\label{eq:qf-Qf}
\bq^{\bff}-\bQ^{\bff}
= \omega^2\,\gamma \left(\mathcal{W}^{\bQ^\bff} \right)(\bx)
  + \mathcal{O}\!\left(\|\bff\|_{\bbH^{-1/2}(\partial\Omega)}\,\calP^{-4}\right), 
\quad \bx \in \partial\Omega,  
\end{align}
where $\mathcal{W}^{\bQ^\bff}:= \calN^\calP(\rho(\cdot)\,\bQ^\bff)$ is the function satisfying \eqref{eq:WQf def}. On the boundary $\partial\Omega$, we have 
\[
\bq^{\bff} = \Lambda_\calP(\partial_\nu \bq^{\bff})
\stackrel{ \eqref{eq:system qg}}{=} \Lambda_\calP(\bff),
\]
since $\partial_\nu \bq^{\bff}=\bff$ on $\partial\Omega$. Substituting this into \eqref{eq:qf-Qf} yields \eqref{eq:Lambda P f-Qf in thm}, which completes the proof of Theorem \ref{thm:linearisation}.

\section{Proof of Theorem \ref{thm:CGO}  }\label{sec:Theorem 3 CGO}

The aim of this section is to show how to combine the linearized N--D map, available on the boundary $\partial\Omega$, with  CGO  solutions in order to extract the Fourier modes of the unknown density $\rho(\cdot)$. In particular, once these Fourier coefficients have been recovered, the function $\rho(\cdot)$ inside $\Omega$ can be expanded in a Fourier series. From the analysis in the previous section, we recall that measuring $\bq^\bff(\cdot) - \bQ^\bff(\cdot)$ on $\partial\Omega$ is equivalent, up to a small error, to measuring $\mathcal{W}^{\bQ^\bff}(\cdot)$ on $\partial\Omega$.

We now introduce $\bQ^\bg(\cdot)$ as the solution of
\begin{align}\label{eq:Qg def}
\begin{cases}
(\mathcal{L}_{\lambda,\mu} - \mathcal{P}^2  )\,\bQ^\bg = \mathbf{0} & \text{in } \Omega,\\[0.2em]
\partial_\nu \bQ^\bg = \bg & \text{on } \partial\Omega.
\end{cases}
\end{align}

\noindent
Multiplying the first equation in \eqref{eq:WQf def} by $\bQ^\bg(\cdot)$, the solution of \eqref{eq:Qg def}, and integrating over $\Omega$, we obtain
\begin{align}
    \int_{\Omega} \sigma(\calW^{\bQ^\bff}) : \nabla^{s}(\bQ^{\bg}) \, \rmd \bx 
    + \calP^{2} \int_{\Omega} \calW^{\bQ^\bff} \cdot\bQ^{\bg}  \, \rmd \bx 
    = \int_{\Omega} \rho(\bx)\,\bQ^{\bff} \cdot \bQ^{\bg} \, \rmd \bx.
\end{align}

\noindent
On the other hand, multiplying \eqref{eq:Qg def} by $\calW^{\bQ^\bff}$, the solution of \eqref{eq:WQf def}, and integrating over $\Omega$ yields
\begin{align}
  \int_{\Omega} \sigma(\calW^{\bQ^\bff}) : \nabla^{s}(\bQ^{\bg}) \, \rmd \bx 
  + \calP^{2} \int_{\Omega} \calW^{\bQ^\bff} \cdot\bQ^{\bg}  \, \rmd \bx
  = \langle \mathcal{W}^{\bQ^\bff} , \bg \rangle_{\bbH^{1/2}(\partial\Omega) \times \bbH^{-1/2}(\partial\Omega)}.
\end{align}
Subtracting the two identities above, we arrive at
\begin{align}\label{eq:Wqf g=rho}
\langle  \mathcal{W}^{\bQ^\bff} , \bg \rangle_{\bbH^{1/2}(\partial\Omega) \times \bbH^{-1/2}(\partial\Omega)}
= \langle \rho(\cdot)\,\bQ^\bff , \bQ^\bg \rangle_{\bbL^2(\Omega)},
\end{align}
for all $(\bff,\bg) \in \bbH^{-1/2}(\partial\Omega) \times \bbH^{-1/2}(\partial\Omega)$. Since $\mathcal{W}^{\bQ^\bff}$ is accessible from boundary measurements on $\partial\Omega$ and $\bg$ is prescribed boundary data, the left-hand side is determined by the available boundary measurements. Our objective is to recover $\rho(\cdot)$ in $\Omega$ from this identity.

Throughout this section, we fix any nonzero vector $\boldsymbol{\xi} \in \mathbb{R}^{3}$ and choose an orthonormal basis $\{e_{1},e_{2},e_{3}\}$ of $\mathbb{R}^{3}$ such that
\[
e_{1} := \frac{\boldsymbol{\xi}}{|\boldsymbol{\xi}|}.
\]
We define $\bQ^\bff(\cdot)$ by
\begin{align}\label{eq:Qf def in Section3}
\bQ^\bff(\bx) := e^{\rmi \boldsymbol{\zeta}_1 \cdot \bx} \big( \boldsymbol{\eta}_1 + \mathbf{F}_1(\bx) \big), 
\qquad \bx \in \mathbb{R}^3,
\end{align}
where $\boldsymbol{\zeta}_1$ and $\boldsymbol{\eta}_1$ are given by  
\begin{align}
    \boldsymbol{\zeta}_{1} 
    = -\frac{|\boldsymbol{\xi}|}{2}\, e_{1} 
      + \rmi \sqrt{ t^2 - k_s^2 + \frac{|\boldsymbol{\xi}|^{2}}{4}}\, e_{2}
      + t\, e_3,  
    \qquad  
    \boldsymbol{\eta}_1   =   e_1 + \frac{|\boldsymbol{\xi}|}{2 t}\, e_2,
\end{align}
and where $\mathbf{F}_1(\cdot)$ satisfies
\begin{align}\label{eq:F1 satisfies equation}
&(\lambda+\mu)\Big[
     \nabla(\nabla\cdot\mathbf{F}_1)
     + \mathrm{i}\,\boldsymbol{\zeta}_1(\nabla\cdot\mathbf{F}_1)
     + \mathrm{i}\,\nabla( \boldsymbol{\zeta}_1\cdot\mathbf{F}_1)
     -  \boldsymbol{\zeta}_1( \boldsymbol{\zeta}_1\cdot\mathbf{F}_1)
   \Big] \notag\\
&\quad
+ \mu\Big(
     \Delta\mathbf{F}_1
     + 2\mathrm{i}\,(\boldsymbol{\zeta}_1\cdot\nabla)\mathbf{F}_1
     - k_s^2 \mathbf{F}_1\Big)   - \mathcal{P}^2\mathbf{F}_1 
  = \big(\mu k_s^2  + \mathcal{P}^2\big)\,\boldsymbol{\eta}_1,
\qquad \bx\in\mathbb{R}^3.
\end{align}

\noindent 
Note that the right-hand side depends on $\boldsymbol{\eta}_1$ and $\calP$, and therefore the correction term $\mathbf{F}_1(\cdot)$ also depends on $\boldsymbol{\eta}_1$ and $\calP$.

By \cite[Theorem 3.1]{H02}, the parameter $t$ must satisfy
\begin{align}\label{eq:t satisfies condition}
t>t_0=k_s.
\end{align}

\noindent 
There exists a constant $C_0$, depending only on $\Omega$, such that under the assumption $\calP \gg 1$, equation \eqref{eq:F1 satisfies equation} admits a solution $\mathbf{F}_{1,\boldsymbol{\zeta}_1,\boldsymbol{\eta}_1}(\cdot) \in \bbH^1(\Omega)$ satisfying
\begin{subequations}
\label{eq:F1_estimate}
\begin{align}
\| \mathbf{F}_{1,\boldsymbol{\zeta}_1,\boldsymbol{\eta}_1}  \|_{\bbL^2(\Omega)} 
&\le \frac{C_0\, \mathcal{P}^2\, |\boldsymbol{\eta}_1|}{|\Im(\boldsymbol{\zeta}_1)|}\,|\Omega|^{1/2}, \label{eq:F1 estimation}  \\   
\| \nabla \mathbf{F}_{1,\boldsymbol{\zeta}_1,\boldsymbol{\eta}_1}  \|_{\bbL^2(\Omega)} 
&\le C_0\, \mathcal{P}^2\, |\boldsymbol{\eta}_1|\,|\Omega|^{1/2}. \label{eq:nabla F1 estimation}  
\end{align}
\end{subequations}

In a similar fashion, we prescribe $\bQ^\bg(\cdot)$ by
\begin{align}\label{eq:Qg CGO def}
\bQ^\bg(\bx) :=  e^{\rmi \boldsymbol{\zeta}_2 \cdot \bx} \big(\boldsymbol{\eta}_2 + \mathbf{F}_2(\bx)\big), 
\qquad \bx \in \mathbb{R}^3,
\end{align}
where
\begin{align}
 \boldsymbol{\zeta}_{2} 
 &= -\frac{|\boldsymbol{\xi}|}{2}\, e_{1} 
    - \rmi  \sqrt{ t^2 - k_s^2 + \frac{|\boldsymbol{\xi}|^{2}}{4}}\, e_{2} 
    - t\, e_3,  \\
 \boldsymbol{\eta}_2   &= e_1 - \frac{|\boldsymbol{\xi}|}{2 t}\, e_2,
\end{align}
and $\mathbf{F}_2(\cdot)$ solves
\begin{align}\label{eq:F2 def}
&(\lambda+\mu)\Big[
     \nabla(\nabla\cdot\mathbf{F}_2)
     + \mathrm{i}\,\boldsymbol{\zeta}_2(\nabla\cdot\mathbf{F}_2)
     + \mathrm{i}\,\nabla( \boldsymbol{\zeta}_2\cdot\mathbf{F}_2)
     -  \boldsymbol{\zeta}_2( \boldsymbol{\zeta}_2\cdot\mathbf{F}_2)
   \Big] \notag\\
&\quad
+ \mu\Big(
     \Delta\mathbf{F}_2
     + 2\mathrm{i}\,(\boldsymbol{\zeta}_2 \cdot\nabla)\mathbf{F}_2
     - k_s^2 \mathbf{F}_2 \Big)   - \mathcal{P}^2\mathbf{F}_2
  = \big(\mu k_s^2  + \mathcal{P}^2\big)\,\boldsymbol{\eta}_2,
\qquad \bx\in\mathbb{R}^3.
\end{align}

\noindent
Since the right-hand side depends on $\calP$ and $\boldsymbol{\eta}_2$, the function $\mathbf{F}_2(\cdot)$ also depends on $\calP$ and $\boldsymbol{\eta}_2$. Again, by \cite[Theorem 3.1]{H02}, equation \eqref{eq:F2 def} admits a solution $\mathbf{F}_{2,\boldsymbol{\zeta}_2,\boldsymbol{\eta}_2}  \in \bbH^1(\Omega)$ satisfying
\begin{align}\label{eq:F2 estimate}
\| \mathbf{F}_{2,\boldsymbol{\zeta}_2,\boldsymbol{\eta}_2} \|_{\bbL^2(\Omega)} 
\le \frac{C_0\, \mathcal{P}^2\, |\boldsymbol{\eta}_2|}{|\Im(\boldsymbol{\zeta}_2)|}\,|\Omega|^{1/2},
\quad 
\| \nabla \mathbf{F}_{2,\boldsymbol{\zeta}_2,\boldsymbol{\eta}_2} \|_{\bbL^2(\Omega)} 
\le C_0\, \mathcal{P}^2\, |\boldsymbol{\eta}_2|\,|\Omega|^{1/2}.
\end{align}

\noindent
Using that $\mathcal{P} \gg 1$ and taking $t$ large enough so that \eqref{eq:t satisfies condition} is satisfied, we can simplify the estimates on the $\bbL^2(\Omega)$-norms of $\mathbf{F}_{1,\boldsymbol{\zeta}_1,\boldsymbol{\eta}_1}$ and $\mathbf{F}_{2,\boldsymbol{\zeta}_2,\boldsymbol{\eta}_2}$ to
\begin{align}\label{eq:F1 and F2 L2 norm estimate}
\| \mathbf{F}_{1,\boldsymbol{\zeta}_1,\boldsymbol{\eta}_1} \|_{\bbL^2(\Omega)} = \mathcal{O}\!\left( \frac{\mathcal{P}^2}{t} \right),
\qquad
\|\mathbf{F}_{2,\boldsymbol{\zeta}_2,\boldsymbol{\eta}_2}  \|_{\bbL^2(\Omega)} = \mathcal{O}\!\left( \frac{\mathcal{P}^2}{t} \right).
\end{align}

\noindent
Taking the product of $\bQ^\bff(\cdot)$ from \eqref{eq:Qf def in Section3} and $\bQ^\bg(\cdot)$ from \eqref{eq:Qg CGO def}, we obtain
\begin{align}
(\bQ^\bff \cdot \bQ^\bg)(\bx)
= e^{-\rmi \boldsymbol{\xi}\cdot\bx} \Big( \boldsymbol{\eta}_1\cdot\boldsymbol{\eta}_2 
+ \boldsymbol{\eta}_1\cdot \mathbf{F}_2(\bx) 
+ \mathbf{F}_1(\bx)\cdot \boldsymbol{\eta}_2 
+ \mathbf{F}_1(\bx)\cdot \mathbf{F}_2(\bx) \Big).
\end{align}

\noindent
We choose the solutions in such a way that $(\bQ^\bff \cdot \bQ^\bg)(\cdot)$ approximates $e^{-\rmi \boldsymbol{\xi}\cdot\bx}$, since the family $\{ e^{-\rmi \boldsymbol{\xi}\cdot\bx}  \}$ forms a dense set in $\bbL^1(\Omega)$ (see \cite[Theorem 1.1]{I91}). Returning to \eqref{eq:Wqf g=rho}, we have
\begin{align}
\langle \mathcal{W}^{\bQ^\bff} , \bg \rangle_{\bbH^{1/2}(\partial\Omega) \times \bbH^{-1/2}(\partial\Omega)}
&= \int_\Omega \rho(\bx)\, \bQ^\bff(\bx)\cdot\bQ^\bg(\bx)\, \rmd \bx  \notag\\
&= \int_\Omega \rho(\bx)\, e^{-\rmi \boldsymbol{\xi} \cdot  \bx } \,\rmd \bx + R(\boldsymbol{\xi},t,\mathcal{P}), \notag
\end{align}
where
\begin{align}
\boldsymbol{Rest}(\boldsymbol{\xi},t,\mathcal{P})
&:=-\int_\Omega e^{-\rmi \boldsymbol{\xi} \cdot  \bx }\, \rho(\bx)\,\frac{|\boldsymbol{\xi}|^{2}}{4 t^2} \,\rmd \bx
 + \int_\Omega e^{-\rmi \boldsymbol{\xi} \cdot  \bx }\, \rho(\bx)\,\boldsymbol{\eta}_1\cdot \mathbf{F}_{2,\boldsymbol{\zeta}_2,\boldsymbol{\eta}_2}(\bx)\, \rmd \bx \notag  \\
&\quad + \int_\Omega e^{-\rmi \boldsymbol{\xi} \cdot  \bx }\,\rho(\bx)\,\mathbf{F}_{1,\boldsymbol{\zeta}_1,\boldsymbol{\eta}_1} (\bx)\cdot \boldsymbol{\eta}_2\, \rmd \bx         
 + \int_\Omega e^{-\rmi \boldsymbol{\xi} \cdot  \bx }\, \rho(\bx)\,\mathbf{F}_{1,\boldsymbol{\zeta}_1,\boldsymbol{\eta}_1} (\bx)\cdot \mathbf{F}_{2,\boldsymbol{\zeta}_2,\boldsymbol{\eta}_2}(\bx)\,\rmd \bx, \notag
\end{align}
and this remainder term can be estimated by
\begin{align}
|\boldsymbol{Rest}(\boldsymbol{\xi},t,\mathcal{P})|
&\le \|\rho(\cdot)\|_{\bbL^\infty(\Omega)}\, \frac{|\boldsymbol{\xi}|^{2}}{4 t^2}\,|\Omega|^{1/2}
+ \|\rho(\cdot)\|_{\bbL^\infty(\Omega)}\,|\boldsymbol{\eta}_1|\,\|\mathbf{F}_{2,\boldsymbol{\zeta}_2,\boldsymbol{\eta}_2}\|_{\bbL^2(\Omega)} \,|\Omega|^{1/2} \notag \\ 
&\quad + \|\rho(\cdot)\|_{\bbL^\infty(\Omega)}\,|\boldsymbol{\eta}_2|\,
\| \mathbf{F}_{1,\boldsymbol{\zeta}_1,\boldsymbol{\eta}_1} \|_{\bbL^2(\Omega)}
+ \|\rho(\cdot)\|_{\bbL^\infty(\Omega)}\,\| \mathbf{F}_{1,\boldsymbol{\zeta}_1,\boldsymbol{\eta}_1} \|_{\bbL^2(\Omega)}\,\| \mathbf{F}_{2,\boldsymbol{\zeta}_2,\boldsymbol{\eta}_2}\|_{\bbL^2(\Omega)}. \notag
\end{align}

\noindent
Using \eqref{eq:F1_estimate}, \eqref{eq:F2 estimate} and \eqref{eq:F1 and F2 L2 norm estimate}, together with the fact that $|\boldsymbol{\eta}_1|$ and $|\boldsymbol{\eta}_2|$ are uniformly bounded, this reduces to
\begin{align}\label{eq:Rest xi=}
|\boldsymbol{Rest}(\boldsymbol{\xi},t,\mathcal{P})|
\lesssim  \frac{\mathcal{P}^2}{\sqrt{t^2-k_s^2+\frac{|\boldsymbol{\xi}|^2}{4}}}
=\mathcal{O}\!\left( \frac{\mathcal{P}^2}{t} \right).
\end{align}

By construction (see \eqref{eq:Qf def in Section3}), $\bQ^\bff(\cdot)$ depends on $\boldsymbol{\xi}$, and hence $\mathcal{W}^{\bQ^\bff}(\cdot)$ inherits this dependence. We therefore write 
\begin{align}\label{eq:Wqf g-Rest}
\langle \mathcal{W}^{\bQ^\bff}_{\boldsymbol{\xi}}, \bg \rangle_{\bbH^{1/2}(\partial\Omega) \times \bbH^{-1/2}(\partial\Omega)}
- \boldsymbol{Rest}(\boldsymbol{\xi},t,\mathcal{P}) 
= \int_\Omega \rho(\bx)\,e^{-\rmi \boldsymbol{\xi}\cdot\bx}\,\rmd \bx
=(2\pi)^3 \mathcal{F}(\rho\, \chi_\Omega)(\boldsymbol{\xi}),
\end{align}
where $\mathcal{F}(\cdot)$ denotes the Fourier transform. So we have
\begin{align}
\rho(\bx) = \sum_{\boldsymbol{\xi}\in \mathbb{Z}^3} \mathcal{F}(\rho\,\chi_\Omega)(\boldsymbol{\xi})\, e^{\rmi \boldsymbol{\xi} \cdot \bx}, \quad \bx \in \Omega,
\end{align}
in $\bbL^2(\Omega)$. Combining this representation with \eqref{eq:Wqf g-Rest}, we obtain
\begin{align}\label{eq:rho x=}
\rho(\bx)
= (2\pi)^{-3} \sum_{\boldsymbol{\xi}\in \mathbb{Z}^3}
\langle \mathcal{W}^{\bQ^\bff}_{\boldsymbol{\xi}}, \bg \rangle_{\bbH^{1/2}(\partial\Omega) \times \bbH^{-1/2}(\partial\Omega)}
\, e^{\rmi \boldsymbol{\xi} \cdot \bx}  
+ \mathcal{R}(\bx;t,\calP) ,
\end{align}
in the $\bbL^2(\Omega)$ sense, where the error function $\mathcal{R}(\cdot;t,\calP)$ is defined by
\begin{align}
\mathcal{R}(\bx;t,\calP)
:= - (2\pi)^{-3}  \sum_{\boldsymbol{\xi}\in \mathbb{Z}^3}  \boldsymbol{Rest}(\boldsymbol{\xi},t,\mathcal{P})\, e^{\rmi \boldsymbol{\xi}\cdot \bx} ,
\quad \bx \in \Omega.
\end{align}

\noindent
We next estimate the $\bbL^2(\Omega)$-norm of $\mathcal{R}(\cdot;t,\calP)$. We have
\begin{align}
\|\mathcal{R}(\cdot;t,\calP)\|_{\bbL^2(\Omega)}
\lesssim  \left(  \sum_{\boldsymbol{\xi}\in \mathbb{Z}^3}    |\boldsymbol{Rest}(\boldsymbol{\xi},t,\mathcal{P})|^2\right)^{\frac{1}{2}}
\stackrel{\eqref{eq:Rest xi=} }{\lesssim}  \left( \sum_{\boldsymbol{\xi}\in \mathbb{Z}^3}  \frac{\mathcal{P}^4}{{t^2-k_s^2+\frac{|\boldsymbol{\xi}|^2}{4}}}      \right)^{\frac{1}{2}}.
\end{align}
 
\noindent
We now choose $t$ such that
\[
t = \sqrt{\mathcal{P}^{4+2\iota}+k_s^2}
\]
for some $\iota \in \mathbb{R}^+$. Consequently,
\begin{align}
\|\mathcal{R}(\cdot;t,\calP)\|_{\bbL^2(\Omega)}
&\lesssim  \left( \sum_{\boldsymbol{\xi}\in \mathbb{Z}^3}  \frac{\mathcal{P}^4}{{t^2-k_s^2+\frac{|\boldsymbol{\xi}|^2}{4}}}      \right)^{\frac{1}{2}}   
 =  \mathcal{P}^{-\iota} \sum_{\boldsymbol{\xi}\in \mathbb{Z}^3}   \frac{1}{{1+\frac{|\boldsymbol{\xi}|^2}{4\mathcal{P}^{4+2\iota}}}}.
\end{align}
Using the convergence of this series, we finally obtain 
\begin{align}
   \|\mathcal{R}(\cdot;t,\calP)\|_{\bbL^2(\Omega)} =\mathcal{O}(\mathcal{P}^{-\iota}).
\end{align}

\noindent
Therefore, \eqref{eq:rho x=} can be rewritten as
\begin{align}
\rho(\bx)
= (2\pi)^{-3} \sum_{\boldsymbol{\xi}\in \mathbb{Z}^3} 
\langle \mathcal{W}^{\bQ^\bff}_{\boldsymbol{\xi}}, \bg \rangle_{\bbH^{1/2}(\partial\Omega) \times \bbH^{-1/2}(\partial\Omega)}\,
e^{\rmi \boldsymbol{\xi} \cdot \bx}  
+ \mathcal{O}(\mathcal{P}^{-\iota}),
\end{align}
with convergence in $\bbL^2(\Omega)$. This completes the proof of Theorem \ref{thm:CGO}.

\section{Proof of Theorem \ref{thm:N--D}} \label{sec:Theorem 1}

The aim of this section is to prove Theorem \ref{thm:N--D}, that is, to justify the effective N--D map $\Lambda_{\mathcal{P}}$ associated with the homogenized background and to obtain a quantitative estimate of the difference $\Lambda_{D}-\Lambda_{\mathcal{P}}$. Our starting point is an energy identity for
\[
\mathsf{J}
  =\big\langle (\Lambda_{D}-\Lambda_{\mathcal{P}})\bff,\bg\big\rangle_{\mathbb{H}^{1/2}(\partial \Omega) \times \mathbb{H}^{-1/2}(\partial \Omega)}.
\]

We first rewrite this difference in terms of volume integrals over the hard inclusions and derive a suitable representation formula in Subsection \ref{subsec:4.1}. In Subsection \ref{subsec:4.2} we introduce the limiting Lippmann--Schwinger equation satisfied by the effective field and prove existence, uniqueness, and a priori bounds for its solution. Subsections \ref{subsec:4.3} and \ref{subsec:4.4} are then devoted to comparing the discrete algebraic system with its continuous counterpart and to estimating the associated remainder terms; this analysis yields the operator-norm bound for $\Lambda_{D}-\Lambda_{\mathcal{P}}$ stated in Theorem \ref{thm:N--D}. Since the proof of Theorem \ref{thm:N--D} is lengthy and technically involved, we summarize its main steps in the flowchart shown in Figure \ref{fig:J_derivation_flowchart}.

Recalling \eqref{eq:Lambda D-Lambda P in thm}, we begin by analyzing the term $\mathbb{I}_1$. We have
\begin{align} 
\big\langle ( \Lambda_{D} - \Lambda_{\mathcal{P}} )\mathbf{f}, \mathbf{g} \big\rangle_{\mathbb{H}^{1/2}(\partial \Omega) \times \mathbb{H}^{-1/2}(\partial \Omega)}  
&= \omega^{2} \big\langle \big( \rho_{1} - \rho(\mathbf{x}) \big)\,\mathbf{v}^{\mathbf{g}}, \mathbf{u}^{\mathbf{f}} \big\rangle_{\mathbb{L}^{2}(D)}  
+ \mathcal{P}^{2} \big\langle \mathbf{q}^{\mathbf{g}}, \mathbf{u}^{\mathbf{f}} \big\rangle_{\mathbb{L}^{2}(\Omega)} \notag \\
&=: \mathbb{I}_1 + \mathcal{P}^{2} \big\langle \mathbf{q}^{\mathbf{g}}, \mathbf{u}^{\mathbf{f}} \big\rangle_{\mathbb{L}^{2}(\Omega)}. \notag
\end{align}

\begin{figure}[htbp]
    \centering
    \scalebox{0.9}{
    \begin{tikzpicture}[
        box/.style={rectangle, draw, align=center, minimum width=1.2cm, minimum height=0.8cm}, 
        bigbox/.style={rectangle, draw, align=center, minimum width=6cm, minimum height=1.5cm}, 
        arrow/.style={->, thick, >=stealth} 
    ]
        \node[box] (A) at (-2, 5) 
        {\textbf{Subsection \ref{subsec:4.1}:} \
        $\mathsf{J}
      =\mathbb{I}_1+\mathcal{P}^{2} \big\langle \mathbf{q}^{\mathbf{g}}, \mathbf{u}^{\mathbf{f}} \big\rangle_{\mathbb{L}^{2}(\Omega)}$, \quad $\mathbb{I}_1$ is defined in \eqref{eq:I1 def in subsection 4.1}.\\
      Expanding \(\mathbf{u}^{\bf f}\) by Taylor’s theorem around the centers \(\mathbf{z}_j\), we obtain \\
     $
    \mathbb{I}_1
    = \omega^2 \rho_1 \sum_{j=1}^M \sum_{i=1}^3 u_i^{\bf f}(\mathbf{z}_j)\, \int_{D_j} v_{i,j}^{\bf g}(\mathbf{x})\,\rmd \mathbf{x}
    + \mathbb{J}_1 
    $, \quad  $\mathbb{J}_1 $ is defined  \eqref{eq:J1 def}.
      };

        \node[box] (F) at (-5,0) 
        { \textbf{Subsection \ref{subsec:4.2}:} \\
        $\mathbf{v}^{\bf g}(\bx)
        -\omega^2 \int_{D} \Gamma(\bx,\by)\,\big(\rho_1-\rho(\by)\big)\,\mathbf{v}^{\bf g}(\by)\,\rmd\by$\\
        $=\mathbf{S}(\bx), \qquad \bx\in D.$\\
        Lemma \ref{lem:R(x,y)}$\Rightarrow$ $\Gamma(\mathbf{x},\mathbf{y})=\Gamma^0(\mathbf{x},\mathbf{y})+\mathcal{R}(\mathbf{x},\mathbf{y})$,\\
        and $ \big|\mathcal{R}(\mathbf{x},\mathbf{y})\big| \lesssim \mathcal{O}\big(a^{\frac{h-1}{3}}\big).$
        };
        
        \node[box] (F1) at (-5,-3.5) 
        {
          $\boldsymbol{\beta}_m\int_{D_m}\mathbf{v}^{\mathbf{g}}(\mathbf{x})\,\rmd\mathbf{x}
          -\boldsymbol{\alpha}_m\sum_{\substack{j=1\\ j\neq m}}^{M}\Gamma(\mathbf{z}_m,\mathbf{z}_j)\int_{D_j}\mathbf{v}^{\mathbf{g}}(\mathbf{x})\,\rmd\mathbf{x}$\\
          $=\frac{1}{\omega^{2}\rho_{1}}\;\boldsymbol{\alpha}_m\,\mathbf{S}(\mathbf{z}_m)+\mathrm{Error}_m ,$
        };

        \node[box] (F2) at (-7.5,-6.5) 
        { $\boldsymbol{\alpha}_m$:\eqref{eq:alpha def},  $\boldsymbol{\beta}_m$:\eqref{eq:beta=1-int};\\  $\mathrm{Error}_m$: \eqref{eq:Errorm def}.
        };

        \node[box] (F3) at (-2.5,-6.5) 
        { That the algebraic \\
        system is invertible \\
        was established in\\
        Lemma \ref{lem:The algebraic system is invertible}.
        };

        \node[box] (B) at (-5,-10) 
        {
        $\mathbf{Y}_m+\sum_{\substack{j=1\\ j\neq m}}^{M}\Gamma(\mathbf{z}_m,\mathbf{z}_j)\,\mathcal{P}^{2}\,a^{\,1-h}\,\frac{1}{\boldsymbol{\beta}_j}\,\mathbf{Y}_j $ \\
        $
        = \mathbf{S}(\mathbf{z}_m)+\frac{\omega^{2}\rho_{1}}{\boldsymbol{\alpha}}\,\mathrm{Error}_m$. 
        };

        \node[box] (B1) at (-7,-12.5) 
        {
        $\mathbf{Y}_m:=\frac{\boldsymbol{\beta}_m\,\omega^{2}\rho_{1}}{\boldsymbol{\alpha}}\int_{D_m}\mathbf{v}^{\mathbf{g}}(\mathbf{x})\,\rmd\mathbf{x}$, \\
        Lemma \ref{lem:alpha=-P2} $\Rightarrow$
        $ 
        \boldsymbol{\alpha}:=-\,\mathcal{P}^{2}\,a^{\,1-h},
        $ 
        };

        \node[box] (C) at (3.8, 0) 
        { \textbf{Subsection \ref{subsec:4.3}:} \\
        $\widetilde{\mathbf{Y}}_m
        + \sum_{\substack{j=1\\ j\neq m}}^{M} \Gamma(\mathbf{z}_m,\mathbf{z}_j)\,$\\ $\mathcal{P}^{2}a^{\,1-h}\,\frac{1}{\boldsymbol{\beta}_j}\,\widetilde{\mathbf{Y}}_j
        = \mathbf{S}(\mathbf{z}_m),$  
        };

        \node[box] (C0) at (3.8, -3) 
        {
        $\mathbf{Y}(\mathbf{z}_m) + \mathcal{P}^{2} \sum_{\substack{j=1\\ j\neq m}}^{M}
        \Gamma(\mathbf{z}_m,\mathbf{z}_j)a^{1-h}$\\
        $  \frac{1}{\boldsymbol{\beta}_j}\,\mathbf{Y}(\mathbf{z}_j)
        = \mathbf{S}(\mathbf{z}_m) - \mathcal{P}^{2} \,\mathcal{J}(\mathbf{z}_m). $  
        };

        \node[box] (CC) at (2.1, -5) 
        { \eqref{eq:Jzm def}-\eqref{eq:Jzm=O P2 g Hnorm} for\\
        estimating $\mathcal{J}(\mathbf{z}_m)$
        };

        \node[box] (CC1) at (5, -5) 
        { Lemma \ref{lem:4.5}
        };

        \node[box] (C11) at (2.1, -8) 
        { Lemma  \ref{lem:Y H1 norm}
        };

        \node[box] (C12) at (5.2, -8) 
        { Convergent};

        \node[box] (C1) at (3.8, -10) 
        { $\mathbf{Y}(\mathbf{z})+\mathcal{P}^{2}\int_{\Omega} \Gamma(\mathbf{z},\mathbf{y})$\\
        $\mathbf{Y}(\mathbf{y})\,\rmd\mathbf{y}
        = \mathbf{S}(\mathbf{z}). $   
        };

        \node[box] (D) at (3.5, 3) {$\bigl|\mathbb{J}_1\bigr|
        = \mathcal{O}\!\left(a^{\frac{3h-2}{2}}\,
        \|\mathbf{v}^{\bf g}\|_{\mathbb{L}^2(D)}\,\right.$\\
        $\left.
        \|\mathbf{f}\|_{\mathbb{H}^{-1/2}(\partial\Omega)}\right),$}; 

        \node[box] (E) at (-5, 3) {Lemma \ref{lem:vg estimate}: $\|\mathbf{v}^{\mathbf{g}}\|_{\mathbb{L}^2(D)}
            \;\lesssim\; a^{\frac{5-2h}{6}} \,\|\mathbf{g}\|_{\mathbb{H}^{-1/2}(\partial\Omega)}.$}; 

        \node[bigbox] (Main) at (-2, -16) 
        { \textbf{Subsection \ref{subsec:4.4}:}
        \\
        $ \mathsf{J}
      = \sum_{j=1}^M \bu^{\bff}(\bz_j)
           \Bigg( \omega^2 \rho_1 \int_{D_j} \bv^{\bg}_j(\bx)\,\rmd \bx
                   \;+\; \mathcal{P}^2 \int_{\Omega_j} \mathbf{q}^{\bg}_j(\bx)\,\rmd \bx \Bigg)
           \;+\; \mathsf{Error}$, \
           $\mathsf{Error}$   in \eqref{eq:Err-def}. \\
           $\mathsf{J}\lesssim a^{\frac{(1-h)(9-5\epsilon)}{18(3-\epsilon)}}\,\mathcal{P}^6\,
                \|\bff\|_{\bbH^{-\frac{1}{2}}(\partial\Omega)}\,\|\bg\|_{\bbH^{-\frac{1}{2}}(\partial\Omega)}
           $
        }; 

        \draw[arrow] (B) -- (Main); 
        \draw[arrow] (B1) -- (B); 
        \draw[arrow] (F) -- (F1); 
        \draw[arrow] (F1) -- (B); 
        \draw[arrow] (F2) -- (F1); 
        \draw[arrow] (F3) -- (B); 
        \draw[arrow] (C) -- (C0); 
        \draw[arrow] (CC) -- (C0); 
        \draw[arrow] (CC1) -- (C0); 
        \draw[arrow] (C0) -- (C1); 
        \draw[arrow] (C11) -- (C1); 
        \draw[arrow] (0,4) -- (0,-14.7); 
        \draw[arrow] (C1) -- (1,-14.7); 
        \draw[arrow] (C12) -- (3.8,-8); 
        \draw[arrow] (D) -- (A); 
        \draw[arrow] (E) -- (A); 
    \end{tikzpicture}
    }
    \caption{Flow Chart of the proof strategy  of Theorem \ref{thm:N--D}}
    \label{fig:J_derivation_flowchart}
\end{figure}
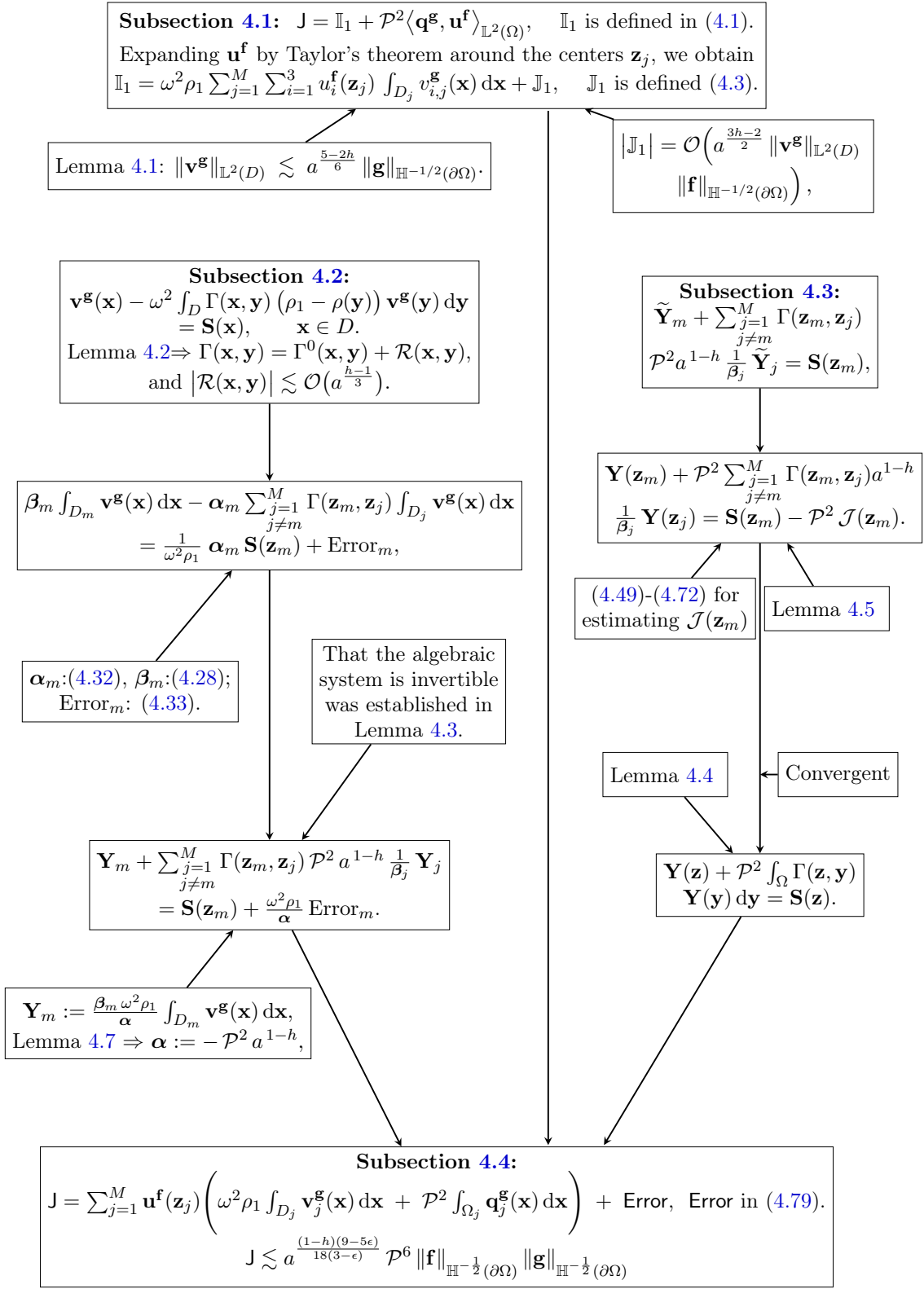
















\subsection{Calculation of the dominant term of $\mathbb{I}_1$}\label{subsec:4.1}

Let \(\mathbf{v}^{\bf g}(\cdot)\) be any solution to \eqref{eq:system vg}, and let \(\mathbf{u}^{\bf f}(\cdot)\) denote the solution to \eqref{eq:system}. For \(j=1,\dots,M\), set \(\mathbf{v}^{\bf g}_j(\cdot):=\mathbf{v}^{\bf g}|_{D_j}(\cdot)\). For each \(i\in\{1,2,3\}\), write \(v_{i,j}\) for the \(i\)-th component of \(\mathbf{v}^{\bf g}_j\). Then
\begin{align}\label{eq:I1 def in subsection 4.1}
\mathbb{I}_1
&= \omega^2 \big\langle \big(\rho_1-\rho(\mathbf{x})\big)\,\mathbf{v}^{\bf g}\,;\,\mathbf{u}^{\bf f}\big\rangle_{\mathbb{L}^2(D)}
= \omega^2 \rho_1 \sum_{j=1}^M \int_{D_j} \mathbf{v}^{\bf g}_j(\mathbf{x})\cdot\mathbf{u}^{\bf f}(\mathbf{x})\,\rmd \mathbf{x}
- \omega^2 \big\langle \rho(\mathbf{x})\,\mathbf{v}^{\bf g}\,;\,\mathbf{u}^{\bf f}\big\rangle_{\mathbb{L}^2(D)} .
\end{align}

\noindent
Expanding \(\mathbf{u}^{\bf f}\) by Taylor’s theorem around the centers \(\mathbf{z}_j\), we obtain
\begin{align}\label{eq:I1 def}
\mathbb{I}_1
= \omega^2 \rho_1 \sum_{j=1}^M \sum_{i=1}^3 u_i^{\bf f}(\mathbf{z}_j)\, \int_{D_j} v_{i,j}^{\bf g}(\mathbf{x})\,\rmd \mathbf{x}
+ \mathbb{J}_1 ,
\end{align}
where
\begin{align}\label{eq:J1 def}
\mathbb{J}_1
&= \omega^2 \rho_1 \sum_{j=1}^M \sum_{i=1}^3 \int_{D_j} v_{i,j}^{\bf g}(\mathbf{x}) \int_0^1 \nabla u_i^{\bf f}\!\left(\mathbf{z}_j+t(\mathbf{x}-\mathbf{z}_j)\right)\cdot(\mathbf{x}-\mathbf{z}_j)\,\rmd t\,\rmd \mathbf{x} \notag\\
&\quad - \omega^2 \sum_{i=1}^3 \int_D \rho(\mathbf{x})\, v_i^{\bf g}(\mathbf{x})\,u_i^{\bf f}(\mathbf{x})\,\rmd \mathbf{x}.
\end{align}

\noindent
Since \(\rho\in \mathbb{W}^{1,\infty}(\Omega)\), the solution \(\mathbf{u}^{\bf f}\in \mathbb{H}^1(\Omega)\) enjoys interior \(\mathbb{W}^{2,\infty}\)-regularity. Using the definition of \(\rho_1\) in \eqref{eq:rho1 def}, we estimate \(\mathbb{J}_1\) as
\begin{align}\label{eq:J1 estimate}
\bigl|\mathbb{J}_1\bigr|
&\lesssim a^{-2}\sum_{j=1}^M \sum_{i=1}^3
\big\| v_{i,j}^{\bf g} \big\|_{\mathbb{L}^2(D_j)}
\left\| \int_0^1 \nabla u_i^{\bf f}\!\left(\mathbf{z}_j+t(\mathbf{x}-\mathbf{z}_j)\right)\cdot(\mathbf{x}-\mathbf{z}_j)\,\rmd t \right\|_{\mathbb{L}^2(D_j)} \notag\\
&\quad + \sum_{i=1}^3 \big\| v_i^{\bf g} \big\|_{\mathbb{L}^2(D)}\, \big\| u_i^{\bf f} \big\|_{\mathbb{L}^2(D)} \notag\\
&\le a^{-1}
\left(\sum_{j=1}^M \sum_{i=1}^3 \big\| v_{i,j}^{\bf g} \big\|_{\mathbb{L}^2(D_j)}^2 \right)^{\!\frac12}
\left(\sum_{j=1}^M \sum_{i=1}^3 \big\| \nabla u_i^{\bf f} \big\|_{\mathbb{L}^2(D_j)}^2 \right)^{\!\frac12} \notag\\
&\quad + \left(\sum_{i=1}^3 \big\| v_i^{\bf g} \big\|_{\mathbb{L}^2(D)}^2\right)^{\!\frac12}
\left(\sum_{i=1}^3 \big\| u_i^{\bf f} \big\|_{\mathbb{L}^2(D)}^2\right)^{\!\frac12} \notag\\
&= \mathcal{O}\!\left(
\big\|\mathbf{v}^{\bf g}\big\|_{\mathbb{L}^2(D)}\,
\left(a^{-1}\big\|\nabla \mathbf{u}^{\bf f}\big\|_{\mathbb{L}^2(D)}+\big\|\mathbf{u}^{\bf f}\big\|_{\mathbb{L}^2(D)}\right)
\right).
\end{align}

\medskip

\noindent
Moreover, from \eqref{eq:system} one has the single-layer potential representation of \(\mathbf{u}^{\bf f}\) with surface density \(\mathbf{f}\):
\begin{align}\label{eq:uf def}
\mathbf{u}^{\bf f}(\mathbf{x})
= \int_{\partial\Omega} \Gamma(\mathbf{x},\mathbf{y})\,\mathbf{f}(\mathbf{y})\,\rmd \sigma(\mathbf{y})
= \big\langle \Gamma(\mathbf{x},\mathbf{y}) , \mathbf{f}(\mathbf{y}) \big\rangle_{\mathbb{H}^{\frac12}(\partial\Omega)\times \mathbb{H}^{-\frac12}(\partial\Omega)},
\qquad \mathbf{x}\in\Omega.
\end{align}

\noindent
Here \(\Gamma(\cdot,\cdot)\) is characterized as the Neumann Green tensor solving
\begin{align}\label{eq:Gamma def}
\left\{
\begin{aligned}
\mathcal{L}_{\lambda,\mu}\,\Gamma(\mathbf{x},\mathbf{y})+\omega^2 \rho(\mathbf{x})\,\Gamma(\mathbf{x},\mathbf{y})
&= -\,\delta_{\mathbf{y}}(\mathbf{x})\,\mathcal{I} && \text{in } \Omega, \\[2mm]
\partial_{\nu}\Gamma(\mathbf{x},\mathbf{y})
&= \mathbf{0} && \text{on } \partial\Omega .
\end{aligned}
\right.
\end{align}

\noindent
The existence and uniqueness of $\Gamma(\cdot,\cdot)$, together with a detailed singularity analysis and pointwise bounds, are established in \cite[Section 5.2]{KS24}. Using \eqref{eq:uf def}, we therefore obtain
\begin{align}\label{eq:uif def}
    u_i^{\bf f}(\mathbf{x})=\sum_{\ell=1}^3\big\langle \Gamma_{i\ell}(\mathbf{x},\cdot), f_\ell(\cdot)\big\rangle_{\mathbb{H}^{\frac12}(\partial\Omega)\times \mathbb{H}^{-\frac12}(\partial\Omega)}
    \quad \mbox{for } \ i=1,2,3.
\end{align}
By the minimizing characterization of the trace operator, we have
\begin{align}\label{eq:Gamma H1/2 H1}
    \|\Gamma(\mathbf{x},\cdot)\|_{\mathbb{H}^{1/2}(\partial\Omega)}
\ \le\
\|\Gamma(\mathbf{x},\cdot)\|_{\mathbb{H}^{1}(\Omega^\diamond)},
\qquad \Omega^\diamond:=\Omega\setminus\overline{D}.
\end{align}
Therefore, using \eqref{eq:uif def} and the Cauchy–Schwarz inequality in the boundary duality, it follows that
\begin{align}\label{eq:uf L2 norm}
\|\mathbf{u}^{\bf f}\|_{\mathbb{L}^2(D)}^2
&= \int_D \Big|\big\langle \Gamma(\mathbf{x},\cdot), \mathbf{f}\big\rangle_{\mathbb{H}^{\frac12}\times \mathbb{H}^{-\frac12}}\Big|^2 \,\rmd\mathbf{x}  
\le \int_D \|\Gamma(\mathbf{x},\cdot)\|_{\mathbb{H}^{1/2}(\partial\Omega)}^2\,\rmd\mathbf{x}\cdot\|\mathbf{f}\|_{\mathbb{H}^{-1/2}(\partial\Omega)}^2 \notag\\
&\le \int_D \|\Gamma(\mathbf{x},\cdot)\|_{\mathbb{H}^{1}(\Omega^\diamond)}^2\,\rmd\mathbf{x}\cdot\|\mathbf{f}\|_{\mathbb{H}^{-1/2}(\partial\Omega)}^2 .
\end{align}
Taking square roots on both sides, we obtain
\begin{align}\label{eq:uf L2norm of Gamma}
\|\mathbf{u}^{\bf f}\|_{\mathbb{L}^2(D)}
&\le
\|\mathbf{f}\|_{\mathbb{H}^{-1/2}(\partial\Omega)}\,
\left(\int_D \|\Gamma(\mathbf{x},\cdot)\|_{\mathbb{H}^{1}(\Omega^\diamond)}^2\,\rmd\mathbf{x}\right)^{1/2}.
\end{align}

\noindent
Using the decomposition $\Gamma=\Gamma^0+\mathcal{R}$ (cf.\ \eqref{eq:Gamma-decomposition}), and recalling the pointwise behaviours
\[
|\Gamma^0(\mathbf{x},\mathbf{y})|\simeq|\mathbf{x}-\mathbf{y}|^{-1},
\qquad
|\nabla_{\mathbf{y}}\Gamma^0(\mathbf{x},\mathbf{y})|\simeq|\mathbf{x}-\mathbf{y}|^{-2},
\]
while $\mathcal{R}$ is controlled in the boundary layer (hence lower order in the present scaling), and recalling that 
\begin{align}
    \| \Gamma(\mathbf{x}, \cdot) \|_{\mathbb{H}^1(\Omega^\diamond)}^2 := \| \nabla \Gamma(\mathbf{x}, \cdot) \|_{\mathbb{L}^2(\Omega^\diamond)}^2 + \| \Gamma(\mathbf{x}, \cdot) \|_{\mathbb{L}2(\Omega^\diamond)}^2,
\end{align}
we obtain, by expanding the $\mathbb{H}^{1}(\Omega^\diamond)$-norm into bulk and gradient parts and keeping the dominant piece,
\begin{align}\label{eq:Gamma H1 norm}
\int_D \|\Gamma(\mathbf{x},\cdot)\|_{\mathbb{H}^{1}(\Omega^\diamond)}^2\,\rmd\mathbf{x}
& {\lesssim}   
\int_D \int_{\Omega^\diamond} \frac{1}{|\mathbf{x}-\mathbf{y}|^{4}}\,\rmd\mathbf{y}\,\rmd\mathbf{x} 
\stackrel{\footnotemark}{\le}   \int_D \int_{\Omega^\diamond} \frac{1}{\kappa(a)^{4}}\,\rmd\mathbf{y}\,\rmd\mathbf{x} \notag\\
&= |D|\,|\Omega^\diamond|\,\kappa(a)^{-4}
\ \lesssim\ |D|\,\kappa(a)^{-4}.   
\end{align}
\footnotetext{For \( \mathbf{x} \in D \) and \( \mathbf{y} \in \Omega^\diamond \), Assumption \ref{assumption1} gives
\begin{align}\label{eq:footnote}
    |\mathbf{x} - \mathbf{y}| \geq \operatorname{dist}(D, \partial\Omega^\diamond) = \operatorname{dist}(D, \partial\Omega) \geq \kappa(a).
\end{align}
}   

\noindent
Substituting \eqref{eq:Gamma H1 norm} into \eqref{eq:uf L2norm of Gamma}, and recalling the scalings $|D|\sim a^{h+2}$ and $\kappa(a)\sim a^{\frac{1-h}{3}}$, we conclude that
\begin{align}\label{eq:uf L2 norm finally}
\|\mathbf{u}^{\bf f}\|_{\mathbb{L}^2(D)}
&\le C\,\|\mathbf{f}\|_{\mathbb{H}^{-1/2}(\partial\Omega)}\,
\left(\,|D|\,\kappa(a)^{-4}\,\right)^{\frac12}= \mathcal{O}\!\left(a^{\frac{2+7h}{6}}  \|\mathbf{f}\|_{\mathbb{H}^{-1/2}(\partial\Omega)}\right).
\end{align}
 
\noindent
Differentiating \eqref{eq:uf def} with respect to $\mathbf{x}$ and repeating the above argument, we obtain
\begin{align}\label{eq:nabla uf L2 norm}
\|\nabla \mathbf{u}^{\bf f}\|_{\mathbb{L}^2(D)}^2
&\le \int_D \|\nabla_{\mathbf{x}}\Gamma(\mathbf{x},\cdot)\|_{\mathbb{H}^{1}(\Omega^\diamond)}^2\,\rmd\mathbf{x}\cdot\|\mathbf{f}\|_{\mathbb{H}^{-1/2}(\partial\Omega)}^2  \notag\\
&\lesssim \int_D \int_{\Omega^\diamond} \frac{1}{|\mathbf{x}-\mathbf{y}|^{6}}\,\rmd\mathbf{y}\,\rmd\mathbf{x}\cdot\|\mathbf{f}\|_{\mathbb{H}^{-1/2}(\partial\Omega)}^2 \notag\\
&\le \int_D \int_{\Omega^\diamond} \frac{1}{\kappa(a)^{6}}\,\rmd\mathbf{y}\,\rmd\mathbf{x}\cdot\|\mathbf{f}\|_{\mathbb{H}^{-1/2}(\partial\Omega)}^2  
\lesssim\ |D|\,\kappa(a)^{-6} \,\|\mathbf{f}\|_{\mathbb{H}^{-1/2}(\partial\Omega)}^2 \notag\\
&=\mathcal{O}\!\left(a^{3h}\,\|\mathbf{f}\|_{\mathbb{H}^{-1/2}(\partial\Omega)}^2\right).
\end{align}

\noindent
Hence, by substituting \eqref{eq:uf L2 norm finally} and \eqref{eq:nabla uf L2 norm} into \eqref{eq:J1 estimate}, we have
\begin{align}\label{eq:E412-final}
\bigl|\mathbb{J}_1\bigr|
= \mathcal{O}\!\left(a^{\frac{3h-2}{2}}\,
\|\mathbf{v}^{\bf g}\|_{\mathbb{L}^2(D)}\,
\|\mathbf{f}\|_{\mathbb{H}^{-1/2}(\partial\Omega)}\right).
\end{align}

\noindent
The subsequent lemma furnishes an a priori estimate for $\|\mathbf{v}^{\bf g}\|_{\mathbb{L}^2(D)}$.

\begin{lem}\label{lem:vg estimate}
Let $\mathbf{v}^{\mathbf{g}}$ be the solution of \eqref{eq:system vg}.  
Then the following a priori estimate holds:
\begin{align}\label{eq:vg L2 norm lesssim g}
    \|\mathbf{v}^{\mathbf{g}}\|_{\mathbb{L}^2(D)}
    \;\lesssim\; a^{\frac{5-2h}{6}} \,\|\mathbf{g}\|_{\mathbb{H}^{-1/2}(\partial\Omega)}.
\end{align}
\end{lem}

\begin{proof}
     See Subsection \ref{subsec:vg estimate}.
\end{proof}

By virtue of Lemma \ref{lem:vg estimate} and the estimate \eqref{eq:E412-final} for \( \mathbb{J}_1 \), we can reduce the task of estimating \( \mathbb{I}_1 \), as specified in \eqref{eq:I1 def}, to
\begin{align} \label{eq:elastic-414}
\mathbb{I}_1
=\omega^2 \rho_1
\sum_{j=1}^{M}\sum_{i=1}^{3}
u^{\bf f}_i({\bf z}_j)\,
\int_{D_j} v^{\bf g}_{i,j}(\bx)\,\rmd\bx
+ \mathcal{O}\!\left(
a^{\frac{7h-1}{6}}\,
\|\mathbf{f}\|_{\mathbb{H}^{-1/2}(\partial\Omega)}
\|\mathbf{g}\|_{\mathbb{H}^{-1/2}(\partial\Omega)}
\right).
\end{align}

\noindent
The objective of the subsequent subsection is to derive the algebraic system satisfied by the scalar family
\[
\left( \int_{D_j} v^{\bf g}_{i,j}(\bx)\,\rmd\bx \right)_{\substack{1\le i\le 3, \ 1\le j\le M, }}
\]
and to establish its invertibility.

\subsection{Invertibility of the algebraic system}\label{subsec:4.2}

Let $\mathbf{v}^{\bf g}(\cdot)$ denote the unique solution to the perturbed boundary–value problem \eqref{eq:system vg}. By standard potential theoretic arguments, $\mathbf{v}^{\bf g}$ satisfies the Lippmann–Schwinger representation
\begin{equation}\label{eq:vg LSE in Omega}
\mathbf{v}^{\bf g}(\bx)
-
\omega^2 \int_{D} \Gamma(\bx,\by)\,\big(\rho_1-\rho(\by)\big)\,\mathbf{v}^{\bf g}(\by)\,\rmd\by
=
\mathbf{S}(\bx),
\qquad \bx\in\Omega,
\end{equation}
where $\Gamma(\cdot,\cdot)$ is the Neumann Green tensor introduced in \eqref{eq:Gamma def}, and $\mathbf{S}(\cdot)$ is the background response field induced by the boundary traction $\mathbf{g}$, namely the solution of
\begin{equation}\label{eq:S(x) satisfies equation}
\begin{cases}
\mathcal{L}_{\lambda,\mu}\,\mathbf{S}(\bx) + \omega^2 \rho(\bx)\,\mathbf{S}(\bx) = \mathbf{0}, & \bx\in\Omega,\\[2mm]
\partial_\nu \mathbf{S}(\bx) = \mathbf{g}(\bx), & \bx\in\partial\Omega.
\end{cases}
\end{equation}

\noindent
Taking the restriction of \eqref{eq:Gamma def} to \(D\), we obtain 
\begin{equation}\label{eq:vg LSE in D}
\mathbf{v}^{\bf g}(\bx)
-
\omega^2 \int_{D} \Gamma(\bx,\by)\,\big(\rho_1-\rho(\by)\big)\,\mathbf{v}^{\bf g}(\by)\,\rmd\by
=
\mathbf{S}(\bx),
\qquad \bx\in D.
\end{equation}

\begin{lem}\label{lem:R(x,y)}
Let \(\Gamma(\mathbf{x},\mathbf{y})\) be the solution of \eqref{eq:Gamma def}. Then \(\Gamma\) admits the decomposition
\begin{align}\label{eq:Gamma-decomposition}
\Gamma(\mathbf{x},\mathbf{y})=\Gamma^0(\mathbf{x},\mathbf{y})+\mathcal{R}(\mathbf{x},\mathbf{y}),
\end{align}
where \(\Gamma^0(\mathbf{x},\mathbf{y})\) is given by \eqref{eq:Gamma0 def} and satisfies \eqref{eq:Gamma0 satisfy}. The remainder \(\mathcal{R}\) solves
\begin{align}\label{eq:R(x,y) satisfiy}
\left\{
\begin{aligned}
\big(\mathcal{L}_{\lambda,\mu}+\omega^2\rho(\mathbf{x})\big)\,\mathcal{R}(\mathbf{x},\mathbf{y})
&= -\,\omega^2 \rho(\mathbf{x})\,\Gamma^0(\mathbf{x},\mathbf{y}) && \text{in } \Omega, \\[2mm]
\partial_{\nu_{\mathbf{x}}}\mathcal{R}(\mathbf{x},\mathbf{y})
&= -\,\partial_{\nu_{\mathbf{x}}}\Gamma^0(\mathbf{x},\mathbf{y}) && \text{on } \partial\Omega .
\end{aligned}
\right.
\end{align}
For all \(\mathbf{x},\mathbf{y}\in\Omega\), the remainder \(\mathcal{R}\) satisfies the pointwise bound
\begin{align}\label{eq:Rxy lesssim dist}
    \big|\mathcal{R}(\mathbf{x},\mathbf{y})\big|
\lesssim \operatorname{dist}(\mathbf{x},\partial\Omega)^{-\frac13}\,
\operatorname{dist}(\mathbf{y},\partial\Omega)^{-\frac23}=\mathcal{O}\big(a^{\frac{h-1}{3}}\big).
\end{align}
\end{lem}

\begin{proof}
See Subsection~\ref{subsec:lemR(x,y) proof}.
\end{proof}

{
\begin{rem}
We note that the singular behavior of the $\mathcal{R}(\mathbf{x},\mathbf{y})$
depends on the distances of $\mathbf{x}$ and $\mathbf{y}$ to the boundary
$\partial\Omega$. In particular, the lower bound on $\dist(\mathbf{x},\partial\Omega)$
and $\dist(\mathbf{y},\partial\Omega)$ implied by the geometry allows us to control
the $\bbH^{1}$-norm of $\mathcal{R}(\mathbf{x},\mathbf{y})$ in the regime considered
here. In the subsequent analysis, when $\mathbf{x},\mathbf{y}\in D$, the most singular
configuration for $|\mathcal{R}(\mathbf{x},\mathbf{y})|$ occurs when both points are
closest to $\partial\Omega$; see Figure~\ref{Fig1}. We quantify this configuration
using $\kappa(a)$ defined in \eqref{eq:dist D partial Omega}. Accordingly, we estimate
the contributions of $\Gamma^{0}(\mathbf{x},\mathbf{y})$ and $\mathcal{R}(\mathbf{x},\mathbf{y})$
separately. In general, $\mathcal{R}(\mathbf{x},\mathbf{y})$ contributes at the same
order as $\Gamma^{0}(\mathbf{x},\mathbf{y})$ and therefore cannot be neglected.
\end{rem}

}

For $\mathbf{x}\in D_m$, \eqref{eq:vg LSE in D} can be rewritten as
\begin{align}\label{eq:vmg satisfies equation}
\big(\mathcal{I}-\omega^{2}\rho_{1}\,N_{D_m}\big)\big[\mathbf{v}^{\mathbf{g}}\big](\mathbf{x})
-\omega^{2}\rho_{1}\sum_{\substack{j=1\\ j\neq m}}^{M}\int_{D_j}\Gamma(\mathbf{x},\mathbf{y})\,\mathbf{v}^{\mathbf{g}}(\mathbf{y})\,\rmd\mathbf{y} \notag\\
= \mathbf{S}_m(\mathbf{x})
+\omega^{2}\rho_{1}\int_{D_m}\mathcal{R}(\mathbf{x},\mathbf{y})\,\mathbf{v}^{\mathbf{g}}(\mathbf{y})\,\rmd\mathbf{y}
-\omega^{2}\int_{D}\Gamma(\mathbf{x},\mathbf{y})\,\rho(\mathbf{y})\,\mathbf{v}^{\mathbf{g}}(\mathbf{y})\,\rmd\mathbf{y},
\end{align}
where $\mathbf{S}_m(\cdot):=\mathbf{S}(\cdot)\big|_{D_m}$, the remainder $\mathcal{R}(\cdot,\cdot)$ solves \eqref{eq:R(x,y) satisfiy}, and $N_{D_m}$ denotes the Newtonian operator on $D_m$ as in \eqref{eq:Newtonian def}. Multiplying \eqref{eq:vmg satisfies equation} by $(\omega^{2}\rho_{1})^{-1}$ and introducing the operator $\mathbf{W}_m$ as the inverse of $\big((\omega^{2}\rho_{1})^{-1}\mathcal{I}-N_{D_m}\big)$, i.e.
\begin{align}\label{eq:Wx satisfies equation}
\frac{1}{\omega^{2}\rho_{1}}\;\mathbf{W}_m(\mathbf{x})-N_{D_m}\!\big[\mathbf{W}_m\big](\mathbf{x})=\mathcal{I},
\qquad \mathbf{x}\in D_m,
\end{align}
and then applying $\mathbf{W}_m$ to both sides of \eqref{eq:vmg satisfies equation} and integrating over $D_m$ yield
\begin{align}\label{eq:vmg-LSE-M} 
&\boldsymbol{\beta}_m\int_{D_m}\mathbf{v}^{\mathbf{g}}(\mathbf{x})\,\rmd\mathbf{x}
-\sum_{\substack{j=1\\ j\neq m}}^{M}\int_{D_m}\mathbf{W}_m(\mathbf{x})\!\int_{D_j}\Gamma(\mathbf{x},\mathbf{y})\,\mathbf{v}^{\mathbf{g}}(\mathbf{y})\,\rmd\mathbf{y}\,\rmd\mathbf{x} \notag\\
=& \frac{1}{\omega^{2}\rho_{1}}\int_{D_m}\mathbf{W}_m(\mathbf{x})\,\mathbf{S}_m(\mathbf{x})\,\rmd\mathbf{x}
 -\frac{1}{\rho_{1}}\int_{D_m}\mathbf{W}_m(\mathbf{x})\!\int_{D}\Gamma(\mathbf{x},\mathbf{y})\,\rho(\mathbf{y})\,\mathbf{v}^{\mathbf{g}}(\mathbf{y})\,\rmd\mathbf{y}\,\rmd\mathbf{x}  \notag\\
&\quad+\int_{D_m}\mathbf{W}_m(\mathbf{x})\!\int_{D_m}\!\int_{0}^{1}
\nabla_{\mathbf{y}}\mathcal{R}\!\left(\mathbf{x},\mathbf{z}_m+t(\mathbf{y}-\mathbf{z}_m)\right)\cdot(\mathbf{y}-\mathbf{z}_m)\,\rmd t\,\mathbf{v}^{\mathbf{g}}(\mathbf{y})\,\rmd\mathbf{y}\,\rmd\mathbf{x}.
\end{align}

\noindent
The coefficient $\boldsymbol{\beta}_m\in\mathbb{C}$ is defined by
\begin{align}\label{eq:beta=1-int}
\boldsymbol{\beta}_m:=1-\int_{D_m}\mathbf{W}_m(\mathbf{x})\,\mathcal{R}(\mathbf{x},\mathbf{z}_m)\,\rmd\mathbf{x}.
\end{align}
Moreover,
\begin{align}\label{eq:Wm R int Dm}
\left|\int_{D_m}\mathbf{W}_m(\mathbf{x})\,\mathcal{R}(\mathbf{x},\mathbf{z}_m)\,\rmd\mathbf{x}\right|
&\le \|\mathbf{W}_m\|_{\mathbb{L}^{2}(D_m)}\,\|\mathcal{R}(\cdot,\mathbf{z}_m)\|_{\mathbb{L}^{2}(D_m)}  \notag \\
&\overset{\eqref{eq:Rxy lesssim dist}}{\underset{\eqref{eq:Wm L2 norm}}{\lesssim}}\frac{a^{1-h}}{\operatorname{dist}(D_m,\partial\Omega)}
\overset{\eqref{eq:dist D partial Omega}}{=}\mathcal{O}\!\left(a^{\frac{2(1-h)}{3}}\right),
\end{align}
and hence
\begin{align}\label{eq:betam=}
\boldsymbol{\beta}_m=1+\mathcal{O}\!\left(a^{\frac{2(1-h)}{3}}\right),\qquad m=1,\ldots,M.
\end{align}

\noindent
To derive a closed algebraic system, we expand $\Gamma(\cdot,\cdot)$ and $\mathbf{S}(\cdot)$ near the centers and apply \eqref{eq:vmg-LSE-M}, obtaining
\begin{align}\label{eq:betam vmg int}
\boldsymbol{\beta}_m\int_{D_m}\mathbf{v}^{\mathbf{g}}(\mathbf{x})\,\rmd\mathbf{x}
-\boldsymbol{\alpha}_m\sum_{\substack{j=1\\ j\neq m}}^{M}\Gamma(\mathbf{z}_m,\mathbf{z}_j)\int_{D_j}\mathbf{v}^{\mathbf{g}}(\mathbf{x})\,\rmd\mathbf{x}
=\frac{1}{\omega^{2}\rho_{1}}\;\boldsymbol{\alpha}_m\,\mathbf{S}(\mathbf{z}_m)+\mathrm{Error}_m ,
\end{align}
where the scattering weight is
\begin{align}\label{eq:alpha def}
\boldsymbol{\alpha}_m:=\int_{D_m}\mathbf{W}_m(\mathbf{x})\,\rmd\mathbf{x},
\end{align}
and the remainder term equals
\begin{align}\label{eq:Errorm def}
\mathrm{Error}_m
&=\sum_{\substack{j=1\\ j\neq m}}^{M}\left(\int_{D_m}\mathbf{W}_m(\mathbf{x})\!\int_{0}^{1}
\nabla_{\mathbf{x}}\Gamma\!\left(\mathbf{z}_m+t(\mathbf{x}-\mathbf{z}_m),\mathbf{z}_j\right)\cdot(\mathbf{x}-\mathbf{z}_m)\,\rmd t\,\rmd\mathbf{x}\right)
\left(\int_{D_j}\mathbf{v}^{\mathbf{g}}(\mathbf{y})\,\rmd\mathbf{y}\right) \notag\\
&\quad+\sum_{\substack{j=1\\ j\neq m}}^{M}\int_{D_m}\mathbf{W}_m(\mathbf{x})\!\int_{D_j}\!\int_{0}^{1}
\nabla_{\mathbf{y}}\Gamma\!\left(\mathbf{x},\mathbf{z}_j+t(\mathbf{y}-\mathbf{z}_j)\right)\cdot(\mathbf{y}-\mathbf{z}_j)\,\rmd t\,\mathbf{v}^{\mathbf{g}}(\mathbf{y})\,\rmd\mathbf{y}\,\rmd\mathbf{x} \notag\\
&\quad+\frac{1}{\omega^{2}\rho_{1}}\int_{D_m}\mathbf{W}_m(\mathbf{x})\!\int_{0}^{1}
\nabla \mathbf{S}_m\!\left(\mathbf{z}_m+t(\mathbf{x}-\mathbf{z}_m)\right)\cdot(\mathbf{x}-\mathbf{z}_m)\,\rmd t\,\rmd\mathbf{x} \notag\\
&\quad+\int_{D_m}\mathbf{W}_m(\mathbf{x})\!\int_{D_m}\!\int_{0}^{1}
\nabla_{\mathbf{y}}\mathcal{R}\!\left(\mathbf{x},\mathbf{z}_m+t(\mathbf{y}-\mathbf{z}_m)\right)\cdot(\mathbf{y}-\mathbf{z}_m)\,\rmd t\,\mathbf{v}^{\mathbf{g}}(\mathbf{y})\,\rmd\mathbf{y}\,\rmd\mathbf{x} \notag\\
&\quad-\frac{1}{\rho_{1}}\int_{D_m}\mathbf{W}_m(\mathbf{x})\!\int_{D}\Gamma(\mathbf{x},\mathbf{y})\,\rho(\mathbf{y})\,\mathbf{v}^{\mathbf{g}}(\mathbf{y})\,\rmd\mathbf{y}\,\rmd\mathbf{x}.
\end{align}

\noindent
Since each $D_j=\mathbf{z}_j+aB$ and $\rho_j=\rho_1=\tilde{\rho}_1 a^{-2}$ for $j=1,\ldots,M$, it follows that $\boldsymbol{\alpha}_m=\boldsymbol{\alpha}$ for all $m$. Multiplying \eqref{eq:betam vmg int} by $\frac{\omega^{2}\rho_{1}}{\boldsymbol{\alpha}}$ and introducing
\begin{align}\label{eq:Ym def}
\mathbf{Y}_m:=\frac{\boldsymbol{\beta}_m\,\omega^{2}\rho_{1}}{\boldsymbol{\alpha}}\int_{D_m}\mathbf{v}^{\mathbf{g}}(\mathbf{x})\,\rmd\mathbf{x},
\qquad\text{with}\quad \boldsymbol{\alpha}:=-\,\mathcal{P}^{2}\,a^{\,1-h},
\end{align}
we obtain the algebraic system
\begin{align}\label{eq:The algebraic system}
\mathbf{Y}_m+\sum_{\substack{j=1\\ j\neq m}}^{M}\Gamma(\mathbf{z}_m,\mathbf{z}_j)\,\mathcal{P}^{2}\,a^{\,1-h}\,\frac{1}{\boldsymbol{\beta}_j}\,\mathbf{Y}_j
= \mathbf{S}(\mathbf{z}_m)+\frac{\omega^{2}\rho_{1}}{\boldsymbol{\alpha}}\,\mathrm{Error}_m .
\end{align}

\begin{lem}\label{lem:The algebraic system is invertible}
Viewed as an operator on $\ell_2$, the system \eqref{eq:The algebraic system} is invertible. Moreover,
\begin{align}\label{eq:discrete system of Ym}
\left(\sum_{m=1}^{M}\|\mathbf{Y}_m\|^{2}\right)^{\frac12}
\lesssim
\left(\sum_{m=1}^{M}\|\mathbf{S}(\mathbf{z}_m)\|^{2}\right)^{\frac12}
+ a^{\,h-3}\left(\sum_{m=1}^{M}\|\mathrm{Error}_m\|^{2}\right)^{\frac12}.
\end{align}
In particular,
\begin{align}\label{eq:sum Ym2 lesssim}
\left(\sum_{m=1}^{M}\|\mathbf{Y}_m\|^{2}\right)^{\frac12}
\lesssim a^{\,\frac{2(h-1)}{3}}\;\|\mathbf{g}\|_{\mathbb{H}^{-1/2}(\partial\Omega)} .
\end{align}
\end{lem}

\begin{proof}
    See Subsection \ref{subsec:The invertibility of system}. 
\end{proof}

\subsection{On the LSE Satisfied by $\bq^{\bf g}(\cdot)$} \label{subsec:4.3}

Let $(\widetilde{\mathbf{Y}}_1,\ldots,\widetilde{\mathbf{Y}}_M)$ denote the unique solution of the unperturbed algebraic system associated with \eqref{eq:The algebraic system}. In explicit form,
\begin{align}\label{eq:Ym LSE}
\widetilde{\mathbf{Y}}_m
+ \sum_{\substack{j=1\\ j\neq m}}^{M} \Gamma(\mathbf{z}_m,\mathbf{z}_j)\, \mathcal{P}^{2}a^{\,1-h}\,\frac{1}{\boldsymbol{\beta}_j}\,\widetilde{\mathbf{Y}}_j
= \mathbf{S}(\mathbf{z}_m),  
\end{align}
for $m=1,\dots,M$.

We next consider the LSE
\begin{align}\label{eq:Yz LSE def}
\mathbf{Y}(\mathbf{z})+\mathcal{P}^{2}\int_{\Omega} \Gamma(\mathbf{z},\mathbf{y})\,\mathbf{Y}(\mathbf{y})\,\rmd\mathbf{y}
= \mathbf{S}(\mathbf{z}),\qquad \mathbf{z}\in\Omega. 
\end{align}
Here $\Gamma(\cdot,\cdot)$ is the Neumann Green tensor solving \eqref{eq:Gamma def} and $\mathbf{S}(\cdot)$ is the solution of \eqref{eq:S(x) satisfies equation}. The following lemma will be used. 

\begin{lem}\label{lem:Y H1 norm}
    The LSE \eqref{eq:Yz LSE def} admits a unique solution $\mathbf{Y}(\cdot)$, and one has the estimate
\begin{align}\label{eq:Y P2 SH1 norm}
\|\mathbf{Y}\|_{\mathbb{H}^{1}(\Omega)}\lesssim \mathcal{P}^{2}\,\|\mathbf{S}\|_{\mathbb{H}^{1}(\Omega)}. 
\end{align}
\end{lem}
 
\begin{proof}
The operator in \eqref{eq:Yz LSE def} is invertible as a map $\mathbb{L}^{2}(\Omega)\to \mathbb{L}^{2}(\Omega)$; consequently,
\begin{align}\label{eq:Y lesssim S}
\|\mathbf{Y}\|_{\mathbb{L}^{2}(\Omega)}\lesssim \|\mathbf{S}\|_{\mathbb{L}^{2}(\Omega)}\le \|\mathbf{S}\|_{\mathbb{H}^{1}(\Omega)}.  
\end{align}
Taking the $\mathbb{H}^{1}(\Omega)$-norm on both sides of \eqref{eq:Yz LSE def} yields
\begin{align}
\|\mathbf{Y}\|_{\mathbb{H}^{1}(\Omega)}
\lesssim \|\mathbf{S}\|_{\mathbb{H}^{1}(\Omega)}+\mathcal{P}^{2}\,\|\mathcal N(\mathbf{Y})\|_{\mathbb{H}^{1}(\Omega)},
\end{align}
where $\mathcal N$ denotes the Newton potential,
\begin{align}\label{eq:Newton potential in Omega}
\mathcal N({\bf f})(\mathbf{x}):=\int_{\Omega} \Gamma(\mathbf{x},\mathbf{y})\,{\bf f}(\mathbf{y})\,\rmd\mathbf{y},\qquad \mathbf{x}\in\Omega.  
\end{align}
Using the continuity $\mathcal N:\mathbb{L}^{2}(\Omega)\to \mathbb{H}^{1}(\Omega)$ we infer
\begin{align}
\|\mathbf{Y}\|_{\mathbb{H}^{1}(\Omega)}
\lesssim \|\mathbf{S}\|_{\mathbb{H}^{1}(\Omega)}+\mathcal{P}^{2}\,\|\mathbf{Y}\|_{\mathbb{L}^{2}(\Omega)}
\overset{\eqref{eq:Y lesssim S}}{\lesssim} \mathcal{P}^{2}\,\|\mathbf{S}\|_{\mathbb{H}^{1}(\Omega)}.
\end{align}
This completes the proof of Lemma \ref{lem:Y H1 norm}. 
\end{proof}

\begin{rem}
    The function $\mathbf{S}(\cdot)$—the solution of \eqref{eq:S(x) satisfies equation}—admits a single–layer potential representation with density $\mathbf{g}(\cdot)$, namely
\begin{align}\label{eq:S-single-layer}
\mathbf{S}(\mathbf{x})=\mathbf{S} (\mathbf{g})(\mathbf{x}):=\int_{\partial\Omega} \Gamma(\mathbf{x},\mathbf{y})\,\mathbf{g}(\mathbf{y})\,\rmd\sigma(\mathbf{y}),\qquad \mathbf{x}\in\Omega,  
\end{align}
where $\Gamma(\cdot,\cdot)$ is the Neumann Green tensor from \eqref{eq:Gamma def}. In view of \eqref{eq:Y P2 SH1 norm} we obtain
\begin{align}
\|\mathbf{Y}\|_{\mathbb{H}^{1}(\Omega)}\lesssim \mathcal{P}^{2}\,\|\mathbf{S} (\mathbf{g})\|_{\mathbb{H}^{1}(\Omega)}.
\end{align}
Since the single–layer operator $\mathbf{S} :\mathbb{H}^{-1/2}(\partial\Omega)\to \mathbb{H}^{1}(\Omega)$ is bounded, we finally arrive at
\begin{align}\label{eq:Y H1norm lesssim g}
\|\mathbf{Y}\|_{\mathbb{H}^{1}(\Omega)}\lesssim \mathcal{P}^{2}\,\|\mathbf{g}\|_{\mathbb{H}^{-1/2}(\partial\Omega)}.
\end{align}
\end{rem}

Using the notation introduced in Assumption \ref{assumption1} (notably \eqref{Decoupage-Omega-new}) and the size relation
$|\Omega_j| \sim a^{\,1-h}$ for $1\le j\le M$ and $0\le h<1$, the discrete relation \eqref{eq:Ym LSE} can be recast in the form
\begin{align}\label{eq:Yzm+p2=Szm-P2Jzm}
\mathbf{Y}(\mathbf{z}_m) + \mathcal{P}^{2} \sum_{\substack{j=1\\ j\neq m}}^{M}
\Gamma(\mathbf{z}_m,\mathbf{z}_j)\, a^{1-h} \frac{1}{\boldsymbol{\beta}_j}\,\mathbf{Y}(\mathbf{z}_j)
= \mathbf{S}(\mathbf{z}_m) - \mathcal{P}^{2} \,\mathcal{J}(\mathbf{z}_m).
\end{align}

Here the remainder $\mathcal{J}(\mathbf{z}_m)$ is defined by removing the leading cell-wise contribution from the domain integral, namely
\begin{align}\label{eq:Jzm def}
\mathcal{J}(\mathbf{z}_m)
= \int_{\Omega} \Gamma(\mathbf{z}_m,\mathbf{y})\,\mathbf{Y}(\mathbf{y})\,\rmd\mathbf{y}
- \sum_{\substack{j=1\\ j\neq m}}^{M} \Gamma(\mathbf{z}_m,\mathbf{z}_j)\,\mathbf{Y}(\mathbf{z}_j)\,\frac{1}{\boldsymbol{\beta}_j}\,|\Omega_j|.
\end{align}

To estimate $\mathcal{J}(\mathbf{z}_m)$, we decompose the integral over $\Omega$ into a sum over the cubic subdomains and the thin boundary remainder region:
\begin{align}\label{eq:Gamma decompose Omega}
\int_{\Omega} \Gamma(\mathbf{z}_m,\mathbf{y})\,\mathbf{Y}(\mathbf{y})\,\rmd\mathbf{y}
= \int_{\Omega_{\text{cube}}} \Gamma(\mathbf{z}_m,\mathbf{y})\,\mathbf{Y}(\mathbf{y})\,\rmd\mathbf{y}
+ \int_{\Omega_r} \Gamma(\mathbf{z}_m,\mathbf{y})\,\mathbf{Y}(\mathbf{y})\,\rmd\mathbf{y}.
\end{align}

Consider the family of subdomains $\{\Omega_j^*\}_{j=1}^{\aleph}$, which are not assumed to be cubic, and which are chosen so that each of them intersects the boundary $\partial\Omega$; that is,
\[
\partial\Omega_n \cap \partial\Omega \neq \emptyset, \qquad 1 \leq n \leq \aleph.
\]
By construction, every $\Omega_j$ has volume $|\Omega_j| = a^{1-h}$, so a typical length scale of $\Omega_j$ (for instance, its maximal radius) is of order $a^{(1-h)/3}$. Consequently, the portion of $\Omega_j$ that meets $\partial\Omega$ has surface measure of order $a^{\frac{2}{3}(1-h)}$. Since the surface measure of $\partial\Omega$ itself is of order one, the total number of such subdomains intersecting $\partial\Omega$ is bounded above by a constant multiple of $a^{-\frac{2}{3}(1-h)}$. Hence, if we denote by $\Omega_r$ the union of all these boundary-touching subdomains, its volume satisfies
\begin{align}\label{eq:Omega_r=}
|\Omega_r|=\mathcal{O} \big(a^{(1-h)/3}\big).
\end{align}

Concerning the second term on the right-hand side of \eqref{eq:Gamma decompose Omega}, and noting that $\Gamma(\bz_m,\cdot) \in \mathbb{L}^{3-\epsilon}(\Omega)$ for fixed $\bz_m$, where $\epsilon>0$ is arbitrarily small (cf.\ Lemma \ref{lem:R(x,y)}), it follows that
\begin{align}
\left| \int_{\Omega_r} \Gamma(\bz_m,\by)\, \mathbf{Y}(\by) \, \rmd \by \right|
&\leq \| \mathbf{Y} \|_{\mathbb{L}^2(\Omega_r)} \,\| \Gamma(\bz_m,\cdot) \|_{\mathbb{L}^2(\Omega_r)} \notag \\
&\leq |\Omega_r|^{1/3} \| \mathbf{Y} \|_{\mathbb{L}^6(\Omega_r)} \,\| \Gamma(\bz_m,\cdot) \|_{\mathbb{L}^2(\Omega_r)}  \notag\\
&\leq |\Omega_r|^{1/3} \| \mathbf{Y} \|_{\mathbb{L}^6(\Omega)} \,\| \Gamma(\bz_m,\cdot) \|_{\mathbb{L}^2(\Omega_r)}. \notag
\end{align}

By virtue of Lemma \ref{lem:Y H1 norm}, we have \( \mathbf{Y}(\cdot) \in \mathbb{H}^1(\Omega) \subset \mathbb{L}^6(\Omega) \). Subsequently,
\begin{align}
\left| \int_{\Omega_r} \Gamma(\bz_m,\by)\, \mathbf{Y}(\by) \, \rmd \by \right|
&\lesssim |\Omega_r|^{1/3} \| \mathbf{Y} \|_{\mathbb{H}^1(\Omega)} \,\| \Gamma(\bz_m,\cdot) \|_{\mathbb{L}^2(\Omega_r)} \notag \\
&\stackrel{\eqref{eq:Y P2 SH1 norm}}{\lesssim} \mathcal{P}^2\,|\Omega_r|^{1/3} \| \mathbf{S} \|_{\mathbb{H}^1(\Omega)} \,\| \Gamma(\bz_m,\cdot) \|_{\mathbb{L}^2(\Omega_r)},  \notag
\end{align}
which, through the application of \eqref{eq:S-single-layer} and the continuity of the single-layer operator mapping \( \mathbb{H}^{-1/2}(\partial\Omega) \) to \( \mathbb{H}^1(\Omega) \), may be simplified to
\begin{align}
\left| \int_{\Omega_r} \Gamma(\bz_m,\by)\, \mathbf{Y}(\by) \, \rmd \by \right| 
&\lesssim  \mathcal{P}^2\, |\Omega_r|^{1/3} \| \bg \|_{\mathbb{H}^{-1/2}(\partial\Omega)} \,\| \Gamma(\bz_m,\cdot) \|_{\mathbb{L}^2(\Omega_r)}  \notag \\
&\lesssim \mathcal{P}^2\, |\Omega_r|^{\frac{9-5\epsilon}{6(3-\epsilon)}}   \| \bg \|_{\mathbb{H}^{-1/2}(\partial\Omega)} \,\| \Gamma(\bz_m,\cdot) \|_{\mathbb{L}^{3-\epsilon}(\Omega_r)}.  \notag
\end{align}

Subsequently, by leveraging the fact that \( \Gamma(\bz_m,\cdot) \in \mathbb{L}^{3-\epsilon}(\Omega_r) \), whence \( \| \Gamma(\bz_m,\cdot) \|_{\mathbb{L}^{3-\epsilon}(\Omega_r)} \sim 1 \), we simplify the preceding estimate to
\begin{align}\label{eq:Gamma Y int in Omegar}
\left| \int_{\Omega_r} \Gamma(\bz_m,\by)\, \mathbf{Y}(\by) \, \rmd \by \right| 
\lesssim \mathcal{P}^2\, |\Omega_r|^{\frac{9-5\epsilon}{6(3-\epsilon)}} \| \bg \|_{\mathbb{H}^{-\frac{1}{2}}(\partial\Omega)} 
\stackrel{\eqref{eq:Omega_r=}}{=} \mathcal{O}\!\left( a^{\frac{(1-h)(9-5\epsilon)}{18(3-\epsilon)}} \mathcal{P}^2 \| \bg \|_{\mathbb{H}^{-\frac{1}{2}}(\partial\Omega)} \right).   
\end{align}

Collecting \eqref{eq:Jzm def}–\eqref{eq:Gamma Y int in Omegar} leads to the separation
\begin{align}\label{eq:Jzm=widetilde Jzm}
\mathcal{J}(\mathbf{z}_m)
= \widetilde{\mathcal{J}}(\mathbf{z}_m)
+ \mathcal{O}\Big(\mathcal{P}^{2}\,\|\bg\|_{\mathbb{H}^{-\frac{1}{2}}(\partial\Omega)}\,
  a^{\frac{(1-h)(9-5\epsilon)}{18(3-\epsilon)}}\Big),
\qquad a\ll 1,
\end{align}
with
\begin{align}\label{eq:widetilde J_zm}
\widetilde{\mathcal{J}}(\mathbf{z}_m)
:= \sum_{\ell=1}^{M}\int_{\Omega_\ell} \Gamma(\mathbf{z}_m,\mathbf{y})\,\mathbf{Y}(\mathbf{y})\,\rmd\mathbf{y}
   - \sum_{\substack{j=1\\ j\neq m}}^{M}
     \Gamma(\mathbf{z}_m,\mathbf{z}_j)\,\mathbf{Y}(\mathbf{z}_j)\,\frac{1}{\boldsymbol{\beta}_j}\,|\Omega_j|.
\end{align}

Our next step involves approximating $\widetilde{\mathcal{J}}(\mathbf{z}_m)$. For this purpose, we rewrite it as
\begin{align}
\widetilde{\mathcal{J}}(\mathbf{z}_m) 
&= \sum_{\substack{\ell=1 \\ \ell \neq m}}^M \int_{\Omega_\ell} \left[ \Gamma(\bz_m,\by)\,\mathbf{Y}(\by) - \Gamma(\bz_m,\bz_\ell)\, \mathbf{Y}(\bz_\ell) \frac{1}{\boldsymbol{\beta}_\ell} \right] \rmd \by 
+ \int_{\Omega_m} \Gamma(\bz_m,\by)\,\mathbf{Y}(\by)\,\rmd \by \notag\\
&\stackrel{\eqref{eq:betam=}}{=} \sum_{\substack{\ell=1 \\ \ell \neq m}}^M \int_{\Omega_\ell} \left[ \Gamma(\bz_m,\by)\, \mathbf{Y}(\by) - \Gamma(\bz_m,\bz_\ell)\, \mathbf{Y}(\bz_\ell) \right] \rmd \by 
+ \int_{\Omega_m} \Gamma(\bz_m,\by)\,\mathbf{Y}(\by)\,\rmd \by  \notag\\
&\quad - \sum_{\substack{\ell=1 \\ \ell \neq m}}^M \Gamma(\bz_m,\bz_\ell)\, \mathbf{Y}(\bz_\ell)\,|\Omega_\ell|\,
\frac{\displaystyle \int_{D_\ell} \mathbf{W}_\ell(\bx)\,\mathcal{R}(\bx, \bz_\ell)\,\rmd \bx }{\displaystyle 1 - \int_{D_\ell} \mathbf{W}_\ell(\bx)\,\mathcal{R}(\bx, \bz_\ell)\, \rmd \bx}.
\end{align}

Through the application of a Taylor expansion of $\Gamma(\bz_m,\cdot)\,\mathbf{Y}(\cdot)$ around the point $\bz_\ell$, we derive
\begin{align}
\widetilde{\mathcal{J}}(\mathbf{z}_m) 
&= \sum_{\substack{\ell=1 \\ \ell \neq m}}^M \int_{\Omega_\ell} \int_0^1 \Gamma\big(\bz_m,\bz_\ell + t(\by - \bz_\ell)\big)\,\nabla \mathbf{Y}\big(\bz_\ell + t(\by - \bz_\ell)\big)\cdot (\by - \bz_\ell)\, t  \,\rmd t \,\rmd \by \notag \\
&\quad + \sum_{\substack{\ell=1 \\ \ell \neq m}}^M \int_{\Omega_\ell} \int_0^1 \mathbf{Y}\big(\bz_\ell + t(\by - \bz_\ell)\big)\,\nabla_{\mathbf{y}} \Gamma\big(\bz_m,\bz_\ell + t(\by - \bz_\ell)\big)\cdot (\by - \bz_\ell)\, t \,\rmd t \,\rmd \by  \notag \\
&\quad + \int_{\Omega_m} \Gamma(\bz_m,\by)\,\mathbf{Y}(\by)\,\rmd \by 
- \sum_{\substack{\ell=1 \\ \ell \neq m}}^M \Gamma(\bz_m,\bz_\ell)\, \mathbf{Y}(\bz_\ell)\,|\Omega_\ell|\,
\frac{\displaystyle \int_{D_\ell} \mathbf{W}_\ell(\bx)\,\mathcal{R}(\bx, \bz_\ell)\,\rmd \bx}{\displaystyle 1 - \int_{D_\ell} \mathbf{W}_\ell(\bx)\,\mathcal{R}(\bx, \bz_\ell) \,\rmd \bx}. 
\end{align}

Using the decomposition $\Gamma(\mathbf{x},\mathbf{y})=\Gamma^0(\mathbf{x},\mathbf{y})+\mathcal R(\mathbf{x},\mathbf{y})$ valid for $\mathbf{x}\neq \mathbf{y}$ from Lemma~\ref{lem:R(x,y)} and retaining the dominant $\Gamma^0(\mathbf{x},\mathbf{y})$ part, we rely on the standard kernel bounds
\begin{align}\label{eq:Gamma nabla Gamma}
|\Gamma(\mathbf{x},\mathbf{y})|\lesssim \frac{1}{|\mathbf{x}-\mathbf{y}|},
\qquad
|\nabla_{\mathbf{y}} \Gamma(\mathbf{x},\mathbf{y})|\lesssim \frac{1}{|\mathbf{x}-\mathbf{y}|^{2}},
\qquad \mathbf{x}\neq \mathbf{y}.
\end{align}

After the Taylor expansion of $\Gamma(\mathbf{z}_m,\cdot)\,\mathbf{Y}(\cdot)$ about $\mathbf{z}_\ell$ and the estimate \eqref{eq:Gamma nabla Gamma}, we obtain
\begin{align}\label{eq:Jzm lesssim Lzm}
|\widetilde{\mathcal{J}}(\mathbf{z}_m)|
&\lesssim
a^{\frac{1-h}{3}}\sum_{\substack{\ell=1\\ \ell\neq m}}^{M}
\int_{\Omega_\ell} \int_0^1 \frac{\big|\nabla \mathbf{Y} \big(\mathbf{z}_\ell+t(\mathbf{y}-\mathbf{z}_\ell)\big)\big|}{\big|\mathbf{z}_m-\mathbf{z}_\ell-t(\mathbf{y}-\mathbf{z}_\ell)\big|}\,\rmd t\,\rmd\mathbf{y} \notag\\
&\quad
+ a^{\frac{1-h}{3}}\sum_{\substack{\ell=1\\ \ell\neq m}}^{M}
\int_{\Omega_\ell} \int_0^1
\frac{\big|\mathbf{Y} \big(\mathbf{z}_\ell+t(\mathbf{y}-\mathbf{z}_\ell)\big)-\mathbf{Y}(\mathbf{z}_\ell)\big|}{\big|\mathbf{z}_m-\mathbf{z}_\ell-t(\mathbf{y}-\mathbf{z}_\ell)\big|^{2}}\,\rmd t \,\rmd\mathbf{y} \notag\\
&\quad
+ \int_{\Omega_m}|\Gamma(\mathbf{z}_m,\mathbf{y})|\,|\mathbf{Y}(\mathbf{y})|\,\rmd\mathbf{y} + |\mathbb{L}(\mathbf{z}_m)|,
\end{align}
where the term $\mathbb{L}(\mathbf{z}_m)$ is
\begin{align}\label{eq:Lzm def}
\mathbb{L}(\mathbf{z}_m)
:= \sum_{\substack{\ell=1\\ \ell\neq m}}^{M}
\Gamma(\mathbf{z}_m,\mathbf{z}_\ell)\,\mathbf{Y}(\mathbf{z}_\ell)\,|\Omega_\ell|\,
\frac{\displaystyle \int_{D_\ell} \mathbf{W}_\ell(\mathbf{x})\,\mathcal R(\mathbf{x},\mathbf{z}_\ell)\,\rmd\mathbf{x}}
     {\displaystyle 1-\int_{D_\ell} \mathbf{W}_\ell(\mathbf{x})\,\mathcal R(\mathbf{x},\mathbf{z}_\ell)\,\rmd\mathbf{x}}.
\end{align}

By taking absolute values in \eqref{eq:Lzm def}, exploiting $|\Omega_\ell|=a^{\,1-h}$, \eqref{eq:Wm R int Dm}, and \eqref{eq:Gamma nabla Gamma}, we derive
\begin{align}
|\mathbb{L}(\mathbf{z}_m)|
\;\lesssim\; a^{\frac{5(1-h)}{3}}
\Bigg(\sum_{\substack{\ell=1\\ \ell\neq m}}^{M}\frac{1}{|\mathbf{z}_m-\mathbf{z}_\ell|^{2}}\Bigg)^{\!1/2}
\Bigg(\sum_{\ell=1}^{M} |\mathbf{Y}(\mathbf{z}_\ell)|^{2}\Bigg)^{\!1/2}.
\end{align}

A refinement via Cauchy–Schwarz with the auxiliary vector
$(\widetilde{\mathbf{Y}}_1,\dots,\widetilde{\mathbf{Y}}_M)$ solving \eqref{eq:Ym LSE} and \eqref{eq:sum Ym2 lesssim} yields
\begin{align}\label{eq:Lzm lesssim g Hnorm}
|\mathbb{L}(\mathbf{z}_m)|
&\;\lesssim\;
\Bigg(\sum_{\substack{\ell=1\\ \ell\neq m}}^{M}\frac{1}{|\mathbf{z}_m-\mathbf{z}_\ell|^{2}}\Bigg)^{\!1/2}
\Bigg[
a^{\frac{5(1-h)}{3}}\Bigg(\sum_{\ell=1}^{M}\big|\mathbf{Y}(\mathbf{z}_\ell)-\widetilde{\mathbf{Y}}_\ell\big|^{2}\Bigg)^{\!1/2}
+ a^{ 1-h  }\,\|\bg\|_{\mathbb{H}^{-1/2}(\partial\Omega)}
\Bigg].
\end{align}

Returning to \eqref{eq:Lzm lesssim g Hnorm} and applying Cauchy–Schwarz together with \eqref{eq:Jzm lesssim Lzm}, we arrive at
\begin{align}
|\widetilde{\mathcal{J}}(\mathbf{z}_m)|
&\;\lesssim\;
a^{\frac{5(1-h)}{6}}\,\|\nabla \mathbf{Y}\|_{\mathbb{L}^2(\Omega)}
\Bigg(\sum_{\substack{\ell=1\\ \ell\neq m}}^{M}\frac{1}{|\mathbf{z}_m-\mathbf{z}_\ell|^{2}}\Bigg)^{\frac{1}{2}}
+ a^{\frac{5(1-h)}{6}}\,\|\mathbf{Y}\|_{\mathbb{L}^2(\Omega)}
\Bigg(\sum_{\substack{\ell=1\\ \ell\neq m}}^{M}\frac{1}{|\mathbf{z}_m-\mathbf{z}_\ell|^{4}}\Bigg)^{\frac{1}{2}} \notag\\
&\quad+ \|\mathbf{Y}\|_{\mathbb{L}^2(\Omega)}
\Bigg(\int_{\Omega_m}\frac{1}{|\mathbf{z}_m-\mathbf{y}|^{2}}\,\rmd\mathbf{y}\Bigg)^{\frac{1}{2}}
+ a^{\frac{5(1-h)}{3}}
\Bigg(\sum_{\substack{\ell=1\\ \ell\neq m}}^{M}\frac{1}{|\mathbf{z}_m-\mathbf{z}_\ell|^{2}}\Bigg)^{\frac{1}{2}}
\Bigg(\sum_{\ell=1}^{M}\big|\mathbf{Y}(\mathbf{z}_\ell)-\widetilde{\mathbf{Y}}_\ell\big|^{2}\Bigg)^{\frac{1}{2}}
\notag\\
&\quad+ a^{1-h}
\Bigg(\sum_{\substack{\ell=1\\ \ell\neq m}}^{M}\frac{1}{|\mathbf{z}_m-\mathbf{z}_\ell|^{2}}\Bigg)^{\frac{1}{2}}
\|\bg\|_{\mathbb{H}^{-\frac{1}{2}}(\partial\Omega)} .
\end{align}

The following lemma provides estimates for two types of discrete sums: sums of negative powers of the distances \(|\mathbf{z}_m - \mathbf{z}_j|\) between the centers of the hard inclusions, and sums of negative powers of the distances \(\mathrm{dist}(D_j, \partial\Omega)\) from the inclusions to the boundary \(\partial\Omega\).

\begin{lem}\cite[Lemma 4.6]{GS25}\label{lem:4.5}
Consider a collection of sets $\{D_m\}_{m=1}^M$ with $D_m=\mathbf{z}_m+aB$ contained in a domain $\Omega$. Then the following bounds hold.

\begin{enumerate}
  \item \emph{Reciprocal powers of center-to-center separations among hard inclusions.}
  \begin{align}\label{eq:sum zm-zj}
       \sum_{\substack{j=1\\ j\neq m}}^{M} |\mathbf{z}_m-\mathbf{z}_j|^{-k}
    =
    \begin{cases}
      \mathcal{O}\!\left(d^{-3}\right), & k<3,\\[2mm]
      \mathcal{O}\!\left(d^{-k}\right), & k>3,
    \end{cases}
  \end{align}

  \item \emph{Reciprocal powers of the distances from hard inclusions to the boundary.}
  \begin{align}\label{eq:dist Dj and partial Omega}
      \sum_{j=1}^{M} \dfrac{1}{\operatorname{dist}^{\,k}\!\left(D_j,\partial\Omega\right)}
    =
    \begin{cases}
      \mathcal{O}\!\left(d^{-3}\right), & k<3,\\[2mm]
      \mathcal{O}\!\left(d^{-k}\right), & k>3.
    \end{cases}
  \end{align} 
\end{enumerate}   
\end{lem}

Invoking the spacing and boundary-sum estimates (cf.\ \eqref{eq:sum zm-zj}) transforms the discrete sums into powers of the minimal separation $d$ defined in \eqref{eq:dmin} and boundary distance, giving a key relation for subsequent analysis,
\begin{align}\label{eq:widetilde Jzm}
|\widetilde{\mathcal{J}}(\mathbf{z}_m)|
&\;\lesssim\;
a^{\frac{1-h}{3}}\,\|\nabla \mathbf{Y}\|_{\mathbb{L}^2(\Omega)}
+ a^{\frac{1-h}{6}}\,\|\mathbf{Y}\|_{\mathbb{L}^2(\Omega)}
+ \|\mathbf{Y}\|_{\mathbb{L}^2(\Omega)}
  \Bigg(\int_{\Omega_m}\frac{1}{|\mathbf{z}_m-\mathbf{y}|^{2}}\,\rmd\mathbf{y}\Bigg)^{\frac{1}{2}}\notag\\
&\quad
+ a^{\frac{7(1-h)}{6}}
\Bigg(\sum_{\ell=1}^{M}\big|\mathbf{Y}(\mathbf{z}_\ell)-\widetilde{\mathbf{Y}}_\ell\big|^{2}\Bigg)^{\frac{1}{2}}
+ a^{\frac{1-h}{2}}\,\|\bg\|_{\mathbb{H}^{-\frac{1}{2}}(\partial\Omega)} .
\end{align}

To obtain a more explicit estimate of $\widetilde{\mathcal{J}}(\mathbf{z}_m)$, we define and bound the third term on the right-hand side of \eqref{eq:widetilde Jzm} as
\begin{align}
\widetilde{\mathcal{J}}_3(\mathbf{z}_m)
&:= \int_{B(\bz_m; r)} \frac{1}{|\bz_m - \by|^2} \,\rmd \by
   + \int_{\Omega_m \setminus B(\bz_m; r)} \frac{1}{|\bz_m - \by|^2} \,\rmd \by,
\end{align}
where $B(\bz_m; r)$ denotes the ball with center $\bz_m$ and radius $r$, and $r$ belongs to $ \left[0, \tfrac{\sqrt{3}}{2} a^{\frac{1-h}{3}}\right]$. Subsequently,
\begin{align}
\widetilde{\mathcal{J}}_3(\mathbf{z}_m)
&\lesssim \int_0^r \int_{\partial B(\bz_m; s)} \frac{1}{|\bz_m - \by|^2} \,\rmd \sigma(\by)\,\rmd s
      + \frac{1}{r^2} \big|\Omega_m \setminus B(\bz_m; r)\big|\notag \\
&= \int_0^r \frac{1}{s^2} |\partial B(\bz_m; s)| \,\rmd s
   + \frac{1}{r^2} \big|\Omega_m \setminus B(\bz_m; r)\big| \notag \\
&= \frac{8\pi r}{3} + \frac{1}{r^2} a^{1-h}
 \;\leq\; \max_{\,r \in \left[0, \frac{\sqrt{3}}{2} a^{\frac{1-h}{3}}\right]} \operatorname{dist}(r,a),
\end{align}
where
\begin{align}
\operatorname{dist}(r,a) := \frac{8\pi}{3}\, r + \frac{1}{r^2}\, a^{1-h}.
\end{align}
Combining the admissible range for $r$, it is straightforward to check that 
\begin{align}
\max_{\,r \in \left[0, \frac{\sqrt{3}}{2} a^{\frac{1-h}{3}}\right]} \operatorname{dist}(r,a)
= (48\pi^2)^{1/3} a^{(1-h)/3}.
\end{align}
Thus,
\begin{align}\label{eq:J3zm is O}
\widetilde{\mathcal{J}}_3(\mathbf{z}_m) = \mathcal{O}\!\left(a^{\frac{1-h}{3}}\right).  
\end{align}
Ultimately, by combining \eqref{eq:widetilde Jzm} and \eqref{eq:J3zm is O} and leveraging the relation $d \sim a^{\frac{1-h}{3}}$ defined in \eqref{eq:dmin}, we arrive at
\begin{align}
|\widetilde{\mathcal{J}}(\mathbf{z}_m)|
&\lesssim a^{\frac{1-h}{3}} \|\nabla \mathbf{Y}\|_{\mathbb{L}^2(\Omega)} + a^{\frac{1-h}{6}} \|\mathbf{Y}\|_{\mathbb{L}^2(\Omega)} \notag\\
&\quad+ a^{\frac{7(1-h)}{6}} \left( \sum_{\ell=1}^M \big|\mathbf{Y}(\bz_\ell) - \widetilde{\mathbf{Y}}_\ell\big|^2 \right)^{\frac{1}{2}}
     + a^{\frac{1-h}{2}} \|\bg\|_{\mathbb{H}^{-\frac{1}{2}}(\partial\Omega)} \notag \\
&\stackrel{\eqref{eq:Y H1norm lesssim g}}{\lesssim} a^{\frac{1-h}{6}} \mathcal{P}^{2}\,\|\bg\|_{\mathbb{H}^{-\frac{1}{2}}(\partial\Omega)}
  + a^{\frac{7(1-h)}{6}} \left( \sum_{\ell=1}^M \big|\mathbf{Y}(\bz_\ell) - \widetilde{\mathbf{Y}}_\ell\big|^2 \right)^{\frac{1}{2}}.
\end{align}

By substituting the preceding estimate into \eqref{eq:Jzm=widetilde Jzm}, we obtain
\begin{align}\label{eq:Jzm=O P2 g Hnorm}
\mathcal{J}(\mathbf{z}_m)
&= \mathcal{O}\!\left(
\mathcal{P}^{2}\, a^{\frac{(9-5\epsilon)(1-h)}{18(3-\epsilon)}} \,
\|\bg\|_{\mathbb{H}^{-1/2}(\partial\Omega)}
+ a^{\frac{7(1-h)}{6}}
\left( \sum_{\ell=1}^{M} \big|\mathbf{Y}(\mathbf{z}_\ell)-\widetilde{\mathbf{Y}}_\ell\big|^{2} \right)^{\!\frac{1}{2}}
\right).
\end{align}

Returning to \eqref{eq:Yzm+p2=Szm-P2Jzm} and using \eqref{eq:Jzm=O P2 g Hnorm}, we infer
\begin{align}\label{eq:Yzm+P2=Szm+O}
\mathbf{Y}(\mathbf{z}_m)
&+ \mathcal{P}^{2}\!\sum_{\substack{j=1\\ j\neq m}}^{M}
\Gamma(\mathbf{z}_m,\mathbf{z}_j)\, a^{\,1-h}\,\frac{1}{\boldsymbol{\beta}_j}\,\mathbf{Y}(\mathbf{z}_j)
= \mathbf{S}(\mathbf{z}_m)
+ \mathcal{O}\!\left(
\mathcal{P}^{4}\, a^{\frac{(9-5\epsilon)(1-h)}{18(3-\epsilon)}} \,
\|\bg\|_{\mathbb{H}^{-1/2}(\partial\Omega)}
\right) \notag\\
&\quad+ \mathcal{O}\!\left(
a^{\frac{7(1-h)}{6}} \mathcal{P}^{2}
\left( \sum_{\ell=1}^{M} \big|\mathbf{Y}(\mathbf{z}_\ell)-\widetilde{\mathbf{Y}}_\ell\big|^{2} \right)^{\!\frac{1}{2}}
\right).
\end{align}

Subtracting \eqref{eq:Ym LSE} from \eqref{eq:Yzm+P2=Szm+O}  yields the difference system
\begin{align}
\big(\widetilde{\mathbf{Y}}_m-\mathbf{Y}(\mathbf{z}_m)\big)
&+ \sum_{\substack{j=1\\ j\neq m}}^{M}
\Gamma(\mathbf{z}_m,\mathbf{z}_j)\, \mathcal{P}^{2} a^{\,1-h}\,\frac{1}{\boldsymbol{\beta}_j}\,
\big(\widetilde{\mathbf{Y}}_j-\mathbf{Y}(\mathbf{z}_j)\big)
= \mathcal{O}\!\left(
\mathcal{P}^{4}\, a^{\frac{(9-5\epsilon)(1-h)}{18(3-\epsilon)}} \,
\|\bg\|_{\mathbb{H}^{-1/2}(\partial\Omega)}
\right) \notag\\
&\quad 
+ \mathcal{O}\!\left(
a^{\frac{7(1-h)}{6}} \mathcal{P}^{2}
\left( \sum_{\ell=1}^{M} \big|\mathbf{Y}(\mathbf{z}_\ell)-\widetilde{\mathbf{Y}}_\ell\big|^{2} \right)^{\!\frac{1}{2}}
\right).
\end{align}

Applying Lemma \ref{lem:The algebraic system is invertible} gives
\begin{align}
\left( \sum_{m=1}^{M} \big|\widetilde{\mathbf{Y}}_m-\mathbf{Y}(\mathbf{z}_m)\big|^{2} \right)^{\!\frac{1}{2}}
&\lesssim \mathcal{P}^{4}\,
a^{-\frac{(1-h)(9-2\epsilon)}{9(3-\epsilon)}}\,
\|\bg\|_{\mathbb{H}^{-1/2}(\partial\Omega)}
+ a^{\frac{2(1-h)}{3}} \mathcal{P}^{2}
\left( \sum_{\ell=1}^{M} \big|\mathbf{Y}(\mathbf{z}_\ell)-\widetilde{\mathbf{Y}}_\ell\big|^{2} \right)^{\!\frac{1}{2}}.
\end{align}

Since $h<1$, the preceding inequality simplifies to
\begin{align}\label{eq:sum Ym-Yzm}
\left( \sum_{m=1}^{M} \big|\widetilde{\mathbf{Y}}_m-\mathbf{Y}(\mathbf{z}_m)\big|^{2} \right)^{\!\frac{1}{2}}
= \mathcal{O}\!\left(
\mathcal{P}^{4}\,
a^{-\frac{(1-h)(9-2\epsilon)}{9(3-\epsilon)}}\,
\|\bg\|_{\mathbb{H}^{-1/2}(\partial\Omega)}
\right).
\end{align}

The above estimate verifies the convergence of the discrete algebraic system to the continuous LSE.

\subsection{Completion of the Proof of Theorem \ref{thm:N--D}}\label{subsec:4.4}

We define the quantity
\begin{align}\label{eq:J-def}
\mathsf{J}
  &:= \omega^2 \rho_1 \,\langle \bv^{\bg}; \bu^{\bff}\rangle_{\mathbb{L}^2(D)}
   \;+\; \mathcal{P}^2 \,\langle \mathbf{q}^{\bg}; \bu^{\bff}\rangle_{\mathbb{L}^2(\Omega)}
   \;-\; \omega^2 \,\langle \rho(\bx) \bv^{\bg}; \bu^{\bff}\rangle_{\mathbb{L}^2(D)} \notag\\
  &= \sum_{j=1}^M \Big( \omega^2 \rho_1 \,\langle \bv^{\bg}_j; \bu^{\bff}\rangle_{\mathbb{L}^2(D_j)}
      + \mathcal{P}^2 \,\langle \mathbf{q}^{\bg}_j; \bu^{\bff}\rangle_{\mathbb{L}^2(\Omega_j)} \Big)
      \;-\; \omega^2 \,\langle \rho(\bx) \bv^{\bg}; \bu^{\bff}\rangle_{\mathbb{L}^2(D)} . 
\end{align}

Using the Taylor expansion of $\bu^{\bff}$ around the cell centers $\bz_j$,
\[
\bu^{\bff}(\bx)=\bu^{\bff}(\bz_j)+\int_0^1 \nabla \bu^{\bff}\!\big(\bz_j+t(\bx-\bz_j)\big)\!\cdot(\bx-\bz_j)\,\rmd t ,
\]
we get
\begin{align}\label{eq:J-expansion}
\mathsf{J}
  &= \sum_{j=1}^M \bu^{\bff}(\bz_j)
     \Bigg( \omega^2 \rho_1 \int_{D_j} \bv^{\bg}_j(\bx)\,\rmd \bx
             \;+\; \mathcal{P}^2 \int_{\Omega_j} \mathbf{q}^{\bg}_j(\bx)\,\rmd \bx \Bigg)
     \;+\; \mathsf{Error} , 
\end{align}
where
\begin{align}\label{eq:Err-def}
\mathsf{Error}
  &:= \omega^2 \rho_1 \sum_{j=1}^M
      \int_{D_j} \bv^{\bg}_j(\bx)
      \int_{0}^{1} \nabla \bu^{\bff}\!\big(\bz_j + t(\bx - \bz_j)\big)\!\cdot(\bx - \bz_j)\,\rmd t\,\rmd \bx \notag\\
  &\quad\; + \mathcal{P}^2 \sum_{j=1}^M
      \int_{\Omega_j} \mathbf{q}^{\bg}_j(\bx)
      \int_{0}^{1} \nabla \bu^{\bff}\!\big(\bz_j + t(\bx - \bz_j)\big)\!\cdot(\bx - \bz_j)\,\rmd t\,\rmd \bx \notag\\
  &\quad\;-\; \omega^2 \,\langle \rho(\bx) \bv^{\bg}; \bu^{\bff}\rangle_{\mathbb{L}^2(D)} .
\end{align}

Using $\rho_1 \sim a^{-2}$, $|\Omega_j| \sim  a^{1-h}$, and Cauchy--Schwarz in each term, we get
\begin{align}\label{eq:Error estimate}
|\mathsf{Error}|
 &\le  a^{-1} \sum_{j=1}^M  \|\bv^{\bg}_j\|_{\mathbb{L}^2(D_j)}\big\|\nabla \bu^{\bff}\big\|_{\mathbb{L}^2(D_j)}
 + a^{\frac{1-h}{3}} \mathcal{P}^2 \sum_{j=1}^M \|\mathbf{q}^{\bg}_j\|_{\mathbb{L}^2(\Omega_j)} \big\|\nabla \bu^{\bff} \big\|_{\mathbb{L}^2(\Omega_j)} \notag \\
 & \quad +  
 \|\bv^{\bg}\|_{\mathbb{L}^2(D)}\|\bu^{\bff}\|_{\mathbb{L}^2(D)}  \notag \\
 &\lesssim
   a^{-1}\|\bv^{\bg}\|_{\mathbb{L}^2(D)}\|\nabla \bu^{\bff}\|_{\mathbb{L}^2(D)}
   + a^{\frac{1-h}{3}}\mathcal{P}^2 \|\mathbf{q}^{\bg}\|_{\mathbb{L}^2(\Omega)}\|\nabla \bu^{\bff}\|_{\mathbb{L}^2(\Omega)}
   +\|\bv^{\bg}\|_{\mathbb{L}^2(D)}\|\bu^{\bff}\|_{\mathbb{L}^2(D)} \notag \\
 & \overset{\eqref{eq:uf L2 norm finally}}{\underset{\eqref{eq:nabla uf L2 norm}}{\lesssim}} a^{\frac{3h-2}{2}} \|\bv^\bg\|_{\mathbb{L}^2(D)} \| \bff \|_{\mathbb{H}^{-1/2}(\partial\Omega)}
 + a^{\frac{1-h}{3}} \mathcal{P}^2 \| \mathbf{q}^\bg \|_{\mathbb{L}^2(\Omega)} \|\nabla \bu^\bff \|_{\mathbb{L}^2(\Omega)}.
\end{align}

By the mapping property of the single-layer operator $\mathbf{S}: \bbH^{-1/2}(\partial\Omega) \to \bbH^1(D)$ applied to $\bu^{\bff}=\bS(\bff)$,
\begin{align}\label{eq:grad-bound}
\|\nabla \bu^{\bff}\|_{\mathbb{L}^2(\Omega)} =\mathcal{O}\left( \|\bff\|_{\bbH^{-1/2}(\partial\Omega)} \right).
\end{align}
Then, substituting \eqref{eq:vg L2 norm lesssim g} and \eqref{eq:grad-bound} into \eqref{eq:Error estimate} gives:
\begin{align}\label{eq:thm h from equation}
|\mathsf{Error}| &\lesssim \| \bff \|_{\mathbb{H}^{-1/2}(\partial\Omega)} \| \bg \|_{\mathbb{H}^{-1/2}(\partial\Omega)} \left( a^{(7h-1)/6} + a^{(1-h)/3} \mathcal{P}^2 \| \mathbf{q}^\bg \|_{\mathbb{L}^2(\Omega)} \right).  
\end{align}

Our next step is to find a bound for \( \mathbf{q}^\bg(\cdot) \) in terms of \( \bg(\cdot) \). As \( \mathbf{q}^\bg(\cdot) \) solves \eqref{eq:system qg}, it meets the integral relation
\begin{align}\label{eq:qg satisfies equation}
\mathbf{q}^\bg(\bx) + \mathcal{P}^2 \mathcal{N}(\mathbf{q}^\bg)(\bx) = \mathbf{S}(\bg)(\bx),  \qquad \text{for } \bx \in \Omega,
\end{align}
where \( \mathcal{N}(\cdot) \) represents the Newtonian operator as defined in \eqref{eq:Newton potential in Omega}. By combining \eqref{eq:Yz LSE def}, \eqref{eq:Y H1norm lesssim g}, and \eqref{eq:qg satisfies equation}, we can derive that
\begin{align}\label{eq:ug L2norm lesssim gH}
\| \mathbf{q}^\bg \|_{\mathbb{L}^2(\Omega)} \lesssim \|\bg \|_{\mathbb{H}^{-1/2}(\partial\Omega)}. 
\end{align}
Finally, by inserting  \eqref{eq:ug L2norm lesssim gH} into \eqref{eq:Error estimate} and taking advantage of the condition \( h > \frac{1}{3} \), we can establish that
\begin{align}\label{eq:Err-final}
|\mathsf{Error}|
 \;\lesssim\; a^{\frac{1-h}{3}}\mathcal{P}^2\,\|\bff\|_{\bbH^{-\frac{1}{2}}(\partial\Omega)}\|\bg\|_{\bbH^{-\frac{1}{2}}(\partial\Omega)}. 
\end{align}

Returning to \eqref{eq:J-expansion}, we obtain
\begin{align}
\mathsf{J}
&= \sum_{j=1}^{M} \bu^{\bff}(\bz_j)\left[ \omega^2 \rho_1 \int_{D_j} \bv^{\bg}_j(\bx)\,\rmd\bx
      + \mathcal{P}^2 \int_{\Omega_j} \mathbf{q}^{\bg}_j(\bx)\,\rmd\bx \right] \notag\\
&\quad + \mathcal{O}\!\left( a^{\frac{1-h}{3}}\,\mathcal{P}^2\,\|\bff\|_{\bbH^{-\frac{1}{2}}(\partial\Omega)}\,\|\bg\|_{\bbH^{-\frac{1}{2}}(\partial\Omega)} \right).
\end{align}

Note that \( \mathbf{q}^\bg(\cdot) \) (the solution of \eqref{eq:system qg}) coincides with the LSE solution given in \eqref{eq:Yz LSE def}; that is,
\( \mathbf{q}^\bg(\cdot) = \mathbf{Y}(\cdot) \) in \( \Omega \). Hence
\begin{align}
\int_{\Omega_j} \mathbf{q}^{\bg}_j(\bx)\,\rmd\bx
= \int_{\Omega_j} \mathbf{Y}(\bx)\,\rmd\bx
= \mathbf{Y}(\bz_j)\,|\Omega_j|
  + \int_{\Omega_j}\!\int_{0}^{1}  \nabla \mathbf{Y} \big(\bz_j+t(\bx-\bz_j)\big) \cdot (\bx-\bz_j) \,\rmd t\,\rmd\bx .
\end{align}

Therefore,
\begin{align}\label{eq:J=sum ufzj}
\mathsf{J}
&= \sum_{j=1}^{M} \bu^{\bff}(\bz_j)\left[ \omega^2 \rho_1 \int_{D_j} \bv^{\bg}_j(\bx)\,\rmd\bx
      + \mathcal{P}^2\,\mathbf{Y}(\bz_j)\,|\Omega_j| \right] \notag\\
&\quad + \mathcal{P}^2 \sum_{j=1}^{M} \bu^{\bff}(\bz_j)
      \int_{\Omega_j}\!\int_{0}^{1} \nabla \mathbf{Y}\!\big(\bz_j+t(\bx-\bz_j)\big)  \cdot (\bx-\bz_j) \,\rmd t\,\rmd\bx \notag\\
&\quad + \mathcal{O}\!\left( a^{\frac{1-h}{3}}\,\mathcal{P}^2\,\|\bff\|_{\bbH^{-\frac{1}{2}}(\partial\Omega)}\,\|\bg\|_{\bbH^{-\frac{1}{2}}(\partial\Omega)} \right). 
\end{align}

We estimate the second term on the right-hand side by setting
\begin{align}\label{eq:mathsf J2 lesssim}
\mathsf{J}_2
&:= \mathcal{P}^2 \sum_{j=1}^{M} \bu^{\bff}(\bz_j)
      \int_{\Omega_j}\!\int_{0}^{1}  \nabla \mathbf{Y}\!\big(\bz_j+t(\bx-\bz_j)\big)  \cdot (\bx-\bz_j) \,\rmd t\,\rmd\bx, \notag\\
|\mathsf{J}_2|
&\lesssim \mathcal{P}^2 \sum_{j=1}^{M} \big|\bu^{\bff}(\bz_j)\big|\,
          \left| \int_{\Omega_j}\!\int_{0}^{1}  \nabla \mathbf{Y}\!\big(\bz_j+t(\bx-\bz_j)\big)  \cdot (\bx-\bz_j)  \rmd t\,\rmd\bx \right| \notag\\
&\le \mathcal{P}^2 \left(\sum_{j=1}^{M} \big|\bu^{\bff}(\bz_j)\big|^2\right)^{\frac{1}{2}}
                 \left(\sum_{j=1}^{M} \left| \int_{\Omega_j}\!\int_{0}^{1}
                       \nabla \mathbf{Y}\!\big(\bz_j+t(\bx-\bz_j)\big)  \cdot (\bx-\bz_j) \,\rmd t\,\rmd\bx \right|^2\right)^{\frac{1}{2}} \notag\\
&= \mathcal{O}\!\left( \mathcal{P}^2
          \left(\sum_{j=1}^{M} \big|\bu^{\bff}(\bz_j)\big|^2\right)^{\frac{1}{2}}
          a^{\frac{5(1-h)}{6}}\,\|\nabla \mathbf{Y}\|_{\mathbb{L}^2(\Omega)} \right).
\end{align}

We next estimate \( \sum_{j=1}^{M} \big|\bu^{\bff}(\bz_j)\big|^2 \). Recall that \( \bu^{\bff}(\cdot) \) solves \eqref{eq:system}. Introduce
\( \bu^{\bff}_0(\cdot) \) as the solution of
\begin{align}\label{eq:u0f satisfies system}
\begin{cases}
(\mathcal{L}_{\lambda,\mu}+\omega^2) \bu^{\bff}_0 = \mathbf{0} & \text{in } \Omega,\\[2pt]
\partial_\nu \bu^{\bff}_0 = \bff & \text{on } \partial\Omega,
\end{cases} 
\end{align}
and subtract \eqref{eq:system} from \eqref{eq:u0f satisfies system} to obtain
\begin{align}
\begin{cases}
(\mathcal{L}_{\lambda,\mu}+\omega^2 \rho(\bx))\big(\bu^{\bff}-\bu^{\bff}_0\big)
   = \omega^2\,(1-\rho(\bx))\,\bu^{\bff}_0 & \text{in } \Omega,\\[2pt]
\partial_\nu\big(\bu^{\bff}-\bu^{\bff}_0\big) = \mathbf{0} & \text{on } \partial\Omega.
\end{cases}
\end{align}
Hence, for \( \bz\in\Omega \),
\begin{align}
\big(\bu^{\bff}-\bu^{\bff}_0\big)(\bz)
= \omega^2 \int_{\Omega} \Gamma(\bz,\by)\,(1-\rho(\by))\,\bu^{\bff}_0(\by)\,\rmd\by,
\end{align}
where \( \Gamma(\cdot,\cdot) \) is the Neumann Green tensor of \eqref{eq:Gamma def}. Taking the modulus yields
\begin{align}\label{eq:uf-u0f lesssim u0f}
\big|\big(\bu^{\bff}-\bu^{\bff}_0\big)(\bz)\big|
&\le \|\omega^2(1-\rho)\|_{L^\infty(\Omega)}\,\|\Gamma(\bz,\cdot)\|_{\mathbb{L}^2(\Omega)}\,
      \|\bu^{\bff}_0\|_{\mathbb{L}^2(\Omega)} \notag\\
&\lesssim \|\bu^{\bff}_0\|_{\mathbb{L}^2(\Omega)}.  
\end{align}
The last inequality follows from the \( \mathbb{L}^2(\Omega) \)-integrability of \( \Gamma(\bz,\cdot) \), uniformly in \(\bz\in\Omega\), and the boundedness of \( \|\rho\|_{L^\infty(\Omega)} \). From the well-posedness of \eqref{eq:u0f satisfies system},
\begin{align}\label{eq:u0f lesssim fH norm}
\|\bu^{\bff}_0\|_{\mathbb{L}^2(\Omega)} \lesssim \|\bff\|_{\bbH^{-\frac{1}{2}}(\partial\Omega)}.  
\end{align}
Combining \eqref{eq:uf-u0f lesssim u0f} and \eqref{eq:u0f lesssim fH norm} gives
\begin{align}\label{eq:uf-u0f}
\big|\big(\bu^{\bff}-\bu^{\bff}_0\big)(\bz)\big|
\lesssim \|\bff\|_{\bbH^{-\frac{1}{2}}(\partial\Omega)}.  
\end{align}

\begin{lem}\label{lem:elastic-pointwise-L2}
Let $\bu_0^{\bff}$ be the solution to \eqref{eq:u0f satisfies system}. For the ball $B_r(\bz_j)$ centered at $\bz_j$ with radius $r$, contained in the cube $\Omega_j$, we have
\begin{align}\label{eq:elastic-pointwise-L2}
\big|\bu_0^{\bff}(\bz_j)\big|
\lesssim |B_r(\bz_j)|^{-\frac{1}{2}} \,\|\bu_0^{\bff}\|_{\mathbb{L}^2(B_r(\bz_j))}.
\end{align}
\end{lem}

\begin{proof}
Consider the longitudinal/transverse decomposition
\begin{align}
\bu_0^{\bff}=\bu_p^{\bff}+\bu_s^{\bff}
\end{align}
with
\begin{align}
\bu_p^{\bff}=\frac{1}{k_p^2}\nabla(\nabla\cdot\bu_0^{\bff}),\qquad
\bu_s^{\bff}=\bu_0^{\bff}-\bu_p^{\bff}. \nonumber
\end{align}
Then one verifies that
\begin{align}\label{eq:ps-systems}
(\Delta + k_{p}^{2}) \, &\mathbf{u}_{p} = \mathbf{0}, \quad \nabla \times \mathbf{u}_{p} = \mathbf{0}, \notag \\
(\Delta + k_{s}^{2}) \, &\mathbf{u}_{s} = \mathbf{0}, \quad \nabla \cdot \mathbf{u}_{s} = 0.
\end{align}
See \cite[formula (36), p.~288]{CH62}. We have the following mean-value integral formulas:
\begin{align}
\frac{1}{|\partial B_{r}(\bz_j)|} \int_{\partial B_{r}(\bz_j)} \mathbf{u}_{p}^{\mathbf{f}}(\bx)\, \mathrm{d}\sigma(\bx)
&= j_{0}(k_{p} r)\, \mathbf{u}_{p}^{\mathbf{f}}(\bz_j), \notag \\
\frac{1}{|\partial B_{r}(\bz_j)|} \int_{\partial B_{r}(\bz_j)} \mathbf{u}_{s}^{\mathbf{f}}(\bx)\, \mathrm{d}\sigma(\bx)
&= j_{0}(k_{s} r)\, \mathbf{u}_{s}^{\mathbf{f}}(\bz_j),
\end{align}
where $j_0(t)=\dfrac{\sin t}{t}$ denotes the spherical Bessel function of order zero. Hence,
\begin{equation}\label{eq:elastic-surface-mean}
\frac{1}{|\partial B_{r}(\bz_j)|} \int_{\partial B_{r}(\bz_j)} \mathbf{u}^{\mathbf{f}}_{0}(\bx)\, \mathrm{d}\sigma(\bx)
= j_{0}(k_{p} r)\, \mathbf{u}_{p}^{\mathbf{f}}(\bz_j) + j_{0}(k_{s} r)\, \mathbf{u}_{s}^{\mathbf{f}}(\bz_j).
\end{equation}
Rearranging gives, for each $\alpha\in\{p,s\}$,
\[
\bu_{\alpha}^{\bff}(\bz_j)
= \frac{k_\alpha^{3}}{4\pi\big(\sin(k_\alpha r)-k_\alpha r\cos(k_\alpha r)\big)}
\int_{B_r(\bz_j)} \bu_{\alpha}^{\bff}(\bx)\,\mathrm{d}\bx.
\]
Summing the two components $\bu_0^{\bff}=\bu_p^{\bff}+\bu_s^{\bff}$ yields
\begin{align}
\bu_0^{\bff}(\bz_j)
= \sum_{\alpha\in\{p,s\}}
\frac{k_\alpha^{3}}{4\pi\big(\sin(k_\alpha r)-k_\alpha r\cos(k_\alpha r)\big)}
\int_{B_r(\bz_j)} \bu_\alpha^{\bff}(\bx)\,\mathrm{d}\bx. \nonumber
\end{align}
Applying the Cauchy--Schwarz inequality yields
\begin{align}\label{eq:step1-CS}
\big|\bu_0^{\bff}(\bz_j)\big|
\le \sum_{\alpha\in\{p,s\}}
\frac{k_\alpha^{3}}{4\pi\big|\sin(k_\alpha r)-k_\alpha r\cos(k_\alpha r)\big|}
\,|B_r(\bz_j)|^{1/2}\, \|\bu_\alpha^{\bff}\|_{\mathbb{L}^2(B_r(\bz_j))}.
\end{align}
For $\alpha \in \{p,s\}$, as $r \to 0^{+}$,
\begin{equation}\label{eq:elastic-small-r}
4\pi\big(\sin(k_{\alpha} r) - k_{\alpha} r \cos(k_{\alpha} r)\big)
= k_{\alpha}^{3}|B_{r}(\bz_j)| - \frac{2\pi}{15} k_{\alpha}^{5} r^{5} + \mathcal{O}(r^{7}).
\end{equation}
Plugging \eqref{eq:elastic-small-r} into \eqref{eq:step1-CS} gives
\begin{align}\label{eq:step2-CS}
\big|\bu_0^{\bff}(\bz_j)\big|
&\le \sum_{\alpha\in\{p,s\}}
\frac{1}{|B_{r}(\bz_j)|+\mathcal{O}(r^5)}
\,|B_r(\bz_j)|^{1/2}\, \|\bu_\alpha^{\bff}\|_{\mathbb{L}^2(B_r(\bz_j))}  \notag \\
&\lesssim |B_r(\bz_j)|^{-\frac{1}{2}} \,\|\bu_0^{\bff}\|_{\mathbb{L}^2(B_r(\bz_j))}.
\end{align}
This proves \eqref{eq:elastic-pointwise-L2}. Here and throughout, $|B_r(\bz_j)|$ and $|\partial B_r(\bz_j)|$ denote the volume and surface measure, respectively, of the ball and its boundary in $\mathbb{R}^3$.
\end{proof}

Therefore,
\begin{align}
\sum_{j=1}^{M} \big|\bu^{\bff}(\bz_j)\big|^2
&= \sum_{j=1}^{M} \big|\bu^{\bff}_0(\bz_j)
               + \big(\bu^{\bff}(\bz_j)-\bu^{\bff}_0(\bz_j)\big)\big|^2
\lesssim \sum_{j=1}^{M} \big|\bu^{\bff}_0(\bz_j)\big|^2
      + \sum_{j=1}^{M} \big|\bu^{\bff}(\bz_j)-\bu^{\bff}_0(\bz_j)\big|^2 .
\end{align}
Using \eqref{eq:uf-u0f} and \eqref{eq:elastic-pointwise-L2}, we obtain
\begin{align}
\sum_{j=1}^{M} \big|\bu^{\bff}(\bz_j)\big|^2
&\lesssim \sum_{j=1}^{M} |B_r(\bz_j)|^{-1}\,\|\bu^{\bff}_0\|_{\mathbb{L}^2(B_r(\bz_j))}^2
       + M\,\|\bff\|_{\bbH^{-\frac{1}{2}}(\partial\Omega)}^2 \notag\\
&\lesssim |B_r(\bz_{j_0})|^{-1}\,\|\bu^{\bff}_0\|_{\mathbb{L}^2(\cup_{j=1}^{M}B_r(\bz_j))}^2
       + M\,\|\bff\|_{\bbH^{-\frac{1}{2}}(\partial\Omega)}^2 .
\end{align}
Since \( |B_r(\bz_j)|^{-1}\sim M \) and \( \cup_{j=1}^{M}B_r(\bz_j)\subset\Omega \),
\begin{align}\label{eq:sum ufzj}
\sum_{j=1}^{M} \big|\bu^{\bff}(\bz_j)\big|^2
&\overset{\eqref{eq:u0f lesssim fH norm}}{\lesssim}M\left(\|\bu^{\bff}_0\|_{\mathbb{L}^2(\Omega)}^2
                 + \|\bff\|_{\bbH^{-\frac{1}{2}}(\partial\Omega)}^2\right)
\lesssim M\,\|\bff\|_{\bbH^{-\frac{1}{2}}(\partial\Omega)}^2.  
\end{align}

Continuing from \eqref{eq:mathsf J2 lesssim} and using \eqref{eq:sum ufzj},
\begin{align}
|\mathsf{J}_2\|
&\lesssim \mathcal{P}^2\,M^{\frac{1}{2}}\,\|\bff\|_{\bbH^{-\frac{1}{2}}(\partial\Omega)}\,
          a^{\frac{5(1-h)}{6}}\,\|\nabla \mathbf{Y}\|_{\mathbb{L}^2(\Omega)} \notag\\
&\underset{\eqref{eq:Y H1norm lesssim g}}{\overset{\eqref{eq:M def}}{=}}  \mathcal{O}\!\left(
   \mathcal{P}^4\,a^{\frac{1-h}{3}}\,\|\bff\|_{\bbH^{-\frac{1}{2}}(\partial\Omega)}\,\|\bg\|_{\bbH^{-\frac{1}{2}}(\partial\Omega)} \right).
\end{align}
Hence, the estimate \eqref{eq:J=sum ufzj} becomes
\begin{align}\label{eq:J=sum ufzj+O}
\mathsf{J}
&= \sum_{j=1}^{M} \bu^{\bff}(\bz_j)\left[ \omega^2 \rho_1 \int_{D_j}\bv^{\bg}_j(\bx)\,\rmd\bx
     + \mathcal{P}^2\, \mathbf{Y}(\bz_j)\,|\Omega_j| \right] 
 + \mathcal{O}\!\left(
   a^{\frac{1-h}{3}}\,\mathcal{P}^2\,\|\bff\|_{\bbH^{-\frac{1}{2}}(\partial\Omega)}\,\|\bg\|_{\bbH^{-\frac{1}{2}}(\partial\Omega)} \right).
\end{align}

To relate \( \mathcal{P}^2 \) to the scattering coefficient \(\boldsymbol{\alpha}\), we invoke Lemma \ref{lem:alpha=-P2} and \eqref{eq:alpha=-P2a1-h in lemma}.

\begin{lem}\label{lem:alpha=-P2} 
The scattering coefficient $\boldsymbol{\alpha}$, introduced in \eqref{eq:alpha def}, obeys the asymptotic relation
\begin{align}\label{eq:alpha=-P2a1-h in lemma}
\boldsymbol{\alpha} = - \mathcal{P}^{2} a^{\,1-h} + \mathcal{O}(a), \qquad 0<h<1, 
\end{align}
where 
\begin{align}
    \mathcal{P}^{2} &:= -\frac{\left(\langle \mathcal{I}; \tilde{\mathbf{e}}_{n_{0}}\rangle_{\mathbb{L}^{2}(B)}\right)^{2}}
              {\lambda^{B}_{n_{0}}\,c_{n_{0}}}  .
\end{align}
Moreover, one has
\begin{align}\label{eq:Wm L2 norm}
\|\mathbf{W}_{m}\|_{\mathbb{L}^{2}(D_{m})} &\lesssim a^{-(2+h)}\,\|1\|_{\mathbb{L}^{2}(D_{m})}
= \mathcal{O} \left(a^{-\left(\tfrac{1}{2}+h\right)}\right).  
\end{align}
Here $\mathbf{W}_{m}(\cdot)$ denotes the solution to \eqref{eq:Wx satisfies equation}.
\end{lem}

\begin{proof}
    See Subsection \ref{subsec:alpha=proof}.
\end{proof}

Since \( |\Omega_j|=a^{1-h} \), \eqref{eq:alpha=-P2a1-h in lemma} gives \( \mathcal{P}^2 \mathbf{Y}(\bz_j)|\Omega_j| = -\boldsymbol{\alpha}\,\mathbf{Y}(\bz_j) \). Moreover, from \eqref{eq:Ym def},
\[
\omega^2\rho_1\boldsymbol{\beta}_j\int_{D_j}\bv^{\bg}_j(\bx)\,\rmd\bx = \boldsymbol{\alpha}\,\mathbf{Y}_j .
\]
Thus, \eqref{eq:J=sum ufzj+O}  becomes
\begin{align}\label{eq:J=sum ufzj+O-2}
\mathsf{J}
&= \sum_{j=1}^{M} \bu^{\bff}(\bz_j)\,\boldsymbol{\alpha} \left[\frac{1}{\boldsymbol{\beta}_j} \mathbf{Y}_j - \mathbf{Y}(\bz_j)\right]
   + \mathcal{O}\!\left(
     a^{\frac{1-h}{3}}\,\mathcal{P}^2\,\|\bff\|_{\bbH^{-\frac{1}{2}}(\partial\Omega)}\,\|\bg\|_{\bbH^{-\frac{1}{2}}(\partial\Omega)} \right) \notag\\
&= \sum_{j=1}^{M} \bu^{\bff}(\bz_j)\,\boldsymbol{\alpha}\,\big[\widetilde{\mathbf{Y}}_j - \mathbf{Y}(\bz_j)\big]
 + \sum_{j=1}^{M} \bu^{\bff}(\bz_j)\,\boldsymbol{\alpha}\left[\frac{1}{\boldsymbol{\beta}_j} \mathbf{Y}_j - \widetilde{\mathbf{Y}}_j\right] \notag\\
&\quad
 + \mathcal{O}\!\left(
     a^{\frac{1-h}{3}}\,\mathcal{P}^2\,\|\bff\|_{\bbH^{-\frac{1}{2}}(\partial\Omega)}\,\|\bg\|_{\bbH^{-\frac{1}{2}}(\partial\Omega)} \right). 
\end{align}

Define and estimate the second term on the right-hand side:
\begin{align}
\mathsf{Q}_2
&:= \sum_{j=1}^{M} \bu^{\bff}(\bz_j)\,\boldsymbol{\alpha}\left[\frac{1}{\boldsymbol{\beta}_j} \mathbf{Y}_j - \widetilde{\mathbf{Y}}_j\right] \notag \\
&\stackrel{\eqref{eq:beta=1-int}}{=} \sum_{j=1}^{M} \bu^{\bff}(\bz_j)\,\boldsymbol{\alpha}\,(\mathbf{Y}_j-\widetilde{\mathbf{Y}}_j)
 + \sum_{j=1}^{M} \bu^{\bff}(\bz_j)\,\boldsymbol{\alpha}\,
   \frac{\int_{D_j} \mathbf{W}_j(\bx)\,\mathcal{R}(\bx,\bz_j)\,\rmd\bx}{\boldsymbol{\beta}_j}\, \mathbf{Y}_j. \notag
\end{align}
Using \eqref{eq:Wm R int Dm}, \eqref{eq:betam=}, and \eqref{eq:alpha=-P2a1-h in lemma}, we deduce
\begin{align}\label{eq:mathsfQ2 eatimate}
|\mathsf{Q}_2|
&\lesssim \mathcal{P}^2 a^{1-h}\left(\sum_{j=1}^{M}\big|\bu^{\bff}(\bz_j)\big|^2\right)^{\frac{1}{2}}
\left[
  \left(\sum_{j=1}^{M}|\mathbf{Y}_j-\widetilde{\mathbf{Y}}_j|^2\right)^{\frac{1}{2}}
 + a^{\frac{2(1-h)}{3}}\left(\sum_{j=1}^{M}|\mathbf{Y}_j|^2\right)^{\frac{1}{2}}
\right] \notag\\
&\stackrel{\eqref{eq:sum ufzj}}{\lesssim} \mathcal{P}^2 a^{1-h} M^{\frac{1}{2}}\,\|\bff\|_{\bbH^{-\frac{1}{2}}(\partial\Omega)}
\left[
  \left(\sum_{j=1}^{M}|\mathbf{Y}_j-\widetilde{\mathbf{Y}}_j|^2\right)^{\frac{1}{2}}
 + a^{\frac{2(1-h)}{3}}\left(\sum_{j=1}^{M}|\mathbf{Y}_j|^2\right)^{\frac{1}{2}}
\right] \notag\\
&\stackrel{\eqref{eq:sum Ym2 lesssim}}{\lesssim} \mathcal{P}^2 a^{1-h} M^{\frac{1}{2}}\,\|\bff\|_{\bbH^{-\frac{1}{2}}(\partial\Omega)}
\left[ \left(\sum_{j=1}^{M}|\mathbf{Y}_j-\widetilde{\mathbf{Y}}_j|^2\right)^{\frac{1}{2}}
 + a^{\frac{4(1-h)}{3}}\,\|\bg\|_{\bbH^{-\frac{1}{2}}(\partial\Omega)}
\right]. 
\end{align}

Subtracting \eqref{eq:The algebraic system} from \eqref{eq:Ym LSE} gives the algebraic system
\begin{align}
(\mathbf{Y}_m-\widetilde{\mathbf{Y}}_m)
+ \sum_{\substack{j=1\\ j\ne m}}^{M} \Gamma(\bz_m,\bz_j)\,\mathcal{P}^2 a^{1-h}\,\frac{1}{\boldsymbol{\beta}_j}\,(\mathbf{Y}_j-\widetilde{\mathbf{Y}}_j)
= \omega^2\rho_1\,\frac{\mathrm{Error}_m}{\boldsymbol{\alpha}}.
\end{align}
By Lemma \ref{lem:The algebraic system is invertible} and the fact \( \boldsymbol{\alpha}\sim a^{1-h} \) and $\rho_1\sim a^{-2}$, we obtain
\begin{align}
\left(\sum_{j=1}^{M}|\mathbf{Y}_j-\widetilde{\mathbf{Y}}_j|^2\right)^{\frac{1}{2}}
&\lesssim a^{h-3}\left(\sum_{j=1}^{M}|\mathrm{Error}_j|^2\right)^{\frac{1}{2}}
\stackrel{ \eqref{eq:sum Errorm}}{\lesssim} a^{3+h}\,\|\bg\|_{\bbH^{-\frac{1}{2}}(\partial\Omega)}.
\end{align}
Substituting into \eqref{eq:mathsfQ2 eatimate}  and using \( M\sim a^{h-1} \), we get
\begin{align}
|\mathsf{Q}_2|
\lesssim \mathcal{P}^2\,a^{\frac{11(1-h)}{6}}\,\|\bff\|_{\bbH^{-\frac{1}{2}}(\partial\Omega)}\,\|\bg\|_{\bbH^{-\frac{1}{2}}(\partial\Omega)}.
\end{align}

Taking the modulus in  \eqref{eq:J=sum ufzj+O-2} and applying the above estimate, we get
\begin{align}
|\mathsf{J}|
&\lesssim |\boldsymbol{\alpha}|\left(\sum_{j=1}^{M}|\widetilde{\mathbf{Y}}_j-\mathbf{Y}(\bz_j)|^2\right)^{\frac{1}{2}}  \left(\sum_{j=1}^{M}\big|\bu^{\bff}(\bz_j)\big|^2\right)^{\frac{1}{2}}
 + a^{\frac{1-h}{3}}\,\mathcal{P}^2\,\|\bff\|_{\bbH^{-\frac{1}{2}}(\partial\Omega)}\,\|\bg\|_{\bbH^{-\frac{1}{2}}(\partial\Omega)} \notag\\
&\stackrel{\eqref{eq:sum ufzj}}{\lesssim} |\boldsymbol{\alpha}|\,M^{\frac{1}{2}}
   \left(\sum_{j=1}^{M}|\widetilde{\mathbf{Y}}_j-\mathbf{Y}(\bz_j)|^2\right)^{\frac{1}{2}}\,\|\bff\|_{\bbH^{-\frac{1}{2}}(\partial\Omega)}
 + a^{\frac{1-h}{3}}\,\mathcal{P}^2\,\|\bff\|_{\bbH^{-\frac{1}{2}}(\partial\Omega)}\,\|\bg\|_{\bbH^{-\frac{1}{2}}(\partial\Omega)}.
\end{align}

Since \( M\sim a^{h-1} \) and \( \boldsymbol{\alpha}\sim \mathcal{P}^2 a^{1-h} \) (see \eqref{eq:alpha=-P2a1-h in lemma}), applying \eqref{eq:sum Ym-Yzm} gives
\begin{align}\label{eq:mathsf J lesssim O}
|\mathsf{J}|
&\lesssim a^{\frac{(1-h)(9-5\epsilon)}{18(3-\epsilon)}}\,\mathcal{P}^6\,
          \|\bff\|_{\bbH^{-\frac{1}{2}}(\partial\Omega)}\,\|\bg\|_{\bbH^{-\frac{1}{2}}(\partial\Omega)}
 + a^{\frac{1-h}{3}}\,\mathcal{P}^2\,
          \|\bff\|_{\bbH^{-\frac{1}{2}}(\partial\Omega)}\,\|\bg\|_{\bbH^{-\frac{1}{2}}(\partial\Omega)} \notag\\
&= \mathcal{O}\!\left( a^{\frac{(1-h)(9-5\epsilon)}{18(3-\epsilon)}}\,\mathcal{P}^6\,
          \|\bff\|_{\bbH^{-\frac{1}{2}}(\partial\Omega)}\,\|\bg\|_{\bbH^{-\frac{1}{2}}(\partial\Omega)} \right).  
\end{align}

Finally, combining \eqref{eq:LambdaD-LambdaP}, \eqref{eq:J-def}, and \eqref{eq:mathsf J lesssim O}, we obtain
\begin{align}
\big| \langle (\Lambda_D - \Lambda_{\mathcal{P}})(\bff)\,;\,\bg \rangle_{\bbH^{\frac{1}{2}}(\partial\Omega)\times\bbH^{-\frac{1}{2}}(\partial\Omega)} \big|
\lesssim a^{\frac{(1-h)(9-5\epsilon)}{18(3-\epsilon)}}\,\mathcal{P}^6\,
          \|\bff\|_{\bbH^{-\frac{1}{2}}(\partial\Omega)}\,\|\bg\|_{\bbH^{-\frac{1}{2}}(\partial\Omega)}.
\end{align}
Hence
\begin{align}
\|\Lambda_D - \Lambda_{\mathcal{P}}\|_{\mathcal{L}\big(\bbH^{-\frac{1}{2}}(\partial\Omega);\bbH^{\frac{1}{2}}(\partial\Omega)\big)}
\lesssim a^{\frac{(1-h)(9-5\epsilon)}{18(3-\epsilon)}}\,\mathcal{P}^6.
\end{align}
This proves \eqref{eq:Lambda D-Lambda P in thm} and concludes the proof of Theorem \ref{thm:N--D}.

\section{ Proofs of the auxiliary results}\label{sec:appendix}


To keep Section~\ref{sec:Theorem 1} concise, we collect here the proofs of several auxiliary results used in the preceding analysis.

{ 

\subsection{Properties of the Newtonian potential $\mathcal{N}^{\calP}$ defined in \eqref{eq:Newtonian potential NP in section2}}

 \begin{lem}\label{lem:NP-spectral}
For each $\calP$, let $\Gamma_\calP(\bx,\by)$ be the Neumann Green tensor of
$\mathcal{L}_{\lambda,\mu}-\calP^2$:
\[
\begin{cases}
(\mathcal{L}_{\lambda,\mu}-\calP^{2})\,\Gamma_{\calP}(\bx,\by)
= -\delta_{\by}(\bx)\,\mathcal{I} & \text{in }\Omega,\\[1mm]
\partial_{\nu_{\bx}} \Gamma_{\calP}(\bx,\by)=\mathbf{0} & \text{on }\partial\Omega,
\end{cases}
\]
and $\mathcal{L}_{\lambda,\mu}$ and $\partial_{\nu_{\bx}}$ are defined in
\eqref{eq:mathcal L def} and \eqref{eq:partial nu def}, respectively. Define the Newtonian potential
\[
    \mathcal{N}^{\calP}:\mathbb{L}^{2}(\Omega)^3\to \mathbb{H}^{2}(\Omega)^3,\qquad
    \mathcal{N}^{\calP}(\bg)(\bx)
    :=\int_{\Omega}\Gamma_{\calP}(\bx,\by)\,\bg(\by)\,{\rm d}\by.
\]
Then, for each $\calP$, the operator $\mathcal{N}^{\calP}$ is bounded, self-adjoint, positive, and compact.
\end{lem}

\begin{proof}
\textbf{Step 1: Representation via a non-negative self-adjoint operator; boundedness and compactness.}
Assume first that $\calP\neq 0$. (The case $\calP=0$ can be handled by the standard normalization modulo rigid motions.)

Set
\[
    \mathcal{A} := -\mathcal{L}_{\lambda,\mu}.
\]
Then $\mathcal{A}$ is a non-negative, self-adjoint, second-order elliptic operator \cite{McL00}
(with Neumann boundary condition) on $\mathbb{L}^2(\Omega)^3$. The Green tensor identity
\[
    (\mathcal{L}_{\lambda,\mu}-\calP^2)\Gamma_\calP(\cdot,\by)
    = -\delta_{\by}\mathcal{I}
\]
is equivalent to
\[
    ( \mathcal{A}+\calP^2)\Gamma_\calP(\cdot,\by)=\delta_{\by}\mathcal{I}.
\]
Hence, for any $\bg\in \mathbb{L}^2(\Omega)^3$ the Newtonian potential
$\bu:=\mathcal{N}^{\calP}\bg$ solves
\begin{equation}\label{eq:AP-u-g}
\begin{cases}
( \mathcal{A}+\calP^2)\bu = \bg & \text{in }\Omega,\\[1mm]
\partial_{\nu_{\bx}} \bu=\mathbf{0} & \text{on }\partial\Omega,
\end{cases}
\end{equation}
in the weak sense.

Consider the sesquilinear form
$a_\calP:\mathbb{H}^1(\Omega)^3\times \mathbb{H}^1(\Omega)^3\to\mathbb{C}$
\[
    a_\calP(\bu,\bv)
    := \int_\Omega \sigma(\bu):\nabla^s\overline{\bv}\,{\rm d}\bx
       + \calP^2\int_\Omega \bu\cdot\overline{\bv}\,{\rm d}\bx.
\]
By strong convexity of $(\lambda,\mu)$ and Korn's inequality there exists
$c_1>0$ such that
\begin{equation}\label{eq:aP-coercive-2}
    a_\calP(\bu,\bu)
    \ge c_1\bigl(\|\nabla\bu\|_{\mathbb{L}^2(\Omega)}^2+\|\bu\|_{\mathbb{L}^2(\Omega)}^2\bigr)
    \qquad\forall\,\bu\in \mathbb{H}^1(\Omega)^3,
\end{equation}
with $c_1$ independent of $\calP$ on compact subsets of $\mathbb{R}\setminus\{0\}$. Moreover
\[
    |a_\calP(\bu,\bv)|
    \lesssim \|\bu\|_{\mathbb{H}^1(\Omega)}\|\bv\|_{\mathbb{H}^1(\Omega)}.
\]

Given $\bg\in \mathbb{L}^2(\Omega)^3$, the weak formulation of
\eqref{eq:AP-u-g} reads
\begin{equation}\label{eq:variational-AP}
    a_\calP(\bu,\bv)
    = \int_\Omega \bg\cdot\overline{\bv}\,{\rm d}\bx
    \qquad\forall\, \bv\in \mathbb{H}^1(\Omega)^3.
\end{equation}
By Lax--Milgram and \eqref{eq:aP-coercive-2}, there exists a unique
$\bu\in \mathbb{H}^1(\Omega)^3$ solving \eqref{eq:variational-AP} and
\[
    \|\bu\|_{\mathbb{H}^1(\Omega)}\lesssim \|\bg\|_{\mathbb{L}^2(\Omega)},
\]
with a constant uniform for $\calP$ in compact subsets of $\mathbb{R}\setminus\{0\}$. Standard elliptic
regularity for the strongly elliptic system $ \mathcal{A}+\calP^2$ with Neumann boundary
condition implies
\[
    \bu\in \mathbb{H}^2(\Omega)^3,\qquad
    \|\bu\|_{\mathbb{H}^2(\Omega)}\lesssim \|\bg\|_{\mathbb{L}^2(\Omega)}.
\]
Therefore $\mathcal{N}^{\calP}:\mathbb{L}^2(\Omega)^3\to \mathbb{H}^2(\Omega)^3$ is bounded. Since the
embedding $\mathbb{H}^2(\Omega)^3\hookrightarrow \mathbb{L}^2(\Omega)^3$ is compact, the
operator $\mathcal{N}^{\calP}$ (viewed on $\mathbb{L}^2(\Omega)^3$) is bounded and compact.

\medskip\noindent
\textbf{Step 2: Self-adjointness and positivity.}
Let $\bg,\bh\in \mathbb{L}^2(\Omega)^3$ and set
\[
    \bu:=\mathcal{N}^{\calP}\bg,\qquad \bv:=\mathcal{N}^{\calP}\bh.
\]
Then, by definition,
\[
\begin{cases}
(\mathcal{A}+\calP^2)\bu=\bg,\\
(\mathcal{A}+\calP^2)\bv=\bh,
\end{cases}
\qquad
\partial_{\nu_{\bx}}\bu
=\partial_{\nu_{\bx}}\bv=\mathbf{0}\quad\text{on }\partial\Omega.
\]
Using \eqref{eq:variational-AP} with test function $\bv$ and $\bu$, respectively,
we obtain
\[
    (\bg,\bv)_{\mathbb{L}^2}
    = a_\calP(\bu,\bv)
    = \overline{a_\calP(\bv,\bu)}
    = \overline{(\bh,\bu)_{\mathbb{L}^2}}
    = (\bu,\bh)_{\mathbb{L}^2}.
\]
Since $\bu=\mathcal{N}^{\calP}\bg$ and $\bv=\mathcal{N}^{\calP}\bh$, this identity is equivalent to
\[
    (\mathcal{N}^{\calP}\bg,\bh)_{\mathbb{L}^2}
    = (\bg,\mathcal{N}^{\calP}\bh)_{\mathbb{L}^2}
    \qquad\forall\,\bg,\bh\in \mathbb{L}^2(\Omega)^3,
\]
so $\mathcal{N}^{\calP}$ is self-adjoint on $\mathbb{L}^2(\Omega)^3$.

Taking $\bh=\bg$ and $\bv=\bu$ yields
\[
    (\mathcal{N}^{\calP}\bg,\bg)_{\mathbb{L}^2}
    = a_\calP(\bu,\bu)
    = \int_\Omega \sigma(\bu):\nabla^s\overline{\bu}\,{\rm d}\bx
      + \calP^2\int_\Omega |\bu|^2\,{\rm d}\bx.
\]
By \eqref{eq:aP-coercive-2},
\[
    (\mathcal{N}^{\calP}\bg,\bg)_{\mathbb{L}^2}\gtrsim \|\bu\|_{\mathbb{H}^1(\Omega)}^2\ge 0.
\]
Thus $\mathcal{N}^{\calP}$ is positive:
$(\mathcal{N}^{\calP}\bg,\bg)_{\mathbb{L}^2}\ge0$ for all $\bg\in \mathbb{L}^2(\Omega)^3$ and
$\calP$.

This completes the proof of the lemma.
\end{proof}

}

\subsection{Proof of Lemma \ref{lem:NP norm lem}} \label{subsec:NP norm proof}

From standard spectral considerations, the operator norm of the Neumann resolvent at the point \(\calP^2>0\) satisfies
\begin{align}
\big\| \calN^{\calP} \big\|_{\mathcal{L}(\mathbb{L}^2(\Omega);\mathbb{L}^2(\Omega))}
&= \big\| \mathcal{R}(\calP^2;\mathcal{L}_{\lambda,\mu}) \big\|_{\mathcal{L}(\mathbb{L}^2(\Omega);\mathbb{L}^2(\Omega))}
\le \frac{1}{\mathrm{dist}\!\big(\calP^2;\sigma(\mathcal{L}_{\lambda,\mu})\big)}.
\end{align}
For the Neumann Lam\'e operator \(\mathcal{L}_{\lambda,\mu}\) on \(\mathbb{L}^2(\Omega)\), the spectrum is real, discrete, and contained in \((-\infty,0]\); in particular, \(0\in\sigma(\mathcal{L}_{\lambda,\mu})\) (rigid motions), while the remaining eigenvalues accumulate at \(-\infty\).
Hence
\begin{align}
\mathrm{dist}\!\big(\calP^2;\sigma(\mathcal{L}_{\lambda,\mu})\big)=\calP^2,
\end{align}
and therefore
\begin{align}
\big\| \calN^{\calP} \big\|_{\mathcal{L}(\mathbb{L}^2(\Omega);\mathbb{L}^2(\Omega))}
\le \frac{1}{\calP^2}.
\end{align}
This yields the bound corresponding to \eqref{eq:NP norm}. To establish the estimate analogous to \eqref{eq:gamma NP norm}, fix an arbitrary \(\bff\in\mathbb{L}^2(\Omega)\) and set \(\widetilde{\bu}:=\calN^{\calP}(\bff)\).
By definition of the resolvent, \(\widetilde{\bu}\) solves
\begin{align}
(\mathcal{L}_{\lambda,\mu} - \calP^2 \mathcal{I}) \,\widetilde{\bu}= -\,\bff \quad \text{in }\Omega,
\qquad
\partial_\nu \widetilde{\bu}=\mathbf{0} \quad \text{on }\partial\Omega,
\end{align}
where \(\partial_\nu \widetilde{\bu}\), defined in \eqref{eq:partial nu def}, denotes the traction operator. Multiplying the equation by \(\widetilde{\bu}\) and applying Green's identity, we obtain the following energy identity 
\begin{align}
a(\widetilde{\bu},\widetilde{\bu})+\calP^2\|\widetilde{\bu}\|_{\mathbb{L}^2(\Omega)}^2
= (\bff,\widetilde{\bu})_{\mathbb{L}^2(\Omega)}
\le \|\bff\|_{\mathbb{L}^2(\Omega)}\,\|\widetilde{\bu}\|_{\mathbb{L}^2(\Omega)},
\end{align}
where
\begin{align}
    a(\bu,\bv):=\int_\Omega \big(2\mu\,\varepsilon(\bu):\varepsilon(\bv)+\lambda\,(\nabla\!\cdot\!\bu)(\nabla\!\cdot\!\bv)\big)\, \rmd \bx,
\end{align}
is the Lam\'e bilinear form. By strong ellipticity and Korn’s inequality there exists \(c_0>0\) (depending only on \(\lambda,\mu,\Omega\)) such that
\begin{align}
a(\widetilde{\bu},\widetilde{\bu}) \;\ge\; c_0\,\|\nabla \widetilde{\bu}\|_{\mathbb{L}^2(\Omega)}^2 .
\end{align}
Consequently,
\begin{align}
c_0\,\|\nabla \widetilde{\bu}\|_{\mathbb{L}^2(\Omega)}^2 + \calP^2 \|\widetilde{\bu}\|_{\mathbb{L}^2(\Omega)}^2
\;\le\; \|\bff\|_{\mathbb{L}^2(\Omega)}\,\|\widetilde{\bu}\|_{\mathbb{L}^2(\Omega)} .
\end{align}
Using the \(\bbL^2 \to \bbL^2\) resolvent bound just derived, we have
\begin{align}
    \|\widetilde{\bu}\|_{\mathbb{L}^2(\Omega)} \le \|\calN^{\calP}\|\,\|\bff\|_{\mathbb{L}^2(\Omega)} \le \calP^{-2}\|\bff\|_{\mathbb{L}^2(\Omega)}.
\end{align}
Then, we infer
\begin{align}
\|\nabla \widetilde{\bu}\|_{\mathbb{L}^2(\Omega)}^2
\;\le\; \frac{1}{c_0\,\calP^2}\,\|\bff\|_{\mathbb{L}^2(\Omega)}^2 .
\end{align}
Taking the supremum over \(\|\bff\|_{\mathbb{L}^2(\Omega)}=1\) gives
\begin{align}
\big\|\nabla \calN^{\calP}\big\|_{\mathcal{L}(\mathbb{L}^2(\Omega);\mathbb{L}^2(\Omega))}
\;\le\; \frac{1}{\sqrt{c_0} \calP}
= \mathcal{O}\!\Big(\frac{1}{\calP}\Big).
\end{align}
Consequently,
\begin{align}
\big\|\calN^{\calP}\big\|_{\mathcal{L}(\mathbb{L}^2(\Omega);\mathbb{H}^1(\Omega))}
&:= \Big(
\big\|\calN^{\calP}\big\|_{\mathcal{L}(\mathbb{L}^2(\Omega);\mathbb{L}^2(\Omega))}^2
+
\big\|\nabla \calN^{\calP}\big\|_{\mathcal{L}(\mathbb{L}^2(\Omega);\mathbb{L}^2(\Omega))}^2
\Big)^{1/2}
\;\le\; C\,\frac{1}{\calP}
= \mathcal{O}\!\Big(\frac{1}{\calP}\Big),
\end{align}
with \(C\) depending only on \(\lambda,\mu,\Omega\).
Finally, by continuity of the trace map \(\gamma:\mathbb{H}^1(\Omega)\to \mathbb{H}^{1/2}(\partial\Omega)\),
\begin{align}
\big\|\gamma \calN^{\calP}\big\|_{\mathcal{L}(\mathbb{L}^2(\Omega);\mathbb{H}^{1/2}(\partial\Omega))}
= \mathcal{O}\!\Big(\frac{1}{\calP}\Big).
\end{align}
These bounds provide the desired analogue of \eqref{eq:gamma NP norm} and complete the proof of Lemma \ref{lem:NP norm lem}.

\subsection{Proof of (\ref{eq:partial SL=f})} \label{subsec:partial SL=f}

We recall, from \eqref{eq:SLP def elastic}, the definition of the elastic single-layer operator with parameter $\calP$:  
\begin{align}\label{eq:A elastic SLP def}
\mathbf{SL}^{\mathcal{P}}(\mathbf{f})(\mathbf{x})
:= \int_{\partial\Omega} \Gamma_{\mathcal{P}}(\mathbf{x},\mathbf{y}) \, \mathbf{f}(\mathbf{y}) \, \mathrm{d}\sigma(\mathbf{y}),
\qquad \mathbf{x} \in \Omega,
\end{align}
where $\Gamma_{\mathcal{P}}(\cdot,\cdot)$ solves \eqref{eq:Gamma P def}, and $\mathbf{f}\in \mathbb{H}^{-1/2}(\partial\Omega)^{3}$ is arbitrary. Our goal is to prove that
\begin{align}\label{eq:A goal elastic}
\partial_{\nu} \mathbf{SL}^{\mathcal{P}}(\mathbf{f}) = \mathbf{f}
\qquad \text{on } \partial\Omega.
\end{align}

\noindent
Let $\Phi_{i\mathcal{P}}(\cdot,\cdot)$ be the full-space fundamental solution of the shifted elastic operator, that is,
\begin{align*}
\bigl(\mathcal{L}_{\lambda,\mu} - \mathcal{P}^{2}\bigr)\, \Phi_{i\mathcal{P}}(\mathbf{x},\mathbf{y})
= - \delta_{\mathbf{y}}(\mathbf{x})
\qquad \text{in } \mathbb{R}^{3}.
\end{align*}
For fixed $\mathbf{y}\in\Omega$, the two functions
\begin{align*}
\mathbf{x} \longmapsto \Gamma_{\mathcal{P}}(\mathbf{x},\mathbf{y})
\quad\text{and}\quad
\mathbf{x} \longmapsto \Phi_{i\mathcal{P}}(\mathbf{x},\mathbf{y})
\end{align*}
solve the same differential equation in $\Omega$, but $\Gamma_{\mathcal{P}}$ satisfies the homogeneous Neumann boundary condition in \eqref{eq:Gamma P def}, while $\Phi_{i\mathcal{P}}$ does not. Let $\mathcal{D}^{i\mathcal{P}}$ be the elastic double-layer operator built from $\Phi_{i\mathcal{P}}$. By multiplying the equation for $\Gamma_{\mathcal{P}}(\cdot,\mathbf{y})$ by $\Phi_{i\mathcal{P}}(\cdot,\mathbf{y})$ and integrating over $\Omega$, we obtain the interior relation
\begin{align}\label{eq:A interior relation}
\Gamma_{\mathcal{P}}(\cdot,\mathbf{y})
+ \mathcal{D}^{i\mathcal{P}}\bigl(\Gamma_{\mathcal{P}}(\cdot,\mathbf{y})\bigr)
= \Phi_{i\mathcal{P}}(\cdot,\mathbf{y})
\qquad \text{in } \Omega.
\end{align}

\noindent
Let $\mathcal{K}^{i\mathcal{P}}$ denote the Neumann–Poincar\'e boundary operator associated with $\Phi_{i\mathcal{P}}$. Taking the trace of \eqref{eq:A interior relation} on $\partial\Omega$ and using the jump properties of the elastic double-layer operator, we deduce
\begin{align}\label{eq:A Gamma boundary}
\Gamma_{\mathcal{P}}(\cdot,\mathbf{y})
= \Bigl(\tfrac{1}{2}\mathcal{I} + \mathcal{K}^{i\mathcal{P}}\Bigr)^{-1} \bigl(\Phi_{i\mathcal{P}}(\cdot,\mathbf{y})\bigr)
\qquad \text{on } \partial\Omega,
\end{align}
where $\mathcal{I}$ is the identity on $\mathbb{L}^{2}(\partial\Omega)^{3}$. Define the single-layer operator with kernel $\Phi_{i\mathcal{P}}$ by
\begin{align}\label{eq:A S iP}
\mathcal{S}^{i\mathcal{P}}(\mathbf{g})(\mathbf{x})
:= \int_{\partial\Omega} \Phi_{i\mathcal{P}}(\mathbf{x},\mathbf{y}) \, \mathbf{g}(\mathbf{y}) \, \mathrm{d}\sigma(\mathbf{y}),
\qquad \mathbf{x}\in\Omega.
\end{align}
Inserting \eqref{eq:A Gamma boundary} into \eqref{eq:A elastic SLP def}, we obtain for $\mathbf{x}\in\Omega$,
\begin{align}
\mathbf{SL}^{\mathcal{P}}(\mathbf{f})(\mathbf{x})
&= \int_{\partial\Omega} \Gamma_{\mathcal{P}}(\mathbf{x},\mathbf{y}) \, \mathbf{f}(\mathbf{y}) \, \mathrm{d}\sigma(\mathbf{y}) \notag\\
&= \int_{\partial\Omega} \Phi_{i\mathcal{P}}(\mathbf{x},\mathbf{y}) \,
\Bigl[\Bigl(\tfrac{1}{2}\mathcal{I} + \mathcal{K}^{i\mathcal{P}}\Bigr)^{-1}\Bigr]^{\!*}(\mathbf{f})(\mathbf{y})
\, \mathrm{d}\sigma(\mathbf{y}) \label{eq:A SL as S}\\
&= \mathcal{S}^{i\mathcal{P}}\Bigl( \Bigl(\tfrac{1}{2}\mathcal{I} + \mathcal{K}^{i\mathcal{P}}\Bigr)^{-1} \Bigr)^{\!*}(\mathbf{f})(\mathbf{x}). \notag
\end{align}
Since $\mathbf{f}$ is arbitrary, \eqref{eq:A SL as S} is equivalent to
\begin{align}\label{eq:A operator identity}
\mathbf{SL}^{\mathcal{P}}
= \mathcal{S}^{i\mathcal{P}} \circ \Bigl( \Bigl(\tfrac{1}{2}\mathcal{I} + \mathcal{K}^{i\mathcal{P}}\Bigr)^{-1} \Bigr)^{\!*}
\qquad \text{in } \Omega.
\end{align}

We take the normal derivative of both sides of \eqref{eq:A operator identity}. The single-layer potential $\mathcal{S}^{i\mathcal{P}}$ satisfies the jump formula
\begin{align}\label{eq:A jump SiP}
\partial_{\nu} \mathcal{S}^{i\mathcal{P}}(\mathbf{g})
= \Bigl(\tfrac{1}{2}\mathcal{I} + \mathcal{K}^{i\mathcal{P}}\Bigr)^*\mathbf{g}
\qquad \text{on } \partial\Omega
\end{align}
for every $\mathbf{g} \in \mathbb{H}^{-1/2}(\partial\Omega)^{3}$. We get
\begin{align}\label{eq:A product}
\partial_{\nu} \mathbf{SL}^{\mathcal{P}}(\mathbf{f})
&= \partial_{\nu} \, \mathcal{S}^{i\mathcal{P}}\Bigl(\Bigl(\tfrac{1}{2}\mathcal{I} + \mathcal{K}^{i\mathcal{P}}\Bigr)^{-1}\Bigr)^{\!*}(\mathbf{f}) \notag\\
&= \Bigl(\tfrac{1}{2}\mathcal{I} + \mathcal{K}^{i\mathcal{P}}\Bigr)^*
    \Bigl(\Bigl(\tfrac{1}{2}\mathcal{I} + \mathcal{K}^{i\mathcal{P}}\Bigr)^{-1}\Bigr)^{\!*}(\mathbf{f}) .
\end{align}
Using the symmetry of the kernel we have
\[
\Bigl(\tfrac{1}{2}\mathcal{I} + \mathcal{K}^{i\mathcal{P}}\Bigr)^{*}
= \tfrac{1}{2}\mathcal{I} + \mathcal{K}^{-i\mathcal{P}},
\]
so \eqref{eq:A product} becomes
\begin{align}
\partial_{\nu} \mathbf{SL}^{\mathcal{P}}(\mathbf{f})
&= \Bigl(\tfrac{1}{2}\mathcal{I} + \mathcal{K}^{-i\mathcal{P}}\Bigr)
   \Bigl(\tfrac{1}{2}\mathcal{I} + \mathcal{K}^{-i\mathcal{P}}\Bigr)^{-1}(\mathbf{f})
    \qquad \text{on } \partial\Omega \notag\\
&= \mathbf{f} \qquad \text{on } \partial\Omega,  \label{eq:A final}
\end{align}
which is precisely \eqref{eq:partial SL=f}.

\subsection{Proof of Lemma \ref{lem:The algebraic system is invertible} }\label{subsec:The invertibility of system}

We show that the algebraic system \eqref{eq:The algebraic system} is invertible. The argument proceeds by linking \eqref{eq:The algebraic system} to a continuous integral formulation and then invoking a variational framework to establish invertibility. From \eqref{eq:The algebraic system} we read
\begin{align}
\mathbf{Y}_m + \mathcal{P}^{2} \sum_{\substack{j=1\\ j\neq m}}^{M} \Gamma \left({\bf z}_m;{\bf z}_j\right)a^{1-h}\,\frac{1}{\beta_j}\,\mathbf{Y}_j
=\mathbf{S}  \left({\bf z}_m\right)+\frac{\omega^{2}\rho_{1}}{\boldsymbol{\alpha}}\,\mathrm{Error}_m.
\end{align}

\noindent
Here $\mathbf{Y}_m$ is as in \eqref{eq:Ym def} and $\mathrm{Error}_m$ is defined in \eqref{eq:Errorm def}. Using $\lvert \Omega_j\rvert=a^{1-h}$ for $1\le j\le M$ we rewrite
\begin{align}
\mathbf{Y}_m + \mathcal{P}^{2} \sum_{\substack{j=1\\ j\neq m}}^{M} \Gamma\left({\bf z}_m;{\bf z}_j\right)\lvert \Omega_j\rvert \,\frac{1}{\beta_j}\,\mathbf{Y}_j
= \mathbf{S} \left({\bf z}_m\right)+\frac{\omega^{2}\rho_{1}}{\boldsymbol{\alpha} } \mathrm{Error}_m,
\end{align}
equivalently,
\begin{align}
{\bf Y}_m + \mathcal{P}^{2}  \sum_{\substack{j=1\\ j\neq m}}^{M}  \int_{\Omega}\Gamma \left({\bf  z}_m;{\bf z}_j\right) \chi_{\Omega_j}\!\left(x\right)\,\frac{1}{\beta_j}\,{\bf Y}_j \mathrm{d} {\bx }= {\bf  S} \left({\bf z}_m\right)+\frac{\omega^{2}\rho_{1}}{ \boldsymbol{\alpha}} \mathrm{Error}_m.
\end{align}

\noindent
Multiplying both sides by $\chi_{\Omega_m}\!\left(\cdot\right)$ and summing in $m$ we obtain
\begin{align}
\sum_{m=1}^{M}\chi_{\Omega_m}\!\left(\cdot\right){\bf Y}_m
&+\mathcal{P}^{2}\sum_{m=1}^{M}\chi_{\Omega_m} \left(\cdot\right)  \sum_{\substack{j=1\\ j\neq m}}^{M} 
\int_{\Omega}\Gamma\!\left({\bf z}_m;{\bf z}_j\right)\,\chi_{\Omega_j}\!\left(\bx \right)\,\frac{1}{\beta_j}\,{\bf Y}_j\,\mathrm{d} \bx \notag \\
&= \sum_{m=1}^{M}\chi_{\Omega_m}\!\left(\cdot\right){\bf S}\left({\bf z}_m\right)
+\frac{\omega^{2}\rho_{1}}{ \boldsymbol{\alpha}} \sum_{m=1}^{M} \chi_{\Omega_m}  \left(\cdot\right) \mathrm{Error}_m.
\end{align}

\noindent
Introduce the shorthands
\begin{align}\label{eq:Y S R def}
    \mathbb{Y}  \left(\cdot\right):=\sum_{m=1}^{M}\chi_{\Omega_m}\!\left(\cdot\right){\bf Y}_m,
    \qquad
\mathbb{S}\!\left(\cdot\right):=\sum_{m=1}^{M}\chi_{\Omega_m}\!\left(\cdot\right){\bf S} \left({\bf z}_m\right),\qquad
\mathbb{R}\!\left(\cdot\right):=\sum_{m=1}^{M}\chi_{\Omega_m}\!\left(\cdot\right) \mathrm{Error}_m.  
\end{align}
so that
\begin{align}\label{eq:Y+p2=S+R}
\mathbb{Y}  \left(\cdot\right)
+\mathcal{P}^{2}\sum_{m=1}^{M}\chi_{\Omega_m}\!\left(\cdot\right)  \sum_{\substack{j=1\\ j\neq m}}^{M}  \int_{\Omega}\Gamma\!\left({\bf z}_m;{\bf z}_j\right)\,\chi_{\Omega_j}\!\left(\bx\right)\,\frac{1}{\beta_j}\,{\bf Y}_j\,\mathrm{d} \bx 
= \mathbb{S}\!\left(\cdot\right)+\frac{\omega^{2}\rho_{1}}{\boldsymbol{\alpha}} \mathbb{R} \left(\cdot\right).
\end{align}
The next auxiliary statement shows that the second term on the left-hand side converges in $\mathbb{L}^{2}\!\left(\Omega\right)$ to a function in the range of the Newtonian operator $\mathcal{N}$; see \eqref{eq:Newton potential in Omega} for its definition.

\begin{lem}\label{lem:T1 def}
    Define
\begin{align}\label{eq:T1 def in lemma}
\mathbb{T}_{1} \left(\cdot\right):=
\mathcal{N}  \left(\mathbb{Y}\right) \left(\cdot\right)
-\sum_{m=1}^{M}\chi_{\Omega_m}\!\left(\cdot\right)  \sum_{\substack{j=1\\ j\neq m}}^{M} 
\int_{\Omega}\Gamma \left({\bf z}_m;{\bf z}_j\right)\,\chi_{\Omega_j}\!\left(\bx\right)\,\frac{1}{\beta_j}\,{\bf Y}_j\,\mathrm{d} \bx   \quad\text{in } \ \Omega .
\end{align}
With $\mathcal{N}\!\left(\cdot\right)$ as in \eqref{eq:Newton potential in Omega}, one has the bound
\begin{align}
\lVert \mathbb{T}_{1}\rVert_{\mathbb{L}^{2}  \left(\Omega\right)} \lesssim\; a^{\left(1-h\right)/6}\,\lVert \mathbb{Y}\rVert_{\mathbb{L}^{2}\!\left(\Omega\right)} .
\end{align}
\end{lem}

\begin{proof}
     See Subsection \ref{subsec:T1 proof}.
\end{proof}

With Lemma \ref{lem:T1 def} at hand, \eqref{eq:Y+p2=S+R} becomes
\begin{align}\label{eq:I+p2N}
\left(I+\mathcal{P}^{2}\mathcal{N}\right)\!\left(\mathbb{Y}\right)\!\left(\cdot\right)=\mathbb{S}\!\left(\cdot\right)+ \boldsymbol{r}   \left(\cdot\right)\quad  \text{in }\Omega ,
\end{align}
where $\mathbb{S} \left(\cdot\right)$ is given by \eqref{eq:Y S R def} and
\begin{align}
\boldsymbol{r}   \left(\cdot\right):=\frac{\omega^{2}\rho_{1}}{\boldsymbol{\alpha}}\,\mathbb{R}  \left(\cdot\right)+\mathcal{P}^{2} \mathbb{T}_{1}\!\left(\cdot\right).
\end{align}

\noindent
In the sense of distributions, \eqref{eq:I+p2N} yields
\begin{align}
\left(\mathcal{L}_{\lambda,\mu}+\omega^{2}\rho-\mathcal{P}^{2}\right)\!\left(\mathbb{Y}\right)
=\left(\mathcal{L}_{\lambda,\mu}+\omega^{2}\rho\right)\!\left(\mathbb{S}+\boldsymbol{r}\right)=\boldsymbol{f}\quad  \text{in }\Omega .
\end{align}
By construction (see \eqref{eq:Y S R def}), $\mathbb{Y}\!\left(\cdot\right)=\mathbf{0}$ in a neighborhood of $\partial\Omega$. Extending $\mathbb{Y}$ and $\boldsymbol{f}$ by zero to \(\mathbb{R}^{3}\setminus\Omega\) gives
\begin{align}
\mathcal{L}_{\lambda,\mu} \left(\mathbb{Y}\right)=\left(-\omega^{2}\rho+\mathcal{P}^{2}\right)\!\left(\mathbb{Y}\right)+\boldsymbol{f}\quad\text{in }\mathbb{R}^{3},
\end{align}
with $\boldsymbol{f}\in H^{-2}_{\mathrm{comp}}\!\left(\mathbb{R}^{3}\right)$. Hence
\begin{align}\label{eq:Y+N(P2-omega rho)}
\mathbb{Y}+\mathcal{N}_{\mathbb{R}^{3}}\!\left(\left(\mathcal{P}^{2}-\omega^{2} \rho  \right)\left(\mathbb{Y}\right)\right)
=-\,\mathcal{N}_{\mathbb{R}^{3}}\!\left(\boldsymbol{f}\right)\quad\text{in }\mathbb{L}^2\!\left(\mathbb{R}^{3}\right),
\end{align}
where $\mathcal{N}_{\mathbb{R}^{3}}$ is the Newtonian operator associated with the whole space (cf.~\eqref{eq:Newton potential in Omega}). To address well-posedness of \eqref{eq:Y+N(P2-omega rho)}, multiply by $\left(\mathcal{P}^{2}-\omega^{2}\rho\right)>0$ (for $\mathcal{P}\gg 1$) to get
\begin{align*}
\left(\mathcal{P}^{2}-\omega^{2} \rho\right)\mathbb{Y}
+\left(\mathcal{P}^{2}-\omega^{2} \rho \right)\mathcal{N}_{\mathbb{R}^{3}}\!\left(\left(\mathcal{P}^{2}-\omega^{2} \rho \right)\!\left(\mathbb{Y}\right)\right)
=-\left(\mathcal{P}^{2}-\omega^{2} \rho \right)\mathcal{N}_{\mathbb{R}^{3}}\!\left(\boldsymbol{f}\right),
\end{align*}
in $\mathbb{L}^2\!\left(\mathbb{R}^{3}\right)$. Taking $\mathbb{L}^2\!\left(\mathbb{R}^{3}\right)$ inner products yields
\begin{align*}
\mathscr{J}_{1}\!\left(\mathbb{Y};\mathbb{Z}\right)=\mathscr{J}_{2}\!\left(\mathbb{Z}\right),
\end{align*}
with
\begin{align}
\mathscr{J}_{1}\!\left(\mathbb{Y};\mathbb{Z}\right)
&=\left\langle \sqrt{\mathcal{P}^{2}-\omega^{2}\rho}\,\mathbb{Y}\,;\,\sqrt{\mathcal{P}^{2}-\omega^{2}\rho}\,\mathbb{Z}\right\rangle_{\mathbb{L}^{2}\!\left(\mathbb{R}^{3}\right)} \notag \\
&\quad +\left\langle \mathcal{N}_{\mathbb{R}^{3}}\!\left(\left(\mathcal{P}^{2}-\omega^{2}\rho\right)\mathbb{Y}\right)\,;\,\left(\mathcal{P}^{2}-\omega^{2}\rho\right)\mathbb{Z}\right\rangle_{\mathbb{L}^{2}\!\left(\mathbb{R}^{3}\right)}, \notag
\end{align}
and
\begin{align*}
\mathscr{J}_{2}\!\left(\mathbb{Z}\right):=-\left\langle \mathcal{N}_{\mathbb{R}^{3}}\!\left(\boldsymbol{f}\right)\,;\,\left(\mathcal{P}^{2}-\omega^{2}\rho\right)\mathbb{Z}\right\rangle_{\mathbb{L}^{2}\!\left(\mathbb{R}^{3}\right)} .
\end{align*}
Then $\mathscr{J}_{1}$ is continuous and satisfies
\begin{align*}
\lvert \mathscr{J}_{1}\!\left(\mathbb{Y},\mathbb{Z}\right)\rvert
&\le \lVert \mathbb{Y}\rVert_{\mathbb{L}^{2}\!\left(\mathbb{R}^{3}\right)}\lVert \mathbb{Z}\rVert_{\mathbb{L}^{2}\!\left(\mathbb{R}^{3}\right)}
\lVert \mathcal{P}^{2}-\omega^{2}\rho\rVert_{L^{\infty}\!\left(\mathbb{R}^{3}\right)} \\
&\quad +\lVert \mathbb{Y}\rVert_{\mathbb{L}^{2}\!\left(\mathbb{R}^{3}\right)}\lVert \mathbb{Z}\rVert_{\mathbb{L}^{2}\!\left(\mathbb{R}^{3}\right)}
\lVert \mathcal{P}^{2}-\omega^{2}\rho\rVert_{L^{\infty}\!\left(\mathbb{R}^{3}\right)}\lVert \mathcal{N}_{\mathbb{R}^{3}}\rVert_{\mathcal{L}}.
\end{align*}
It is also coercive:
\begin{align}\label{eq:coercive}
\mathscr{J}_{1}\!\left(\mathbb{Y},\mathbb{Y}\right)
&=\left\lVert \sqrt{\mathcal{P}^{2}-\omega^{2}\rho}\,\mathbb{Y}\right\rVert_{\mathbb{L}^{2}\!\left(\mathbb{R}^{3}\right)}^{2}
+\left\langle \mathcal{N}_{\mathbb{R}^{3}}\!\left(\left(\mathcal{P}^{2}-\omega^{2}\rho\right)\mathbb{Y}\right)\,;\,\left(\mathcal{P}^{2}-\omega^{2}\rho\right)\mathbb{Y}\right\rangle_{\mathbb{L}^{2}\!\left(\mathbb{R}^{3}\right)} \notag \\
&\ge \left\lVert \sqrt{\mathcal{P}^{2}-\omega^{2}\rho}\,\mathbb{Y}\right\rVert_{\mathbb{L}^{2}\!\left(\mathbb{R}^{3}\right)}^{2}
\ge \inf_{\mathbb{R}^{3}}\!\left(\mathcal{P}^{2}-\omega^{2}\rho\right)\,\lVert \mathbb{Y}\rVert_{\mathbb{L}^{2}\!\left(\mathbb{R}^{3}\right)}^{2},
\end{align}
by positivity of the Newtonian operator. Moreover,
\begin{align}\label{eq:bounded}
\lvert \mathscr{J}_{2}\!\left(\mathbb{Z}\right)\rvert
\le \lVert \mathcal{N}_{\mathbb{R}^{3}}\!\left(\boldsymbol{f}\right)\rVert_{\mathbb{L}^{2}\!\left(\mathbb{R}^{3}\right)}
\lVert \mathcal{P}^{2}-\omega^{2}\rho\rVert_{L^{\infty}\!\left(\mathbb{R}^{3}\right)}
\lVert \mathbb{Z}\rVert_{\mathbb{L}^{2}\!\left(\mathbb{R}^{3}\right)}.
\end{align}
By the Lax–Milgram lemma, \eqref{eq:Y+N(P2-omega rho)} admits a unique solution. Combining \eqref{eq:Y+N(P2-omega rho)}, \eqref{eq:coercive}, and \eqref{eq:bounded} yields
\begin{align}\label{eq:inf P2-omega rho}
\inf_{\mathbb{R}^{3}}\!\left(\mathcal{P}^{2}-\omega^{2}\rho\right)\,\lVert \mathbb{Y}\rVert_{\mathbb{L}^{2}\!\left(\mathbb{R}^{3}\right)}
\le \lVert \mathcal{N}_{\mathbb{R}^{3}}\!\left(\boldsymbol{f}\right)\rVert_{\mathbb{L}^{2}\!\left(\mathbb{R}^{3}\right)}
\lVert \mathcal{P}^{2}-\omega^{2}\rho\rVert_{\mathbb{L}^{\infty}\!\left(\mathbb{R}^{3}\right)}.
\end{align}
For $\mathcal{P}^{2}\gg 1$,
\[
\inf_{\mathbb{R}^{3}}\!\left(\mathcal{P}^{2}-\omega^{2}\rho\right)\sim \mathcal{P}^{2},\qquad
\lVert \mathcal{P}^{2}-\omega^{2}\rho\rVert_{\mathbb{L}^{\infty}\!\left(\mathbb{R}^{3}\right)}\sim \mathcal{P}^{2},
\]
and so from \eqref{eq:inf P2-omega rho},
\[
\lVert \mathbb{Y}\rVert_{\mathbb{L}^{2}\!\left(\mathbb{R}^{3}\right)}
\le \lVert \mathcal{N}_{\mathbb{R}^{3}}\!\left(\boldsymbol{f}\right)\rVert_{\mathbb{L}^{2}\!\left(\mathbb{R}^{3}\right)}
\lesssim \lVert \boldsymbol{f}\rVert_{\mathbb{H}^{-2}\!\left(\mathbb{R}^{3}\right)} .
\]
Since $\mathbb{Y}=0$ on $\mathbb{R}^{3}\setminus\overline{\Omega}$ and $\boldsymbol{f}=0$ there as well, we infer
\begin{align}
\lVert \mathbb{Y}\rVert_{\mathbb{L}^{2}\!\left(\Omega\right)}
\lesssim \lVert \boldsymbol{f}\rVert_{\mathbb{H}^{-2}\!\left(\Omega\right)}
:= \big\lVert\left(\mathcal{L}_{\lambda,\mu}+\omega^{2}\rho\right)\!\left(\mathbb{S}+\boldsymbol{r}\right)\big\rVert_{\mathbb{H}^{-2}\!\left(\Omega\right)} .
\end{align}
Retaining the dominant part on the right-hand side gives
\begin{align}
\lVert \mathbb{Y}\rVert_{\mathbb{L}^{2}\!\left(\Omega\right)}
\lesssim \lVert \mathcal{L}_{\lambda,\mu}\rVert_{\mathcal{L}\!\left(\mathbb{L}^{2}\!\left(\Omega\right);\mathbb{H}^{-2}\!\left(\Omega\right)\right)}\,
\lVert \mathbb{S}+\boldsymbol{r}\rVert_{\mathbb{L}^{2}\!\left(\Omega\right)} .
\end{align}
Consequently,
\begin{align}
\lVert \mathbb{Y}\rVert_{\mathbb{L}^{2}\!\left(\Omega\right)}
&\lesssim \lVert \mathbb{S}\rVert_{\mathbb{L}^{2}\!\left(\Omega\right)}+\lVert \boldsymbol{r}\rVert_{\mathbb{L}^{2}\!\left(\Omega\right)} \notag \\
&\lesssim \lVert \mathbb{S}\rVert_{\mathbb{L}^{2}\!\left(\Omega\right)}+\frac{\rho_{1}}{\lvert\boldsymbol{\alpha}\rvert}\,\lVert \mathbb{R}\rVert_{\mathbb{L}^{2}\!\left(\Omega\right)}+\mathcal{P}^{2}\lVert \mathbb{T}_{1}\rVert_{\mathbb{L}^{2}\!\left(\Omega\right)}  \notag \\
&\lesssim \lVert \mathbb{S}\rVert_{\mathbb{L}^{2}\!\left(\Omega\right)}+\frac{\rho_{1}}{\lvert \boldsymbol{\alpha}\rvert}\,\lVert \mathbb{R}\rVert_{\mathbb{L}^{2}\!\left(\Omega\right)}
+\mathcal{P}^{2}a^{\left(1-h\right)/6}\lVert \mathbb{Y}\rVert_{\mathbb{L}^{2}\!\left(\Omega\right)}. \notag
\end{align}
Due to \eqref{eq:Lambda D-Lambda P in thm}, we have $\mathcal{P}^{2}a^{\left(1-h\right)/6}\ll 1$ when $a\ll 1$. This simplifies to
\begin{align}
\lVert \mathbb{Y}\rVert_{\mathbb{L}^{2}\!\left(\Omega\right)}
\lesssim \lVert \mathbb{S}\rVert_{\mathbb{L}^{2}\!\left(\Omega\right)}+\frac{\rho_{1}}{\lvert \boldsymbol{\alpha}\rvert}\,\lVert \mathbb{R}\rVert_{\mathbb{L}^{2}\!\left(\Omega\right)} .
\end{align}

\noindent
Since $\rho_{1}\sim a^{-2}$ (see \eqref{eq:rho1 def}) and $\boldsymbol{\alpha}\sim a^{1-h}$ (see \eqref{eq:alpha=-P2a1-h in lemma}),
\begin{align}
\lVert \mathbb{Y}\rVert_{\mathbb{L}^{2}\!\left(\Omega\right)}
\lesssim \lVert \mathbb{S}\rVert_{\mathbb{L}^{2}\!\left(\Omega\right)}+a^{h-3}\lVert \mathbb{R}\rVert_{\mathbb{L}^{2}\!\left(\Omega\right)}
\overset{\eqref{eq:R L2 norm=a 7-h/2}}{\lesssim}  \lVert \mathbb{S}\rVert_{\mathbb{L}^{2}\!\left(\Omega\right)}+a^{\left(1+h\right)/2}\lVert \mathbf{g}\rVert_{\mathbb{H}^{-1/2}\!\left(\partial\Omega\right)} ,
\end{align}
whence, using \eqref{eq:S-single-layer} and \eqref{eq:Y S R def},
\begin{align} \label{eq:Y L2norm lesssim Sg}
\lVert \mathbb{Y}\rVert_{\mathbb{L}^{2}\!\left(\Omega\right)}
\lesssim a^{\left(1-h\right)/2}
\left(\sum_{m=1}^{M}\lvert \mathbf{S}\!\left(\mathbf{g}\right)\!\left({\bf z}_m\right)\rvert^{2}\right)^{1/2}
+a^{\left(1+h\right)/2}\lVert \mathbf{g}\rVert_{\mathbb{H}^{-1/2}\!\left(\partial\Omega\right)} .
\end{align}

\noindent
To bound $\sum_{m=1}^{M}\lvert \mathbf{S}\!\left(\mathbf{g}\right)\!\left({\bf z}_m\right)\rvert^{2}$ we note
\begin{align}
\sum_{m=1}^{M}\lvert \mathbf{S}\!\left(\mathbf{g}\right)\!\left({\bf z}_m\right)\rvert^{2}
&\overset{\eqref{eq:Gamma H1/2 H1}}{\le} \lVert \mathbf{g}\rVert_{\mathbb{H}^{-1/2}\!\left(\partial\Omega\right)}^{2}
\sum_{m=1}^{M}\lVert \Gamma\!\left({\bf z}_m,\cdot\right)\rVert_{\mathbb{H}^{1/2}\!\left(\partial\Omega\right)}^{2} \notag \\
&\le  \lVert \mathbf{g}\rVert_{\mathbb{H}^{-1/2}\!\left(\partial\Omega\right)}^{2}
\sum_{m=1}^{M}\lVert \Gamma\!\left({\bf z}_m,\cdot\right)\rVert_{\mathbb{H}^{1}\!\left(\Omega^{\circ}\right)}^{2} \notag  \\
&\lesssim  \lVert \mathbf{g}\rVert_{\mathbb{H}^{-1/2}\!\left(\partial\Omega\right)}^{2}
\sum_{m=1}^{M}\frac{1}{\operatorname{dist}^{4}\!\left(D_m;\partial\Omega\right)} \notag \\
&\overset{\eqref{eq:dist Dj and partial Omega}}{\lesssim}  \lVert \mathbf{g}\rVert_{\mathbb{H}^{-1/2}\!\left(\partial\Omega\right)}^{2}\,d^{-4} \notag  \\
&\overset{\eqref{eq:dmin}}{=}  \mathcal{O}\!\left(\lVert \mathbf{g}\rVert_{\mathbb{H}^{-1/2}\!\left(\partial\Omega\right)}^{2}\,a^{-4\left(1-h\right)/3}\right). \notag 
\end{align}
Plugging into \eqref{eq:Y L2norm lesssim Sg} gives
\begin{align}
\lVert \mathbb{Y}\rVert_{\mathbb{L}^{2}\!\left(\Omega\right)}\lesssim a^{\left(h-1\right)/6}\,\lVert \mathbf{g}\rVert_{\mathbb{H}^{-1/2}\!\left(\partial\Omega\right)} .
\end{align}

\noindent
Moreover, by construction,
\begin{align}\label{eq:Y L2 norm 2}
\lVert \mathbb{Y}\rVert_{\mathbb{L}^{2}\!\left(\Omega\right)}^{2}
=\sum_{m=1}^{M}\lvert \mathbf{Y}_m\rvert^{2}\lvert \Omega_m\rvert
=\lvert \Omega_{m_{0}}\rvert \sum_{m=1}^{M}\lvert \mathbf{Y}_m\rvert^{2},
\end{align}
and hence
\begin{align}\label{eq:sum Ym lesssim g H-1/2}
\left(\sum_{m=1}^{M}\lvert \mathbf{Y}_m\rvert^{2}\right)^{1/2}
\lesssim a^{2\left(h-1\right)/3}\,\lVert \mathbf{g}\rVert_{\mathbb{H}^{-1/2}\!\left(\partial\Omega\right)} .
\end{align}
This establishes injectivity of \eqref{eq:I+p2N}. Since a linear map between finite-dimensional spaces of equal dimension is bijective if and only if it is injective, \eqref{eq:I+p2N} is an isomorphism. Therefore, the algebraic system \eqref{eq:discrete system of Ym} is invertible, which completes the proof of Lemma  \ref{lem:The algebraic system is invertible}.

\subsection{Estimate of $\lVert\mathbb{R} \rVert_{\mathbb{L}^2(\Omega)}$}

Recalling the decomposition (cf.~\eqref{eq:Y S R def}),
\begin{align}
\mathbb{R}(\cdot):=\sum_{m=1}^{M}\chi_{\Omega_m}(\cdot)\,\mathrm{Error}_m,
\end{align}
where the subdomains \(\{\Omega_m\}_{m=1}^M\) are pairwise disjoint and
\(|\Omega_m|=|\Omega_0|=a^{\,1-h}\). 
Consequently, combining this with \eqref{eq:Errorm def}, we have
\begin{align}\label{eq:R L2 norm=}
\|\mathbb{R}\|_{\mathbb{L}^2(\Omega)}^{2}
&=\sum_{m=1}^{M}\int_{\Omega_m}\!|\mathrm{Error}_m|^{2}\, \rmd \mathbf{x}
=\sum_{m=1}^{M}|\Omega_m|\,|\mathrm{Error}_m|^{2} \notag \\
&=|\Omega_0|\sum_{m=1}^{M}|\mathrm{Error}_m|^{2}
=a^{\,1-h}\sum_{m=1}^{M}|\mathrm{Error}_m|^{2}. 
\end{align}

\noindent
By the triangle inequality and taking absolute values term-by-term in \eqref{eq:Errorm def}, we obtain
\begin{align}\label{eq:Error absolute}
|\mathrm{Error}_m|
\;\lesssim\;&
\sum_{\substack{j=1\\ j\neq m}}^{M}
\Bigg|
\int_{D_m} \mathbf{W}_m(\mathbf{x})\!\int_0^{1}
\nabla_{\mathbf{x}}\Gamma\!\left(\mathbf{z}_m+t(\mathbf{x}-\mathbf{z}_m),\,\mathbf{z}_j\right)\!\cdot(\mathbf{x}-\mathbf{z}_m)\,\rmd t\, \rmd \mathbf{x}
\Bigg|\,
\Bigg|\int_{D_j} \mathbf{v}_j^{\mathbf{g}}(\mathbf{y})\,\rmd \mathbf{y}\Bigg| \notag \\
&+\sum_{\substack{j=1\\ j\neq m}}^{M}
\Bigg|
\int_{D_m} \mathbf{W}_m(\mathbf{x})\!\int_{D_j}\!\int_{0}^{1}
\nabla_{\mathbf{y}}\Gamma\!\left(\mathbf{x};\,\mathbf{z}_j+t(\mathbf{y}-\mathbf{z}_j)\right)\!\cdot(\mathbf{y}-\mathbf{z}_j)\,\rmd t\,
\mathbf{v}_j^{\mathbf{g}}(\mathbf{y})\, \rmd \mathbf{y}\, \rmd  \mathbf{x}
\Bigg| \notag \\
&+\frac{1}{\omega^{2}\rho_{1}}\Bigg|
\int_{D_m} \mathbf{W}_m(\mathbf{x})\!\int_{0}^{1}
\nabla \mathbf{S}_m\!\left(\mathbf{z}_m+t(\mathbf{x}-\mathbf{z}_m)\right)\!\cdot(\mathbf{x}-\mathbf{z}_m)\,\rmd  t\, \rmd  \mathbf{x} \Bigg|\notag \\
&+\Bigg|
\int_{D_m} \mathbf{W}_m(\mathbf{x})\!\int_{D_m}\!\int_{0}^{1}
\nabla_{\mathbf{y}}\mathcal{R}\!\left(\mathbf{x},\,\mathbf{z}_m+t(\mathbf{y}-\mathbf{z}_m)\right)\!\cdot(\mathbf{y}-\mathbf{z}_m)\,\rmd t\,
\mathbf{v}_m^{\mathbf{g}}(\mathbf{y})\, \rmd  \mathbf{y}\, \rmd  \mathbf{x}
\Bigg| \notag \\
&+\frac{1}{\rho_{1}}\Bigg|
\int_{D_m} \mathbf{W}_m(\mathbf{x})\!\int_{D} \Gamma(\mathbf{x},\mathbf{y})\,\rho(\mathbf{y})\,\mathbf{v}^{\mathbf{g}}(\mathbf{y})\, \rmd  \mathbf{y}\,\rmd  \mathbf{x}
\Bigg|.
\end{align}

\noindent
Denote the third term on the right-hand side of \eqref{eq:Error absolute} by $\varrho_3$. By Cauchy--Schwarz and the change of variables $\mathbf{y}=\mathbf{z}_m+t(\mathbf{x}-\mathbf{z}_m)$, and using the fact that $\rho_1 \sim a^{-2}$ defined in \eqref{eq:rho1 def}, we obtain
\begin{align}\label{eq:varrho3 def}
\varrho_3
&\le  a^2 \|\mathbf{W}_m\|_{\mathbb{L}^2(D_m)}\,
\Bigg\| \int_{0}^{1}\nabla \mathbf{S}_m\!\left(\mathbf{z}_m+t(\cdot-\mathbf{z}_m)\right)\!\cdot(\cdot-\mathbf{z}_m)\,\rmd  t
\Bigg\|_{\mathbb{L}^2(D_m)} \notag \\
&\le a^2 \|\mathbf{W}_m\|_{\mathbb{L}^2(D_m)}\,
\Bigg[ \int_{0}^{1}\frac{1}{t}\!\int_{B(\mathbf{z}_m,ta)}
|\nabla \mathbf{S}_m(\mathbf{y})|^{2}\,|\mathbf{y}-\mathbf{z}_m|^{2}\,\rmd  \mathbf{y}\,\rmd  t
\Bigg]^{1/2} \notag \\
&\le a^{3} \|\mathbf{W}_m\|_{\mathbb{L}^2(D_m)}\,
\Bigg[ \int_{0}^{1}\int_{D_m}|\nabla \mathbf{S}_m(\mathbf{y})|^{2}\, \rmd  \mathbf{y}\,\rmd  t \Bigg]^{1/2}\, \notag \\
&\lesssim a^{3}\,\|\mathbf{W}_m\|_{\mathbb{L}^2(D_m)}\,\|\nabla \mathbf{S}_m\|_{\mathbb{L}^2(D_m)}.
\end{align}

\noindent
Combining \eqref{eq:S-single-layer}, \eqref{eq:Gamma nabla Gamma}, and \eqref{eq:varrho3 def}, we deduce from \eqref{eq:Error absolute} that
\begin{align}
|\mathrm{Error}_m|
\lesssim\;&
\|\mathbf{W}_m\|_{\mathbb{L}^2(D_m)}
\Bigg[
a^{4}\sum_{\substack{j=1\\ j\neq m}}^{M}\frac{1}{|\mathbf{z}_m-\mathbf{z}_j|^{2}}\,
\|\mathbf{v}_j^{\mathbf{g}}\|_{\mathbb{L}^2(D_j)}
\;+\;
a^{3}\,\|\nabla \mathbf{S}(\mathbf{g})\|_{\mathbb{L}^2(D_m)}
\Bigg]\\
& +\|\mathbf{W}_m\|_{\mathbb{L}^2(D_m)}
\Bigg[
a\Bigg(\int_{D_m}\!\int_{D_m}|\nabla\mathcal{R}(\mathbf{x},\mathbf{y})|^{2}\,\rmd\mathbf{y}\,\rmd\mathbf{x}\Bigg)^{\!1/2}
+a^{4}
\Bigg]\,
\|\mathbf{v}_m^{\mathbf{g}}\|_{\mathbb{L}^2(D_m)}\\
& + a^{\,\frac{9}{2}-h}
\sum_{\substack{j=1\\ j\neq m}}^{M}
\frac{1}{|\mathbf{z}_m-\mathbf{z}_j|}\,\|\mathbf{v}_j^{\mathbf{g}}\|_{\mathbb{L}^2(D_j)}.
\end{align}

\noindent
Using Cauchy--Schwarz for the pairwise sums and \eqref{eq:sum zm-zj}, we further deduce
\begin{align}\label{eq:Errorm lesssim}
|\mathrm{Error}_m| & \lesssim 
\|\mathbf{W}_m\|_{\mathbb{L}^2(D_m)}
\Big[a^{\frac{10+2h}{3}}\|\mathbf{v}^{\mathbf{g}}\|_{\mathbb{L}^2(D)}+a^{3}\|\nabla \mathbf{S}(\mathbf{g})\|_{\mathbb{L}^2(D_m)}\Big] \notag \\
&\quad +\|\mathbf{W}_m\|_{\mathbb{L}^2(D_m)}
\Bigg[
a\Bigg(\int_{D_m}\!\int_{D_m}|\nabla\mathcal{R}(\mathbf{x},\mathbf{y})|^{2}\,\rmd\mathbf{y}\,\rmd\mathbf{x}\Bigg)^{\!1/2}
+a^{4}
\Bigg]\,
\|\mathbf{v}_m^{\mathbf{g}}\|_{\mathbb{L}^2(D_m)} \notag \\
&\quad +a^{\frac{8-h}{2}}\|\mathbf{v}^{\mathbf{g}}\|_{\mathbb{L}^2(D)}.
\end{align}

\noindent
By standard mapping properties of layer potentials, \eqref{eq:dist D partial Omega}, and \eqref{eq:Gamma H1/2 H1},
\begin{align}\label{eq:nabla S}
\|\nabla \mathbf{S}(\mathbf{g})\|_{\mathbb{L}^2(D_m)}
 & \lesssim  \|\mathbf{g}\|_{\mathbb{H}^{-1/2}(\partial\Omega)}\Bigg[\int_{D_m} \big\|\nabla \Gamma(\mathbf{x},\cdot)\big\|_{\mathbb{H}^{1/2}(\partial\Omega)}^{2}\, \rmd  \mathbf{x} \Bigg]^{1/2} \notag \\
  & \lesssim  \|\mathbf{g}\|_{\mathbb{H}^{-1/2}(\partial\Omega)}\Bigg[\int_{D_m} \big\|\nabla \Gamma(\mathbf{x},\cdot)\big\|_{\mathbb{H}^{1}(\Omega^{\diamond})}^{2}\, \rmd  \mathbf{x} \Bigg]^{1/2} \notag \\
   &  \lesssim  \|\mathbf{g}\|_{\mathbb{H}^{-1/2}(\partial\Omega)}\Bigg[\int_{D_m} \frac{1}{\operatorname{dist}^6(\bx,\partial \Omega)} \rmd  \mathbf{x} \Bigg]^{1/2} \notag \\
  &\lesssim  \|\mathbf{g}\|_{\mathbb{H}^{-1/2}(\partial\Omega)}\,
 \frac{a^{3/2}}{\operatorname{dist}^{3}(D_m;\partial\Omega)}.   
\end{align}

\noindent
Arguing similarly as for \eqref{eq:Rxy lesssim dist}, we obtain
\begin{align*}
    |\nabla_{\mathbf{y}}\mathcal{R}(\mathbf{x},\mathbf{y})|
\lesssim \operatorname{dist}(\mathbf{x},\partial\Omega)^{-2/3}\,
\operatorname{dist}(\mathbf{y},\partial\Omega)^{-4/3},  \qquad \text{for } \mathbf{x}\neq\mathbf{y}.
\end{align*}
Hence
\begin{align}\label{eq:nabla Rxy}
\int_{D_m}\!\int_{D_m}|\nabla\mathcal{R}(\mathbf{x},\mathbf{y})|^{2}\, \rmd  \mathbf{y}\, \rmd  \mathbf{x}
\;\lesssim\;
\frac{|D_m|^{2}}{\operatorname{dist}^{4}(D_m,\partial\Omega)}.
\end{align}

\noindent
Substituting \eqref{eq:nabla S} and \eqref{eq:nabla Rxy} into \eqref{eq:Errorm lesssim}, we obtain
\begin{align}
|\mathrm{Error}_m|
\lesssim &\|\mathbf{W}_m\|_{\mathbb{L}^2(D_m)}
\Bigg[ a^{\frac{10+2h}{3}}\|\mathbf{v}^{\mathbf{g}}\|_{\mathbb{L}^2(D)}
+\|\mathbf{g}\|_{\mathbb{H}^{-1/2}(\partial\Omega)}\,
\frac{a^{9/2}}{\operatorname{dist}^{3}(D_m;\partial\Omega)}
\Bigg]  \notag \\
&\quad+\|\mathbf{W}_m\|_{\mathbb{L}^2(D_m)}
 \frac{a^{4}}{\operatorname{dist}^{2}(D_m;\partial\Omega)}\,\|\mathbf{v}_m^{\mathbf{g}}\|_{\mathbb{L}^2(D_m)}
+a^{\frac{8-h}{2}}\|\mathbf{v}^{\mathbf{g}}\|_{\mathbb{L}^2(D)} .
\end{align}

\noindent
Using \eqref{eq:Wm L2 norm}, we obtain
\begin{align}\label{eq:Errorm lessim vg L2 norm}
|\mathrm{Error}_m|
\lesssim &
a^{\frac{17-2h}{6}}\|\mathbf{v}^{\mathbf{g}}\|_{\mathbb{L}^2(D)}
+\|\mathbf{g}\|_{\mathbb{H}^{-1/2}(\partial\Omega)}\,
\frac{a^{4-h}}{\operatorname{dist}^{3}(D_m;\partial\Omega)} \notag \\
&\quad +  \frac{a^{\frac{7-2h}{2}}}{\operatorname{dist}^{2}(D_m;\partial\Omega)}\,
\|\mathbf{v}_m^{\mathbf{g}}\|_{\mathbb{L}^2(D_m)}
+ a^{\frac{8-h}{2}}\|\mathbf{v}^{\mathbf{g}}\|_{\mathbb{L}^2(D)}.
\end{align}

\noindent
Summing  \eqref{eq:Errorm lessim vg L2 norm} in \(m\) and using Cauchy--Schwarz together with \eqref{eq:vg L2 norm lesssim g} and \eqref{eq:dist Dj and partial Omega}, we infer
\begin{align}\label{eq:sum Errorm}
\sum_{m=1}^{M}|\mathrm{Error}_m|^{2}
\;\lesssim\;
a^{\frac{14+h}{3}}\|\mathbf{v}^{\mathbf{g}}\|_{\mathbb{L}^2(D)}^{2}
+a^{6}\|\mathbf{g}\|_{\mathbb{H}^{-1/2}(\partial\Omega)}^{2}
\;\lesssim\;
a^{6}\|\mathbf{g}\|_{\mathbb{H}^{-1/2}(\partial\Omega)}^{2}.
\end{align}

\noindent
Consequently, recalling \eqref{eq:R L2 norm=}, we conclude
\begin{align}\label{eq:R L2 norm=a 7-h/2}
\|\mathbb{R}\|_{\mathbb{L}^2(\Omega)}
=\mathcal{O}\!\Big(a^{\frac{7-h}{2}}\,\|\mathbf{g}\|_{\mathbb{H}^{-1/2}(\partial\Omega)}\Big).
\end{align}

\subsection{Proof of Lemma \ref{lem:vg estimate}}\label{subsec:vg estimate}

We first recall that, by \eqref{eq:vg LSE in D}, $\mathbf{v}^{\bf g}(\cdot)$ is the solution of  
\begin{equation}\label{eq:vg satisfies equation in Appendix}
\mathbf{v}^{\bf g}(\bx)
 -\omega^2 \int_{D} \Gamma(\bx,\by)\,\left(\rho_1-\rho(\by)\right)\,\mathbf{v}^{\bf g}(\by)\,\rmd\by
 = \mathbf{S}(\bx), \qquad \bx\in D.
\end{equation}
The proof unfolds in two steps. As a first step, we analyze the single-inclusion case.

\subsubsection{The case of one hard inclusion}

We work inside the single inclusion $D$ and start from the Lippmann--Schwinger relation in $D$ (cf.\ \eqref{eq:vg LSE in D}), combined with the splitting $\Gamma=\Gamma^0+\mathcal{R}$ from \eqref{eq:Gamma-decomposition}. This gives the vector identity
\begin{align}\label{eq:vg LSE in Appendix}
\mathbf{v}^{\bg}(\bx)
 &-\omega^{2}\rho_{1}\!\int_{D}\Gamma^{0}(\bx,\by)\,\mathbf{v}^{\bg}(\by)\,\rmd\by \notag\\
&=\mathbf{S}(\bx)
 + \omega^{2}\rho_{1}\!\int_{D}\mathcal{R}(\bx,\by)\,\mathbf{v}^{\bg}(\by)\,\rmd\by
 - \omega^{2}\!\int_{D}\Gamma(\bx,\by)\,\rho(\by)\,\mathbf{v}^{\bg}(\by)\,\rmd\by,
\qquad \bx\in D.
\end{align}

\noindent
Let $(\lambda_n^D,\mathbf{e}_n)_{n\in\mathbb{N}}$ be an $\mathbb{L}^2(D)^3$-orthonormal eigen-basis of $N_D$ defined in \eqref{eq:Newtonian def}, i.e.,
\begin{align}\label{eq:ND-eigensystem}
N_D[\mathbf{e}_n]=\lambda_n\,\mathbf{e}_n,
\qquad
\langle \mathbf{e}_n,\mathbf{e}_m\rangle_{\mathbb{L}^2(D)}=\delta_{nm}.
\end{align}

\noindent
Taking the inner product of \eqref{eq:vg LSE in Appendix} with $\mathbf{e}_n$, we arrive at
\begin{align}\label{eq:coeff-eq}
\langle \mathbf{v}^{\bg},\mathbf{e}_n\rangle_{\mathbb{L}^2(D)}
&-\omega^{2}\rho_{1}\,\langle N_D[\mathbf{v}^{\bg}],\mathbf{e}_n\rangle_{\mathbb{L}^2(D)}
=\langle \mathbf{S},\mathbf{e}_n\rangle_{\mathbb{L}^2(D)} \\
&\quad+ \omega^{2}\rho_{1}\Big\langle \int_{D}\mathcal{R}(\cdot,\by)\,\mathbf{v}^{\bg}(\by)\,\rmd\by,\mathbf{e}_n\Big\rangle_{\mathbb{L}^2(D)}
-\omega^{2}\Big\langle \int_{D}\Gamma(\cdot,\by)\,\rho(\by)\,\mathbf{v}^{\bg}(\by)\,\rmd\by,\mathbf{e}_n\Big\rangle_{\mathbb{L}^2(D)}. \notag
\end{align}

\noindent
By \eqref{eq:ND-eigensystem} and the self-adjointness of $N_D$, it follows that
\begin{align}\label{eq:use-eigen}
\langle N_D[\mathbf{v}^{\bg}],\mathbf{e}_n\rangle_{\mathbb{L}^2(D)}
=\langle \mathbf{v}^{\bg},N_D[\mathbf{e}_n]\rangle_{\mathbb{L}^2(D)}
=\lambda_n\,\langle \mathbf{v}^{\bg},\mathbf{e}_n\rangle_{\mathbb{L}^2(D)}.
\end{align}
Substituting \eqref{eq:use-eigen} into \eqref{eq:coeff-eq}, we obtain the coefficient identity
\begin{align}\label{eq:coeff-identity}
\left(1-\omega^{2}\rho_{1}\lambda_n\right)\,\langle \mathbf{v}^{\bg},\mathbf{e}_n\rangle_{\mathbb{L}^2(D)}
&=\langle \mathbf{S},\mathbf{e}_n\rangle_{\mathbb{L}^2(D)}
+\omega^{2}\rho_{1}\Big\langle \int_{D}\mathcal{R}(\cdot,\by)\,\mathbf{v}^{\bg}(\by)\,\rmd\by,\mathbf{e}_n\Big\rangle_{\mathbb{L}^2(D)}
\notag\\[-1mm]
&\quad-\omega^{2}\Big\langle \int_{D}\Gamma(\cdot,\by)\,\rho(\by)\,\mathbf{v}^{\bg}(\by)\,\rmd\by,\mathbf{e}_n\Big\rangle_{\mathbb{L}^2(D)}.
\end{align}

\noindent
Assuming a uniform spectral gap of the form \eqref{eq:1-omega2}, we find
\begin{align}\label{eq:coeff-sq}
\big|\langle \mathbf{v}^{\bg},\mathbf{e}_n\rangle_{\mathbb{L}^2(D)}\big|^{2}
&\lesssim a^{-2h}\left(
\big|\langle \mathbf{S},\mathbf{e}_n\rangle_{\mathbb{L}^2(D)}\big|^{2}
+ \omega^{4}|\rho_{1}|^{2}\,\Big|\Big\langle \textstyle\int_{D}\mathcal{R}(\cdot,\by)\,\mathbf{v}^{\bg}(\by)\,\rmd\by,\mathbf{e}_n\Big\rangle_{\mathbb{L}^2(D)}\Big|^{2} \right.
\notag\\
&\quad+ \left. \omega^{4}\,\Big|\Big\langle \textstyle\int_{D}\Gamma(\cdot,\by)\,\rho(\by)\,\mathbf{v}^{\bg}(\by)\,\rmd\by,\mathbf{e}_n\Big\rangle_{\mathbb{L}^2(D)}\Big|^{2} \right).
\end{align}

\noindent
Summing \eqref{eq:coeff-sq} over $n$ and using $\rho_{1}=\tilde{\rho}_{1}\,a^{-2}$ (cf.\ \eqref{eq:rho1 def}), together with Parseval’s identity, we obtain
\begin{align}\label{eq:master-ineq}
\|\mathbf{v}^{\bg}\|_{\mathbb{L}^2(D)}^{2}
&\lesssim a^{-2h}\left(
\|\mathbf{S}\|_{\mathbb{L}^2(D)}^{2}
+ a^{-4}\,\Big\|\int_{D}\mathcal{R}(\cdot,\by)\,\mathbf{v}^{\bg}(\by)\,\rmd\by\Big\|_{\mathbb{L}^2(D)}^{2} \right. \notag\\
&\quad + \left.\Big\|\int_{D}\Gamma(\cdot,\by)\,\rho(\by)\,\mathbf{v}^{\bg}(\by)\,\rmd\by\Big\|_{\mathbb{L}^2(D)}^{2}
\right).
\end{align}

\noindent
We now bound the last two terms in \eqref{eq:master-ineq}. First,
\begin{align}\label{eq:R-term-bound-1}
\Big\|\int_{D}\mathcal{R}(\cdot,\by)\,\mathbf{v}^{\bg}(\by)\,\rmd\by\Big\|_{\mathbb{L}^2(D)}^{2}
&\le \left(\int_{D}\!\int_{D}\big|\mathcal{R}(\bx,\by)\big|^{2}\,\rmd\by\,\rmd\bx\right)
\|\mathbf{v}^{\bg}\|_{\mathbb{L}^2(D)}^{2}.
\end{align}

\noindent
By \eqref{eq:Rxy lesssim dist} in Lemma~\ref{lem:R(x,y)} and the conditions $\dist(\bx,\partial\Omega)\ge\kappa(a)$ and $\dist(\by,\partial\Omega)\ge\kappa(a)$ (cf.\ \eqref{eq:dist D partial Omega}), we infer
\begin{align}\label{eq:R-kernel-L2}
\int_{D}\!\int_{D}\big|\mathcal{R}(\bx,\by)\big|^{2}\,\rmd\by\,\rmd\bx
&\lesssim \left(\frac{|D|}{\kappa(a)}\right)^{2}
= \mathcal{O}\!\left(a^{\frac{2(8+h)}{3}}\right).
\end{align}
Here we used $|D|\sim a^{3}$ for a single inclusion and $\kappa(a)\sim a^{\frac{1-h}{3}}$. Combining \eqref{eq:R-term-bound-1} with \eqref{eq:R-kernel-L2} yields
\begin{align}\label{eq:R-term-final}
\Big\|\int_{D}\mathcal{R}(\cdot,\by)\,\mathbf{v}^{\bg}(\by)\,\rmd\by\Big\|_{\mathbb{L}^2(D)}^{2}
&\lesssim a^{\frac{2(8+h)}{3}}\,\|\mathbf{v}^{\bg}\|_{\mathbb{L}^2(D)}^{2}.
\end{align}

\noindent
For the term involving $\Gamma(\cdot,\cdot)$, we use $\Gamma=\Gamma^{0}+\mathcal{R}$ and the operator-norm bound $\|N_D\|_{\mathcal{L}(\mathbb{L}^2(D)^3;\mathbb{L}^2(D)^3)}=\mathcal{O}(a^{2})$, together with \eqref{eq:R-kernel-L2}, to obtain
\begin{align}\label{eq:Gamma-term-split}
\Big\|\int_{D}\Gamma(\cdot,\by)\,\rho(\by)&\,\mathbf{v}^{\bg} (\by)\,\rmd\by\Big\|_{\mathbb{L}^2(D)}
\le \|N_D[\rho\,\mathbf{v}^{\bg}]\|_{\mathbb{L}^2(D)} 
+ \Big\|\int_{D}\mathcal{R}(\cdot,\by)\,\rho(\by)\,\mathbf{v}^{\bg}(\by)\,\rmd\by\Big\|_{\mathbb{L}^2(D)} \notag\\
&\lesssim a^{2}\|\rho\|_{L^\infty(D)}\,\|\mathbf{v}^{\bg}\|_{\mathbb{L}^2(D)}
+ \|\rho\|_{L^\infty(D)}\left(\int_{D}\!\int_{D}|\mathcal{R}(\bx,\by)|^{2}\,\rmd\by\,\rmd\bx\right)^{\!1/2}\|\mathbf{v}^{\bg}\|_{\mathbb{L}^2(D)}\notag\\
&\lesssim a^{4}\,\|\mathbf{v}^{\bg}\|_{\mathbb{L}^2(D)}^{2}.  
\end{align}

\noindent
Since $0<h<1$ and $a\ll 1$, inserting \eqref{eq:R-term-final} and \eqref{eq:Gamma-term-split} into \eqref{eq:master-ineq} gives
\begin{align}\label{eq:pre-absorb}
\|\mathbf{v}^{\bg}\|_{\mathbb{L}^2(D)}^{2}
&\lesssim a^{-2h}\,\|\mathbf{S}\|_{\mathbb{L}^2(D)}^{2}
+ a^{\frac{4(1-h)}{3}} \|\mathbf{v}^{\bg}\|_{\mathbb{L}^2(D)}^{2}
\lesssim a^{-2h}\,\|\mathbf{S}\|_{\mathbb{L}^2(D)}^{2}.
\end{align}

\noindent
It remains to bound $\|\mathbf{S}\|_{\mathbb{L}^2(D)}$ in terms of $\|\bg\|_{\mathbb{H}^{-1/2}(\partial\Omega)}$. By the single-layer representation associated with \eqref{eq:S(x) satisfies equation} in \eqref{eq:S-single-layer}, using the continuity $\mathcal{S}:\mathbb{H}^{-1/2}(\partial\Omega)\to\mathbb{H}^{1}(\Omega)$ of the single-layer operator, the embedding $\mathbb{H}^{1}(\Omega)\hookrightarrow \mathbb{L}^{6}(\Omega)$~\cite[Corollary 9.14]{B11}, and H\"older’s inequality on $D$, we infer
\begin{align}\label{eq:S-L6}
\|\mathbf{S}\|_{\mathbb{L}^{2}(D)}^{2}
&\le |D|^{\frac{2}{3}}\|\mathbf{S}\|_{\mathbb{L}^{6}(D)}^{2}
\lesssim
|D|^{\frac{2}{3}}\|\bg\|_{\mathbb{H}^{-1/2}(\partial\Omega)}^{2}
\left(\int_{D}\|\Gamma(\bx,\cdot)\|_{\mathbb{H}^{1/2}(\partial\Omega)}^{6}\,\rmd\bx\right)^{\frac{1}{3}}  .
\end{align}

{
{\noindent
Arguing as in \eqref{eq:Gamma H1/2 H1}–\eqref{eq:uf L2 norm finally} (and using $|\bx-\by|\ge \kappa(a)$ for $\bx\in D$ and $\by\in\Omega^\diamond=\Omega\setminus\overline{D}$ ), and noting that $|D|\sim a^{3}$ (because at this point, we consider the case of injecting a single hard inclusion.) while $\kappa(a)\sim a^{\frac{1-h}{3}}$ form \eqref{eq:dist D partial Omega}, we obtain
\begin{align}\label{eq:S-final-scale}
&\left(\int_{D}\|\Gamma(\bx,\cdot)\|_{\mathbb{H}^{1/2}(\partial\Omega)}^{6}\,\rmd\bx\right)^{\frac{1}{3}}  
\overset{\eqref{eq:Gamma H1/2 H1}}{\le}  \left( \int_D \|\Gamma(\mathbf{x},\cdot)\|_{\mathbb{H}^{1}(\Omega^\diamond)}^6\,\rmd\mathbf{x} \right)^{\frac{1}{3}}  
\overset{\eqref{eq:Gamma H1 norm}}{\lesssim}
\left( \int_D \int_{\Omega^\diamond} \frac{1}{|\mathbf{x}-\mathbf{y}|^{12}}\,\rmd\mathbf{y}\,\rmd\mathbf{x} \right)^{\frac{1}{3}}   \notag\\
&\qquad \qquad \qquad \qquad \qquad \stackrel{ \eqref{eq:footnote}}{\le}  \left( \int_D \int_{\Omega^\diamond} \frac{1}{\kappa(a)^{12}}\,\rmd\mathbf{y}\,\rmd\mathbf{x}\right)^{\frac{1}{3}}  
\overset{\eqref{eq:Gamma H1 norm}}{\lesssim}  \kappa(a)^{-4}\,|D|^{\frac{1}{3}} =\mathcal{O}\!\left(a^{\frac{4h-1}{3}}\right).
\end{align}}
}

\noindent
Combining \eqref{eq:S-L6}–\eqref{eq:S-final-scale} then yields
\begin{align}\label{eq:S-L2-scale}
\|\mathbf{S}\|_{\mathbb{L}^{2}(D)}^{2}
&\lesssim |D|^{\frac{2}{3}}\,a^{\frac{4h-1}{3}}\,\|\bg\|_{\mathbb{H}^{-1/2}(\partial\Omega)}^{2}
=\mathcal{O}\!\left(a^{\frac{4h+5}{3}}\|\bg\|_{\mathbb{H}^{-1/2}(\partial\Omega)}^{2}\right).
\end{align}

\noindent
Finally, inserting \eqref{eq:S-L2-scale} into \eqref{eq:pre-absorb} gives
\begin{align}\label{eq:vg-final-sqrt in single nanopartical}
\|\mathbf{v}^{\bg}\|_{\mathbb{L}^2(D)} \lesssim a^{\frac{5-2h}{6}}\|\bg\|_{\mathbb{H}^{-1/2}(\partial\Omega)}.
\end{align}

\subsubsection{The case of multiple hard inclusions}

Starting from \eqref{eq:vg satisfies equation in Appendix} and evaluating at any $\mathbf{x}\in D_m$, we arrive at
\begin{align}
\left(\mathcal{I}-\omega^{2}\rho_{1}\,N_{D_m}\right)&\big[\mathbf{v}_m^{\mathbf{g}}\big](\mathbf{x})
= \mathbf{S}_m(\mathbf{x})
+ \omega^{2}\rho_{1}\sum_{\substack{j=1\\ j\neq m}}^{M}\int_{D_j}\Gamma(\mathbf{x},\mathbf{y})\,\mathbf{v}_j^{\mathbf{g}}(\mathbf{y})\,\rmd\mathbf{y}
+ \omega^{2}\rho_{1}\int_{D_m}\mathcal{R}(\mathbf{x},\mathbf{y})\,\mathbf{v}_m^{\mathbf{g}}(\mathbf{y})\,\rmd\mathbf{y} \notag\\
&\quad -\omega^{2}\int_{D_m}\Gamma(\mathbf{x},\mathbf{y})\,\rho(\mathbf{y})\,\mathbf{v}_m^{\mathbf{g}}(\mathbf{y})\,\rmd\mathbf{y}
- \omega^{2}\sum_{\substack{j=1\\ j\neq m}}^{M}\int_{D_j}\Gamma(\mathbf{x},\mathbf{y})\,\rho(\mathbf{y})\,\mathbf{v}_j^{\mathbf{g}}(\mathbf{y})\,\rmd\mathbf{y}. 
\end{align}

\noindent
Here $N_{D_m}(\cdot)$ denotes the Newtonian operator in \eqref{eq:Newtonian def}. Expanding $\Gamma(\mathbf{x},\cdot)$ and $\Gamma(\mathbf{x},\cdot)\,\rho(\cdot)$ by Taylor’s formula and then applying the inverse of $\left(\mathcal{I}-\omega^{2}\rho_{1}\,N_{D_m}\right)$, we obtain
\begin{align}
\mathbf{v}_m^{\mathbf{g}}
&=\left(\mathcal{I}-\omega^{2}\rho_{1}\,N_{D_m}\right)^{-1}\!
\Bigg[\mathbf{S}_m
+ \omega^{2}\rho_{1}\sum_{\substack{j=1\\ j\neq m}}^{M}\Gamma(\cdot,\mathbf{z}_j)\int_{D_j}\mathbf{v}_j^{\mathbf{g}}(\mathbf{y})\,\rmd\mathbf{y} \notag\\
&\qquad
+\omega^{2}\rho_{1}\sum_{\substack{j=1\\ j\neq m}}^{M}
\int_{D_j}\!\!\int_{0}^{1}\nabla_{\mathbf{y}}\Gamma\!\left(\cdot,\mathbf{z}_j+t(\mathbf{y}-\mathbf{z}_j)\right)\,(\mathbf{y}-\mathbf{z}_j)\,\rmd t\,\mathbf{v}_j^{\mathbf{g}}(\mathbf{y})\,\rmd\mathbf{y} \notag\\
&\qquad
+\omega^{2}\rho_{1}\int_{D_m}\mathcal{R}(\cdot,\mathbf{y})\,\mathbf{v}_m^{\mathbf{g}}(\mathbf{y})\,\rmd\mathbf{y}
- \omega^{2}\int_{D_m}\Gamma(\cdot,\mathbf{y})\,\rho(\mathbf{y})\,\mathbf{v}_m^{\mathbf{g}}(\mathbf{y})\,\rmd\mathbf{y} \notag\\
&\qquad
- \omega^{2}\sum_{\substack{j=1\\ j\neq m}}^{M}\Gamma(\cdot,\mathbf{z}_j)\,\rho(\mathbf{z}_j)\int_{D_j}\mathbf{v}_j^{\mathbf{g}}(\mathbf{y})\,\rmd\mathbf{y} \notag\\
&\qquad
- \omega^{2}\sum_{\substack{j=1\\ j\neq m}}^{M}
\int_{D_j}\!\!\int_{0}^{1}\nabla\!\left(\Gamma(\cdot,\cdot)\,\rho(\cdot)\right)\!\left(\mathbf{z}_j+t(\mathbf{y}-\mathbf{z}_j)\right)\,(\mathbf{y}-\mathbf{z}_j)\,\rmd t\,\mathbf{v}_j^{\mathbf{g}}(\mathbf{y})\,\rmd\mathbf{y}
\Bigg].
\end{align}

\noindent
Within $D_m$, and upon introducing the notation from \eqref{eq:Ym def}, the preceding identity can be recast as
\begin{align}
\mathbf{v}_m^{\mathbf{g}}
&=\left(\mathcal{I}-\omega^{2}\rho_{1}\,N_{D_m}\right)^{-1}\!
\Bigg[
\mathbf{S}_m + \boldsymbol{\alpha}\sum_{\substack{j=1\\ j\neq m}}^{M}\Gamma(\cdot,\mathbf{z}_j)\,\mathbf{Y}_j \notag\\
&\quad  +\omega^{2}\rho_{1}\sum_{\substack{j=1\\ j\neq m}}^{M}
\int_{D_j}\!\!\int_{0}^{1}\nabla_{\mathbf{y}}\Gamma\left(\cdot,\mathbf{z}_j+t(\mathbf{y}-\mathbf{z}_j)\right)\,(\mathbf{y}-\mathbf{z}_j)\,\rmd t\,\mathbf{v}_j^{\mathbf{g}}(\mathbf{y})\,\rmd\mathbf{y} \notag\\
&\quad +\omega^{2}\rho_{1}\int_{D_m}\mathcal{R}(\cdot,\mathbf{y})\,\mathbf{v}_m^{\mathbf{g}}(\mathbf{y})\,\rmd\mathbf{y}
- \omega^{2}\int_{D_m}\Gamma(\cdot,\mathbf{z}_m)\,\rho(\mathbf{y})\,\mathbf{v}_m^{\mathbf{g}}(\mathbf{y})\,\rmd\mathbf{y} 
- \frac{\boldsymbol{\alpha}}{\rho_{1}}\sum_{\substack{j=1\\ j\neq m}}^{M}\Gamma(\cdot,\mathbf{z}_j)\,\rho(\mathbf{z}_j)\,\mathbf{Y}_j \notag\\
&\quad - \omega^{2}\sum_{\substack{j=1\\ j\neq m}}^{M}
\int_{D_j}\!\!\int_{0}^{1}\nabla\!\left(\Gamma(\cdot,\cdot)\,\rho(\cdot)\right)\!\left(\mathbf{z}_j+t(\mathbf{y}-\mathbf{z}_j)\right)\,(\mathbf{y}-\mathbf{z}_j)\,\rmd t\,\mathbf{v}_j^{\mathbf{g}}(\mathbf{y})\,\rmd\mathbf{y}
\Bigg].  
\end{align}

\noindent
Taking the $\mathbb{L}^{2}(D_m)$ norm, using \eqref{eq:R-term-final} together with $\rho_{1}\sim a^{-2}$, and
\[
\left\|\left(\mathcal{I}-\omega^{2}\rho_{1}\,N_{D_m}\right)^{-1}\right\|_{\mathcal{L}(\mathbb{L}^{2}(D_m);\mathbb{L}^{2}(D_m))}\lesssim a^{-h}\quad\text{(cf.\ \eqref{eq:pre-absorb})},
\]
one obtains
\begin{align}
\|\mathbf{v}_m^{\mathbf{g}}\|_{\mathbb{L}^{2}(D_m)}
&\lesssim a^{-h}\,\|\mathbf{S}_m\|_{\mathbb{L}^{2}(D_m)}
+ a^{-h}\,|{\boldsymbol{\alpha}}|\sum_{\substack{j=1\\ j\neq m}}^{M}\|\Gamma(\cdot,\mathbf{z}_j)\|_{\mathbb{L}^{2}(D_m)}\,|\mathbf{Y}_j| \notag\\
&\quad + a^{-2-h}\sum_{\substack{j=1\\ j\neq m}}^{M}
\Big\|\!\int_{D_j}\!\!\int_{0}^{1}\nabla_{\mathbf{y}}\Gamma\left(\cdot,\mathbf{z}_j+t(\mathbf{y}-\mathbf{z}_j)\right)\,(\mathbf{y}-\mathbf{z}_j)\,\rmd t\,\mathbf{v}_j^{\mathbf{g}}(\mathbf{y})\,\rmd \mathbf{y}\Big\|_{\mathbb{L}^{2}(D_m)} \notag\\
&\quad + a^{\frac{2(1-h)}{3}}\,\|\mathbf{v}_m^{\mathbf{g}}\|_{\mathbb{L}^{2}(D_m)}
+ a^{-h}\Big\|\!\int_{D_m}\Gamma(\cdot,\mathbf{y})\,\rho(\mathbf{y})\,\mathbf{v}_m^{\mathbf{g}}(\mathbf{y})\,\rmd \mathbf{y}\Big\|_{\mathbb{L}^{2}(D_m)} \notag\\
&\quad + a^{2-h}\,|{\boldsymbol{\alpha}}|\sum_{\substack{j=1\\ j\neq m}}^{M}\|\Gamma(\cdot,\mathbf{z}_j)\|_{\mathbb{L}^{2}(D_m)}\,|\mathbf{Y}_j| \notag\\
&\quad + a^{-h}\sum_{\substack{j=1\\ j\neq m}}^{M}
\Big\|\!\int_{D_j}\!\!\int_{0}^{1}\nabla\!\left(\Gamma(\cdot,\cdot)\,\rho(\cdot)\right)\!\left(\mathbf{z}_j+t(\mathbf{y}-\mathbf{z}_j)\right)\,(\mathbf{y}-\mathbf{z}_j)\,\rmd t\,\mathbf{v}_j^{\mathbf{g}}(\mathbf{y})\,\rmd \mathbf{y}\Big\|_{\mathbb{L}^{2}(D_m)} .
\end{align}

\noindent   
Estimating the contributions involving $\Gamma(\cdot,\cdot)$ and using $h<1$ further yields
\begin{align}
\|\mathbf{v}_m^{\mathbf{g}}\|_{\mathbb{L}^{2}(D_m)}
&\overset{\eqref{eq:Gamma nabla Gamma}}{\lesssim} a^{-h}\,\|\mathbf{S}_m\|_{\mathbb{L}^{2}(D_m)}
+ a^{\frac{3}{2}-h}\,|{\boldsymbol{\alpha}}|\,
\left(\sum_{\substack{j=1\\ j\neq m}}^{M}\frac{1}{|\mathbf{z}_m-\mathbf{z}_j|^{2}}\right)^{\!1/2}
\left(\sum_{\substack{j=1\\ j\neq m}}^{M}\|\mathbf{Y}_j\|^{2}\right)^{\!1/2} \notag\\
&\quad
+ a^{2-h}\left(\sum_{\substack{j=1\\ j\neq m}}^{M}\frac{1}{|\mathbf{z}_m-\mathbf{z}_j|^{4}}\right)^{\!1/2}
\left(\sum_{\substack{j=1\\ j\neq m}}^{M}\|\mathbf{v}_j^{\mathbf{g}}\|_{\mathbb{L}^{2}(D_j)}^{2}\right)^{\!1/2}
+ a^{2-h}\,\|\mathbf{v}_m^{\mathbf{g}}\|_{\mathbb{L}^{2}(D_m)} \notag\\
&\quad
+ a^{\frac{7}{2}-h}\,|{\boldsymbol{\alpha}}|\,
\left(\sum_{\substack{j=1\\ j\neq m}}^{M}\frac{1}{|\mathbf{z}_m-\mathbf{z}_j|^{2}}\right)^{\!1/2}
\left(\sum_{\substack{j=1\\ j\neq m}}^{M}\|\mathbf{Y}_j\|^{2}\right)^{\!1/2} \notag\\
&\quad
+ a^{4-h}\left(\sum_{\substack{j=1\\ j\neq m}}^{M}\frac{1}{|\mathbf{z}_m-\mathbf{z}_j|^{4}}\right)^{\!1/2}
\left(\sum_{\substack{j=1\\ j\neq m}}^{M}\|\mathbf{v}_j^{\mathbf{g}}\|_{\mathbb{L}^{2}(D_j)}^{2}\right)^{\!1/2}.
\end{align}

\noindent
Invoking \eqref{eq:sum zm-zj}, we first simplify to
\begin{align*}
\|\mathbf{v}_m^{\mathbf{g}}\|_{\mathbb{L}^{2}(D_m)}
&\lesssim a^{-h}\,\|\mathbf{S}_m\|_{\mathbb{L}^{2}(D_m)}
+ a^{\frac32-h}\,|{\boldsymbol{\alpha}}|\,d^{-\frac{3}{2}}
\left(\sum_{j=1}^{M}\|\mathbf{Y}_j\|^{2}\right)^{\!1/2}
+ a^{2-h}\,d^{-2}\left(\sum_{\substack{j=1\\ j\neq m}}^{M}\|\mathbf{v}_j^{\mathbf{g}}\|_{\mathbb{L}^{2}(D_j)}^{2}\right)^{\!1/2}.
\end{align*}
Squaring both sides, using $d\sim a^{\frac{1-h}{3}}$ from \eqref{eq:dmin} and $\boldsymbol{\alpha}\sim a^{\,1-h}$ from \eqref{eq:Ym def}, we infer
\begin{align}
\|\mathbf{v}_m^{\mathbf{g}}\|_{\mathbb{L}^{2}(D_m)}^{2}
\lesssim a^{-2h}\,\|\mathbf{S}_m\|_{\mathbb{L}^{2}(D_m)}^{2}
+ a^{4-3h}\sum_{j=1}^{M}\|\mathbf{Y}_j\|^{2}
+ a^{\frac{2(4-h)}{3}}\sum_{j=1}^{M}\|\mathbf{v}_j^{\mathbf{g}}\|_{\mathbb{L}^{2}(D_j)}^{2}.
\end{align}

\noindent
Summing in $m$ then gives
\begin{align*}
\sum_{m=1}^{M}\|\mathbf{v}_m^{\mathbf{g}}\|_{\mathbb{L}^{2}(D_m)}^{2}
&\lesssim a^{-2h}\sum_{m=1}^{M}\|\mathbf{S}_m\|_{\mathbb{L}^{2}(D_m)}^{2}
+ a^{4-3h}\,M\sum_{j=1}^{M}\|\mathbf{Y}_j\|^{2}
+ a^{\frac{2(4-h)}{3}}\,M\sum_{j=1}^{M}\|\mathbf{v}_j^{\mathbf{g}}\|_{\mathbb{L}^{2}(D_j)}^{2}.
\end{align*}
Since $M\sim a^{\,h-1}$ (see \eqref{eq:M def}), we conclude
\begin{align}
\sum_{m=1}^{M}\|\mathbf{v}_m^{\mathbf{g}}\|_{\mathbb{L}^{2}(D_m)}^{2}
&\lesssim a^{-2h}\sum_{m=1}^{M}\|\mathbf{S}_m\|_{\mathbb{L}^{2}(D_m)}^{2}
+ a^{3-2h}\sum_{j=1}^{M}\|\mathbf{Y}_j\|^{2}
+ a^{\frac{5+h}{3}}\sum_{j=1}^{M}\|\mathbf{v}_j^{\mathbf{g}}\|_{\mathbb{L}^{2}(D_j)}^{2},
\end{align}
and, using $h<1$,
\begin{align}
\sum_{m=1}^{M}\|\mathbf{v}_m^{\mathbf{g}}\|_{\mathbb{L}^{2}(D_m)}^{2}
\lesssim a^{-2h}\sum_{m=1}^{M}\|\mathbf{S}_m\|_{\mathbb{L}^{2}(D_m)}^{2}
+ a^{3-2h}\sum_{j=1}^{M}\|\mathbf{Y}_j\|^{2}.
\end{align}

\noindent
Combining the preceding estimate with \eqref{eq:S-single-layer} and \eqref{eq:sum Ym lesssim g H-1/2}, we further have
\begin{align}\label{eq:vg L2 norm of single potential}
\|\mathbf{v}^{\mathbf{g}}\|_{\mathbb{L}^{2}(D)}^{2}
\lesssim a^{-2h}\,\|\mathbf{S}(\mathbf{g})\|_{\mathbb{L}^{2}(D)}^{2}
+ a^{\frac{5-2h}{3}}\,\|\mathbf{g}\|_{\mathbb{H}^{-1/2}(\partial\Omega)}^{2}.
\end{align}

\noindent
It remains to estimate $\|\mathbf{S}(\mathbf{g})\|_{\mathbb{L}^{2}(D)}^{2}$. Arguing as in the single hard inclusion case (see \eqref{eq:pre-absorb}–\eqref{eq:vg-final-sqrt in single nanopartical}) and applying H\"older’s inequality, we obtain
\begin{align}
\|\mathbf{S}(\mathbf{g})\|_{\mathbb{L}^{2}(D)}^{2}
&\le \sum_{m=1}^{M}\|\mathbf{S}(\mathbf{g})\|_{\mathbb{L}^{6}(D_m)}^{2}\,\|1\|_{\mathbb{L}^{3}(D_m)}^{2} \notag\\
&\lesssim a^{2}\sum_{m=1}^{M}\Bigg[\int_{D_m}\left(\int_{\partial\Omega}\big|\Gamma(\mathbf{x},\mathbf{y})\,\mathbf{g}(\mathbf{y})\big|\,\rmd \sigma(\mathbf{y})\right)^{\!6}\,\rmd \mathbf{x}\Bigg]^{\!1/3} \notag\\
&\le a^{2}\,\|\mathbf{g}\|_{\mathbb{H}^{-1/2}(\partial\Omega)}^{2}\sum_{m=1}^{M}\Bigg[\int_{D_m}\|\Gamma(\mathbf{x},\cdot)\|_{\mathbb{H}^{1/2}(\partial\Omega)}^{6}\,\rmd \mathbf{x}\Bigg]^{\!1/3} \notag\\
&\overset{\eqref{eq:S-L6}}{\underset{\eqref{eq:S-final-scale}}{\lesssim}} a^{2}\,\|\mathbf{g}\|_{\mathbb{H}^{-1/2}(\partial\Omega)}^{2}
\sum_{m=1}^{M}\Bigg(\int_{D_m}\int_{\Omega^\diamond}\frac{1}{|\mathbf{x}-\mathbf{y}|^{12}}\,\rmd \mathbf{y}\,\rmd \mathbf{x}\Bigg)^{\!1/3} \notag\\
&\lesssim a^{3}\,\|\mathbf{g}\|_{\mathbb{H}^{-1/2}(\partial\Omega)}^{2}\sum_{m=1}^{M}\frac{1}{\operatorname{dist}^{4}\!\left(D_m;\partial\Omega\right)}
\overset{\eqref{eq:dist Dj and partial Omega}}{\lesssim} a^{3}\,\|\mathbf{g}\|_{\mathbb{H}^{-1/2}(\partial\Omega)}^{2}\,d^{-4} \notag\\
&= a^{\frac{5+4h}{3}}\,\|\mathbf{g}\|_{\mathbb{H}^{-1/2}(\partial\Omega)}^{2}.
\end{align}

\noindent
Substituting the above bound into \eqref{eq:vg L2 norm of single potential} leads to the global estimate
\begin{align}
\|\mathbf{v}^{\mathbf{g}}\|_{\mathbb{L}^{2}(D)} \;\lesssim\; a^{\frac{5-2h}{6}}\,\|\mathbf{g}\|_{\mathbb{H}^{-1/2}(\partial\Omega)}.
\end{align}
This completes the proof of Lemma \ref{lem:vg estimate}.

\subsection{Proof of Lemma \ref{lem:alpha=-P2} }  \label{subsec:alpha=proof}
 
Fix $m$. We write
\begin{align}
\boldsymbol{\alpha} := \int_{D_{m}} \mathbf{W}_{m}(\bx)\,\rmd \bx
       = \int_{D_{m}} \left( \left(\omega^{2}\rho_{1}\right)^{-1} \mathcal{I} - N_{D_{m}} \right)^{-1} \left(\mathcal{I}\right)(\bx)  \rmd \bx.
\end{align}
Expanding $\mathcal{I}$ in the eigen basis $\{\mathbf{e}_n^{D_m}\}$ of the Newtonian operator $N_{D_{m}}(\cdot)$ yields
\begin{align}
\boldsymbol{\alpha}
&= \sum_{n}\langle \mathcal{I}; \mathbf{e}_{n}^{D_m}\rangle_{\mathbb{L}^{2}(D_{m})}
     \int_{D_{m}}  \left( \left(\omega^{2}\rho_{1}\right)^{-1} \mathcal{I} - N_{D_{m}} \right)^{-1}  (\mathbf{e}_{n}^{D_m})(\bx)  \rmd \bx \\
&= \sum_{n}\left(\langle \mathcal{I}; \mathbf{e}_{n}^{D_m} \rangle_{\mathbb{L}^{2}(D_{m})}\right)^{2}
     \frac{\omega^{2}\rho_{1}}{1-\omega^{2}\rho_{1}\lambda^{D_{m}}_{n}\,}. \notag
\end{align}
Choose $\omega$ so that the dispersion relation
\begin{align*}
1 - \omega^{2}\rho_{1}\lambda^{D_{m}}_{n_{0}} &= c_{n_{0}}\,a^{h}, \qquad c_{n_{0}}\in\mathbb{R},
\end{align*}
holds. Solving gives
\begin{align*}
\omega^{2} &= \frac{1-c_{n_{0}}\,a^{h}}{\rho_{1}\lambda^{D_{m}}_{n_{0}}}.
\end{align*}
Consequently,
\begin{align}\label{eq:1-omega2}
\left|1-\omega^{2}\rho_{1}\lambda^{D_{m}}_{n}\right|
&=
\begin{cases}
a^{h}, & n=n_{0},\\
\mathcal{O}(1),   & \text{otherwise},
\end{cases}  
\end{align}
and therefore
\begin{align}
\boldsymbol{\alpha} 
&= \left(\langle \mathcal{I}; \mathbf{e}_{n_{0}}\rangle_{\mathbb{L}^{2}(D_{m})}\right)^{2}
   \frac{\omega^{2}\rho_{1}}{1-\omega^{2}\rho_{1}\lambda^{D_{m}}_{n_{0}}}
 + \sum_{n\neq n_{0}}\left(\langle \mathcal{I}; \mathbf{e}_{n}\rangle_{\mathbb{L}^{2}(D_{m})}\right)^{2}
   \frac{\omega^{2}\rho_{1}}{1-\omega^{2}\rho_{1}\lambda^{D_{m}}_{n}} .
\end{align}

\noindent
Using \eqref{eq:1-omega2}, the remaining sum is estimated by
\begin{align}
\sum_{n\neq n_{0}}
   \frac{\big|\langle \mathcal{I}; \mathbf{e}_{n}\rangle_{\mathbb{L}^{2}(D_{m})}\big|^{2}|\omega^2\rho_1|}
        {\big|1-\omega^{2}\rho_{1}\lambda^{D_{m}}_{n}\big|}
&\lesssim a^{-2}\,\|1\|^{2}_{\mathbb{L}^{2}(D_{m})}
 = \mathcal{O}(a),
\end{align}
so that
\begin{align}
\boldsymbol{\alpha}
&= \left(\langle \mathcal{I}; \mathbf{e}_{n_{0}}\rangle_{\mathbb{L}^{2}(D_{m})}\right)^{2}
   \frac{\omega^{2}\rho_{1}}{1-\omega^{2}\rho_{1}\lambda^{D_{m}}_{n_{0}}}
 + \mathcal{O}(a).
\end{align}

\noindent
Using the scalings
\begin{align*}
\langle \mathcal{I}; \mathbf{e}_{n_{0}}\rangle_{\mathbb{L}^{2}(D_{m})}
= a^{3/2}\langle \mathcal{I}; \tilde{\mathbf{e}}_{n_{0}}\rangle_{\mathbb{L}^{2}(B)}, \qquad
\rho_{1} =\tilde{\rho}_1 a^{-2} , \qquad
\lambda^{D_{m}}_{n_{0}} = a^{2}\lambda^{B}_{n_{0}},
\end{align*}
we arrive at
\begin{align}\label{eq:alpha=-P2a1-h}
\boldsymbol{\alpha}= \left(\frac{1-c_{n_{0}}a^{h}}{\lambda^{B}_{n_{0}}\,c_{n_{0}}}\right)
   \left(\langle \mathcal{I}; \tilde{\mathbf{e}}_{n_{0}}\rangle_{\mathbb{L}^{2}(B)}\right)^{2} a^{1-h}
 + \mathcal{O}(a) 
= \frac{1}{\lambda^{B}_{n_{0}}\,c_{n_{0}}}
   \left(\langle \mathcal{I}; \tilde{\mathbf{e}}_{n_{0}}\rangle_{\mathbb{L}^{2}(B)}\right)^{2} a^{1-h}
 + \mathcal{O}(a).  
\end{align}

\noindent
Define the rescaled dominant part
\begin{align*}
\mathcal{P}^{2} &:= -\frac{\left(\langle \mathcal{I}; \tilde{\mathbf{e}}_{n_{0}}\rangle_{\mathbb{L}^{2}(B)}\right)^{2}}
              {\lambda^{B}_{n_{0}}\,c_{n_{0}}}.
\end{align*}
Then \eqref{eq:alpha=-P2a1-h} implies
\begin{align*}
\alpha &= -\mathcal{P}^{2}\,a^{1-h}+ \mathcal{O}(a).
\end{align*}

\noindent
Finally, by repeating the same argument,
\begin{align*}
\|\mathbf{W}_{m}\|_{\mathbb{L}^{2}(D_{m})}^{2}
&= \sum_{n}\frac{|\omega^{2}\rho_{1}|^{2}}
                    {|1-\omega^{2}\rho_{1}\lambda^{D_{m}}_{n}|^{2}}
            \,|\langle \mathcal{I}; \mathbf{e}_{n}\rangle_{\mathbb{L}^{2}(D_{m})}|^{2}
 \lesssim a^{-2(2+h)}\,\|1\|^{2}_{\mathbb{L}^{2}(D_{m})},
\end{align*}
and hence
\begin{align*}
\|\mathbf{W}_{m}\|_{\mathbb{L}^{2}(D_{m})}
&\lesssim a^{-(2+h)}\,\|1\|_{\mathbb{L}^{2}(D_{m})}
 = \mathcal{O} \left(a^{-\left(\frac{1}{2}+h\right)}\right).
\end{align*}

\subsection{Proof of Lemma \ref{lem:R(x,y)}}\label{subsec:lemR(x,y) proof}

\noindent
Multiply \eqref{eq:R(x,y) satisfiy} by \(\Gamma^0(\cdot,\cdot)\)—the solution of \eqref{eq:Gamma0 satisfy}—and integrate by parts over \(\Omega\). Using the boundary condition
\begin{align*}
\partial_{\nu_{\mathbf{x}}}\mathcal{R}(\mathbf{x},\mathbf{y}) = -\,\partial_{\nu_{\mathbf{x}}}\Gamma^0(\mathbf{x},\mathbf{y}) \quad \text{on } \partial\Omega,
\end{align*}
we arrive at the boundary–volume representation
\begin{align}
\mathcal{R}(\mathbf{x},\mathbf{y})
&= - \int_{\partial\Omega} \Gamma^0(\mathbf{x},\mathbf{z})\,\partial_{\nu_{\mathbf{z}}}\Gamma^0(\mathbf{z},\mathbf{y})\,\rmd S(\mathbf{z})
- \int_{\partial\Omega} \partial_{\nu_{\mathbf{z}}}\Gamma^0(\mathbf{x},\mathbf{z})\,\mathcal{R}(\mathbf{z},\mathbf{y})\,\rmd S(\mathbf{z}) \notag\\
&\quad + \omega^2 \int_{\Omega} \Gamma^0(\mathbf{x},\mathbf{z})\,\rho(\mathbf{z})\,\mathcal{R}(\mathbf{z},\mathbf{y})\,\rmd \mathbf{z}
+ \omega^2 \int_{\Omega} \Gamma^0(\mathbf{x},\mathbf{z})\,\rho(\mathbf{z})\,\Gamma^0(\mathbf{z},\mathbf{y})\,\rmd \mathbf{z}.
\end{align}
Equivalently,
\begin{align}\label{eq:Rxy def}
\mathcal{R}(\mathbf{x},\mathbf{y})
&= -\,SL_{\partial\Omega}\!\big[\partial_{\nu_{\mathbf{z}}}\Gamma^0(\mathbf{z},\mathbf{y})\big](\mathbf{x})
   -\, DL_{\partial\Omega}\!\big[\mathcal{R}(\cdot,\mathbf{y})\big](\mathbf{x}) \notag\\
  &\quad + \omega^2\, N_{\Omega}\!\big[\rho(\cdot)\,\mathcal{R}(\cdot,\mathbf{y})\big](\mathbf{x})
   + \omega^2\, N_{\Omega}\!\big[\rho(\cdot)\,\Gamma^0(\cdot,\mathbf{y})\big](\mathbf{x}),
\end{align}
where $SL_{\partial\Omega}$, $DL_{\partial\Omega}$ and $N_{\Omega}$ denote, respectively, the single-layer, double-layer and Newtonian potential operators on $\partial\Omega$ and in $\Omega$.

Since $\mathcal{R}(\mathbf{x},\mathbf{y})$ is a $3\times 3$ matrix-valued kernel, we measure its size with any fixed matrix norm $|\cdot|$ (e.g.\ Frobenius); all such norms are equivalent in finite dimension. By \cite[Lemma~5.2]{CGS24}, the Green tensor exhibits the leading-order point singularity
\begin{align*}
\big|\Gamma^0(\mathbf{x},\mathbf{y})\big|\simeq\frac{1}{|\mathbf{x}-\mathbf{y}|}, 
\qquad \mathbf{x}\neq \mathbf{y},
\end{align*}
and therefore $\big|\nabla_{\mathbf{y}}\Gamma^0(\mathbf{z},\mathbf{y})\big|\simeq |\mathbf{z}-\mathbf{y}|^{-2}$; the normal derivative $\partial_{\nu_{\mathbf{z}}}\Gamma^0(\mathbf{z},\mathbf{y})$ inherits the same scaling. Inserting these local behaviors into the representation \eqref{eq:Rxy def} and isolating the most singular boundary contribution, we obtain, up to terms that are strictly less singular in $(\mathbf{x},\mathbf{y})$,
\begin{align}\label{eq:Rxy-leading}
\mathcal{R}(\mathbf{x},\mathbf{y})
\simeq -\, SL_{\partial\Omega}\!\big[\partial_{\nu}\Gamma^0(\cdot,\mathbf{y})\big](\mathbf{x})
\qquad (\mathbf{x}\in\Omega).
\end{align}

\noindent
To estimate the right-hand side of \eqref{eq:Rxy-leading}, write
\begin{align*}
\big|\mathcal{R}(\mathbf{x},\mathbf{y})\big|
\lesssim
\int_{\partial\Omega} 
\frac{1}{|\mathbf{t}-\mathbf{x}|}\,\frac{1}{|\mathbf{t}-\mathbf{y}|^{2}}\,\rmd \sigma(\mathbf{t}),
\qquad \mathbf{x}\in\Omega,
\end{align*}
where we used $\big|\Gamma^0(\mathbf{x},\mathbf{t})\big|\lesssim |\mathbf{t}-\mathbf{x}|^{-1}$ and 
$\big|\partial_{\nu_{\mathbf{t}}}\Gamma^0(\mathbf{t},\mathbf{y})\big|\lesssim |\mathbf{t}-\mathbf{y}|^{-2}$. Applying H\"older’s inequality and the surface integral estimate (cf.\ \eqref{eq:dist D partial Omega}), we obtain
\begin{align*}
\big|\mathcal{R}(\mathbf{x},\mathbf{y})\big|
&\lesssim 
\left(\int_{\partial\Omega}\frac{1}{|\mathbf{t}-\mathbf{x}|^{3}}\,\rmd \sigma(\mathbf{t})\right)^{\!\frac13}
\left(\int_{\partial\Omega}\frac{1}{|\mathbf{t}-\mathbf{y}|^{3}}\,\rmd \sigma(\mathbf{t})\right)^{\!\frac23} \\
&\lesssim 
\operatorname{dist}(\mathbf{x},\partial\Omega)^{-\frac13}\,
\operatorname{dist}(\mathbf{y},\partial\Omega)^{-\frac23}
\simeq \kappa(a)^{-1}
= \mathcal{O}\!\left(a^{\frac{h-1}{3}}\right),
\end{align*}
valid in the boundary-layer regime where 
$\operatorname{dist}(\mathbf{x},\partial\Omega)\simeq \kappa(a)$ and 
$\operatorname{dist}(\mathbf{y},\partial\Omega)\simeq \kappa(a)$. This completes the singularity estimate for $\mathcal{R}$ and thus finishes the proof of Lemma~\ref{lem:R(x,y)}.

\subsection{Proof of Lemma \ref{lem:T1 def}}\label{subsec:T1 proof}

\noindent
We estimate the $\mathbb{L}^{2}(\Omega)$–norm of $\mathbb{T}_{1}(\cdot)$ (defined in \eqref{eq:T1 def in lemma}). Using the shorthand $\mathbb{Y}$ from \eqref{eq:Y S R def}, we start with
\begin{align}
\|\mathbb{T}_{1}\|_{\mathbb{L}^{2}(\Omega)}^{2}
:= \int_{\Omega}\!\left|
\int_{\Omega}\Gamma(\by,\bx)\,\mathbb{Y}(\bx)\,\mathrm{d}\bx
-\sum_{m=1}^{M}\chi_{\Omega_{m}}(\by)\!\sum_{\substack{j=1\\ j\neq m}}^{M}
\int_{\Omega}\Gamma(\bz_{m};\bz_{j})\,\chi_{\Omega_{j}}(\bx)\,\frac{1}{\beta_{j}}\,\mathbf{Y}_{j}\,\mathrm{d}\bx
\right|^{2}\mathrm{d}\by .
\end{align}

\noindent
Unlike Section \ref{subsec:4.3}, where $\Omega$ was partitioned into $\bigcup_{j=1}^M \Omega_j$ and $\bigcup_{k=1}^{\aleph} \Omega_k^{\star}$ to pinpoint the precise leading contribution associated with $\int_{\Omega_j} u^g(x)\,\rmd x$, we now only require bounds for quantities defined on the whole domain $\Omega$. Consequently, both families of subdomains are involved, but no extraction of a dominant term is needed. Since, for every $1 \le j \le M$ and $1 \le k \le \aleph$, the pieces have comparable size, $|\Omega_j| \sim a^{\,1-h} \sim |\Omega_k^{\star}|$, it is immaterial for the computations below whether we refer to $\{\Omega_j\}$ or $\{\Omega_k^{\star}\}$. For brevity, we therefore reuse the symbol $\Omega_j$ to also denote the starred domains. Thus
\begin{align}
\|\mathbb{T}_{1}\|_{\mathbb{L}^{2}(\Omega)}^{2}
= \sum_{m=1}^{M}\int_{\Omega_{m}}\!\left|
\int_{\Omega}\Gamma(\by,\bx)\,\mathbb{Y}(\bx)\,\mathrm{d}\bx
-\sum_{\substack{j=1\\ j\neq m}}^{M}
\int_{\Omega}\Gamma(\bz_{m};\bz_{j})\,\chi_{\Omega_{j}}(\bx)\,\frac{1}{\beta_{j}}\,\mathbf{Y}_{j}\,\mathrm{d}\bx
\right|^{2}\mathrm{d}\by .
\end{align}

\noindent
Invoking the definition of $\mathbb{Y}(\cdot)$ in \eqref{eq:Y S R def} and applying the triangle inequality yield
\begin{align}\label{eq:T1-split}
\|\mathbb{T}_{1}\|_{\mathbb{L}^{2}(\Omega)}^{2}
\lesssim\;&
\sum_{m=1}^{M}|\mathbf{Y}_{m}|^{2}|\Omega_{m}|
\iint_{\Omega_{m}\times\Omega_{m}}\! \left|\Gamma(\by,\bx)\right|^{2}\,\mathrm{d}\bx\,\mathrm{d}\by \\
&\;+\sum_{m=1}^{M}\sum_{\substack{j=1\\ j\neq m}}^{M}
|\mathbf{Y}_{j}|^{2} \sum_{\substack{j=1\\ j\neq m}}^{M}   |\Omega_{j}|
\iint_{\Omega_{m}\times\Omega_{j}}
\left|\Gamma(\by,\bx)-\Gamma(\bz_{m};\bz_{j})\frac{1}{\beta_{j}}\right|^{2} 
\mathrm{d}\bx\,\mathrm{d}\by . \notag
\end{align}

\noindent
Using \eqref{eq:betam=} we have
\begin{align}\label{eq:beta-bound}
\left|\Gamma(\by,\bx)-\Gamma(\bz_{m};\bz_{j})\frac{1}{\beta_{j} }\right|^{2}
\;\lesssim\;
\left|\Gamma(\by,\bx)-\Gamma(\bz_{m};\bz_{j})\right|^{2}
+a^{\frac{4}{3}(1-h)}\,\left|\Gamma(\bz_{m};\bz_{j})\right|^{2}.
\end{align}

\noindent
Substituting \eqref{eq:beta-bound} into \eqref{eq:T1-split} gives
\begin{align}
\|\mathbb{T}_{1}\|_{\mathbb{L}^{2}(\Omega)}^{2}
\lesssim\;&
\sum_{m=1}^{M}|\mathbf{Y}_{m}|^{2}
\left[
\max_{1\le m\le M}\!\left(
|\Omega_{m}|\iint_{\Omega_{m}\times\Omega_{m}}\! \left|\Gamma(\by,\bx)\right|^{2}\,\mathrm{d}\bx\,\mathrm{d}\by
\right)  \right. \notag \\
&\left.
+\sum_{m=1}^{M}\sum_{\substack{j=1\\ j\neq m}}^{M}
|\Omega_{j}|\iint_{\Omega_{m}\times\Omega_{j}}
\left|\Gamma(\by,\bx)-\Gamma(\bz_{m};\bz_{j})\right|^{2}\,\mathrm{d}\bx\,\mathrm{d}\by
\right] \notag \\
&\quad +a^{\frac{10}{3}(1-h)}
\sum_{j=1}^{M} |\mathbf{Y}_{j}|^{2}  \sum_{\substack{j=1\\ j\neq m}}^{M} \left|\Gamma(\bz_{m};\bz_{j})\right|^{2}.\notag
\end{align}

\noindent
We expand $\Gamma(\cdot;\cdot)$ around the cell centers via a first-order Taylor formula:
\begin{align}\label{eq:Taylor-Gamma}
\Gamma(\by,\bx)-\Gamma(\bz_{m};\bz_{j})
&=\int_{0}^{1}\!\nabla_{\bx}\Gamma\!\left(\bz_{m};\,\bz_{j}+t(\bx-\bz_{j})\right)\!\cdot(\bx-\bz_{j})\,\mathrm{d}t \\
&\quad+\int_{0}^{1}\!\nabla_{\by}\Gamma\!\left(\bz_{m}+t(\by-\bz_{m});\,\bx\right)\!\cdot(\by-\bz_{m})\,\mathrm{d}t .
\nonumber
\end{align}
Using \eqref{eq:Taylor-Gamma} together with the pointwise bounds from \eqref{eq:Gamma nabla Gamma}, we deduce
\begin{align}
\|\mathbb{T}_{1}\|_{\mathbb{L}^{2}(\Omega)}^{2}
\lesssim &
\sum_{m=1}^{M}|\mathbf{Y}_{m}|^{2}
\left[
\max_{1\le m\le M}
\left(|\Omega_{m}|\iint_{\Omega_{m}\times\Omega_{m}}\frac{1}{|\by-\bx|^{2}}\,
\mathrm{d}\bx\,\mathrm{d}\by\right) \right.  \notag \\
&\left.+
\sum_{m=1}^{M}\sum_{\substack{j=1\\ j\neq m}}^{M}
|\Omega_{j}|\iint_{\Omega_{m}\times\Omega_{j}}
\frac{|\bx-\bz_{j}|^{2}}{|\by-\bz_{j}|^{4}}\,\mathrm{d}\bx\,\mathrm{d}\by
\right]
+a^{\frac{10}{3}(1-h)}
\sum_{j=1}^{M}
|\mathbf{Y}_{j}|^{2}  \sum_{\substack{j=1\\ j\neq m}}^{M}  \frac{1}{|\bz_{m}-\bz_{j}|^{2}}. \notag
\end{align}

\noindent
Since $|\Omega_{m}|=a^{1-h}$ for all $m$, we infer
\[
\max_{1\le m\le M}
\left(|\Omega_{m}|\iint_{\Omega_{m}\times\Omega_{m}}
\frac{1}{|\by-\bx|^{2}}\,\mathrm{d}\bx\,\mathrm{d}\by\right)
\;\lesssim\; a^{\frac{7}{3}(1-h)}.
\]

\noindent
Consequently,
\begin{align}\label{eq:T1-mid}
\|\mathbb{T}_{1}\|_{\mathbb{L}^{2}(\Omega)}^{2}
\lesssim &
\sum_{m=1}^{M}|\mathbf{Y}_{m}|^{2}
\left[
a^{\frac{7}{3}(1-h)}
+\max_{1\le j\le M}\!\left(
|\Omega_{j}|\sum_{m=1}^{M}\sum_{\substack{j=1\\ j\neq m}}^{M}
\int_{\Omega_{m}}\!\frac{1}{|\by-\bz_{j}|^{4}}\,\mathrm{d}\by
\int_{\Omega_{j}}\!|\bx-\bz_{j}|^{2}\,\mathrm{d}\bx
\right)
\right]  \notag  \\
&\quad +a^{\frac{10}{3}(1-h)}
\sum_{j=1}^{M}  |\mathbf{Y}_{j}|^{2}    \sum_{\substack{j=1\\ j\neq m}}^{M}  \frac{1}{|\bz_{m}-\bz_{j}|^{2}}.
\end{align}

\noindent
Applying Taylor’s formula once more,
\begin{align}\label{eq:Omega-m-1over4}
\int_{\Omega_{m}}\!\frac{1}{|\by-\bz_{j}|^{4}}\,\mathrm{d}\by
= \frac{|\Omega_{m}|}{|\bz_{m}-\bz_{j}|^{4}}
+\int_{\Omega_{m}}\!\int_{0}^{1}
\nabla\!\left(|\cdot-\bz_{j}|^{-4}\right)\!\left(\bz_{m}+t(\by-\bz_{m})\right)\!\cdot(\by-\bz_{m})\,\mathrm{d}t\,\mathrm{d}\by .
\end{align}
Using \eqref{eq:Omega-m-1over4} in \eqref{eq:T1-mid} yields
\begin{align}\label{eq:T1-pre-final}
\|\mathbb{T}_{1}\|_{\mathbb{L}^{2}(\Omega)}^{2}
\lesssim &
\sum_{m=1}^{M}|\mathbf{Y}_{m}|^{2}
\left[
a^{\frac{7}{3}(1-h)}+
a^{1-h}\,
\max_{1\le j\le M}\!\left(\int_{\Omega_{j}}|\bx-\bz_{j}|^{2}\,\mathrm{d}\bx\right)
\left(\sum_{m=1}^{M}|\Omega_{m}|\sum_{\substack{j=1\\ j\neq m}}^{M}\frac{1}{|\bz_{m}-\bz_{j}|^{4}}\right)
\right]  \notag \\
&
\quad +a^{\frac{10}{3}(1-h)}
\sum_{j=1}^{M} |\mathbf{Y}_{j}|^{2}  \sum_{\substack{j=1\\ j\neq m}}^{M}  \frac{1}{|\bz_{m}-\bz_{j}|^{2}}.
\end{align}

\noindent
Using \eqref{eq:sum zm-zj}, $d\sim a^{\frac{1}{3}(1-h)}$, and $|\Omega_{m_{0}}|=a^{1-h}$, we conclude
\[
\max_{1\le j\le M}\int_{\Omega_{j}}|\bx-\bz_{j}|^{2}\,\mathrm{d}\bx
=\mathcal{O}\left(a^{\frac{5}{3}(1-h)}\right), \qquad  \sum_{\substack{j=1\\ j\neq m}}^{M}\frac{1}{|\bz_{m}-\bz_{j}|^{2}}
=\mathcal{O} \left(d^{-3}\right) =\mathcal{O} \left( a^{h-1} \right),
\]
and
\[
\sum_{m=1}^{M}|\Omega_{m}|\sum_{\substack{j=1\\ j\neq m}}^{M}\frac{1}{|\bz_{m}-\bz_{j}|^{4}}
\lesssim d^{-4}\sum_{m=1}^{M}|\Omega_{m}|=\mathcal{O} \left(d^{-4}\right)=\mathcal{O}\left( a^{\frac{4(h-1)}{3}}\right).
\]

\noindent
Thus \eqref{eq:T1-pre-final} simplifies to
\[
\|\mathbb{T}_{1}\|_{\mathbb{L}^{2}(\Omega)}^{2}
\lesssim
\sum_{m=1}^{M}|\mathbf{Y}_{m}|^{2}\,
\left[a^{\frac{7}{3}(1-h)}+a^{\frac{4}{3}(1-h)}\right]
= \mathcal{O} \left(a^{\frac{4}{3}(1-h)}\sum_{m=1}^{M}|\mathbf{Y}_{m}|^{2}\right)
\overset{\eqref{eq:Y L2 norm 2}}{=} \mathcal{O}  \left(a^{\frac{1}{3}(1-h)}\|\mathbb{Y}\|^{2}_{\mathbb{L}^{2}(\Omega)}\right).
\]

\noindent
Therefore,
\[
\|\mathbb{T}_{1}\|_{\mathbb{L}^{2}(\Omega)}
=\mathcal{O} \left(a^{\frac{1}{6}(1-h)}\|\mathbb{Y}\|_{\mathbb{L}^{2}(\Omega)}\right),
\]
which completes the proof of Lemma \ref{lem:T1 def}.

\bigskip

\noindent\textbf{Acknowledgment.}
The work of H. Diao is supported by the National Natural Science Foundation of China  (No. 12371422) and the Fundamental Research Funds for the Central Universities, JLU. The work of M. Sini is partially supported by the Austrian Science Fund (FWF): P36942.

	\end{document}